\documentclass[11pt, reqno]{amsart}

\usepackage{amssymb,amsmath,amsthm,latexsym,color}
\usepackage{pifont}
\usepackage{verbatim}
\usepackage{amsfonts}
\usepackage[title,titletoc,toc]{appendix}
\usepackage[colorlinks, linkcolor=red, citecolor=blue, urlcolor=blue, pagebackref, hypertexnames=false]{hyperref}

\title[life-span results]{New life-span results for the nonlinear heat equation}
\author[S. Tayachi and F. B. Weissler] {Slim Tayachi and Fred B. Weissler}
\address{ \vspace{0cm}\newline Slim Tayachi\newline Universit\'e de Tunis El Manar, Facult\'e des Sciences de Tunis, D\'epartement de
Math\'ematiques, Laboratoire  \'Equations aux D\'eriv\'ees
Partielles LR03ES04,  2092 Tunis,
Tunisia
\newline e-mail: {\tt slim.tayachi@fst.rnu.tn} \vspace{1cm}\newline
Fred B. Weissler\newline Universit\'e Sorbonne Paris Nord, CNRS UMR 7539 LAGA,
 99, Avenue Jean-Baptiste Cl\'ement 93430
Villetaneuse, France \newline e-mail: {\tt
weissler@math.univ-paris13.fr}}

\font\TenEns=msbm10 \font\SevenEns=msbm7 \font\FiveEns=msbm5
\newfam\Ensfam

\textfont\Ensfam=\TenEns \scriptfont\Ensfam=\SevenEns
\scriptscriptfont\Ensfam=\FiveEns

\def\R {\mathbb R}

\def\Rd {{\mathbb R}^N}

\def\Dd {{\mathcal D'}(\Omega)}
\def\Sd {{\mathcal S'}(\Rd)}

\def\F{{\mathcal F}}

\def\m{\mathfrak {m}}
\def\M{{\mathcal M}}

\newcommand\RR{\mathbb{R}}

\newtheorem{To}{Theorem}[section]
\newtheorem{prop}[To]{Proposition}
\newtheorem{cor}[To]{Corollary}

\newtheorem{example}[To]{Example}
\newtheorem{rem}{Remark}
\theoremstyle{definition}

\date{\today}

\begin{document}
\begin{abstract}
We obtain new estimates for the  existence time of the maximal solutions to the nonlinear heat equation $\partial_tu-\Delta u=|u|^\alpha u,\;\alpha>0$ with initial values in Lebesgue, weighted Lebesgue spaces  or measures.   Non-regular, sign-changing,  as well as non polynomial decaying  initial data are considered.   The proofs of the  lower-bound estimates of life-span are based on the local construction of solutions.  The proofs of the upper-bounds  exploit a well-known necessary condition for the existence of  nonnegative solutions. In addition, we establish  new results  for life-span using  dilation methods and we give  new life-span estimates for Hardy-H\'enon parabolic equations.
\end{abstract}

\maketitle

\medskip
\section{Introduction and statement of the results}
\setcounter{equation}{0}
In this paper, we consider the nonlinear heat equation
\begin{equation}
\label{NLH}
\partial_t u=\Delta u+|u|^\alpha u,
\end{equation}
where  $u=u(t,x)\in  \R,\;\; t > 0,\, x\in \Omega,$ a domain of $\R^N$ not necessarily bounded, $ N\geq 1$ and $\alpha>0.$
In the case where the boundary $\partial\Omega\not= \emptyset,$  we suppose $\partial\Omega$ sufficiently smooth and we impose Dirichlet conditions on the boundary:
$$u(t,x)=0,\; t > 0,\;  x\in \partial\Omega.$$
If $\Omega$ is not bounded, we impose Dirichlet conditions at infinity:
$$\lim_{|x| \to \infty,\, x \in \Omega}u(t,x)=0,\; t> 0,$$
or perhaps other convenient formulation (see for example \cite[Definition 15.1, p. 75]{QS3}).  We usually consider  the equation \eqref{NLH}   with the initial value
\begin{equation}
\label{NLHini}u(0,\cdot)= u_0.
\end{equation}

The Cauchy problem \eqref{NLH}-\eqref{NLHini}  is locally well-posed in
{various} Banach spaces.  In other words, each element or initial value $u_0$ in that space gives rise to a trajectory $u(t)=u(t,\cdot)$
which is a solution in some appropriate sense to the given equation, { here equation \eqref{NLH},} and such that $u(0) = u_0$.  In many cases, this trajectory cannot exist for all time $t$, and we denote by $T_{\max}(u_0)$ the maximal possible existence time of such a trajectory.
The term life-span refers to the study of the maximal existence time of solutions with initial { data} of the form $u_0 = \lambda\varphi$ for some fixed element $\varphi$ in the considered Banach space
and  all $\lambda > 0$. Our aim is to establish lower and upper bounds of the life-span  for a large class of initial { data} $\varphi$ in terms of $\lambda$ and  study the asymptotic behavior of  $T_{\max}(\lambda\varphi)$, either as $\lambda\to \infty$ or as $\lambda\to 0$.

It is well known that {  if $u_0\in C_b(\R^N),$ the Banach space of  continuous bounded functions on $\R^N,$ there exists $T_{\max}(u_0)> 0$ such that \eqref{NLH}-\eqref{NLHini} has a unique classical solution $u \in C^{1,2}((0, T_{\max}(u_0) )\times\R^N) \cap C([0, T_{\max}(u_0) )\times\R^N)$ which is bounded in $[0,T]\times \R^N$ for all $T<T_{\max}(u_0),$ and $\|u(t)\|_{L^\infty(\R^N)}\to \infty$ as $t\to T_{\max}(u_0),$ if $T_{\max}(u_0)<\infty.$}   It is proved in \cite{F} that if $\alpha<2/N$ and $\varphi\in C_b(\R^N)$ with $\varphi\geq 0,\; \varphi\not\equiv 0,$ then $T_{\max}(\lambda \varphi)<\infty$ for any $\lambda>0.$  { For all $\alpha>0,$} if $\varphi\in C_b(\R^N),$  $\varphi\geq 0$ and $\liminf_{|x|\to \infty} |x|^{\gamma}\varphi(x)>0$ with $\gamma<2/\alpha,$ then $T_{\max}(\lambda \varphi)<\infty,$ { for all $\lambda>0$}, as shown in \cite{LeeNI}. This last result has been improved in many papers, see  \cite{TW2} for instance and some references therein. If we do not impose the positivity of the initial data, it has been proven  in \cite{L} that {  for a given $\varphi$  sufficiently regular (i.e. with  finite energy), $\varphi\not\equiv 0,$ }and $\lambda>0$ is sufficiently large then  $T_{\max}(\lambda \varphi)<\infty.$  If $\alpha<2/N$ and $\varphi$ not necessarily positive but $\varphi\in L^1(\R^N)\cap C_0(\R^N)$ and $\int_{\R^N}\varphi\not=0$ then it is proved, in \cite{D}, that $T_{\max}(\lambda \varphi)<\infty$ for $\lambda>0$ sufficiently small. Other  blow-up results  for $\lambda$ small are proved in \cite{G, TW1,TW2}. The above mentioned results show in particular  the interest of studying  the behavior of $T_{\max}(\lambda\varphi)$  for  any value of $\lambda$ and with or without any sign restriction on the initial data.

 For example, { it is  proved in \cite{LeeNI} } that given any nontrivial nonnegative initial data $\varphi\in C_b(\R^N)$  then   $T_{\max}(\lambda\varphi)\sim \lambda^{-\alpha},$ as $\lambda \to \infty$ and  $T_{\max}(\lambda\varphi)\sim \lambda^{-\alpha},$ as $\lambda\to 0$ provided that $\varphi_\infty=\liminf_{|x|\to \infty}\varphi(x)>0.$ {  Shortly thereafter the exact limits were given in \cite{GW},}   that is $\lim_{\lambda\to \infty}\lambda^{\alpha}T_{\max}(\lambda\varphi)={1\over \alpha}\|\varphi\|_\infty^{-\alpha}$ and $\lim_{\lambda\to 0}\lambda^{\alpha}T_{\max}(\lambda\varphi)={1\over \alpha}\varphi_\infty^{-\alpha}.$

 An other example is the study of the asymptotic behavior of the life-span $T_{\max}(\lambda \varphi)$ when  $\varphi\in C_b(\R^N)$ is nonnegative nontrivial  and having also a polynomial decay at infinity, that is \begin{equation}
\label{boundfmzero}
0 < \liminf_{|x|\to \infty}|x|^{\gamma}\varphi(x) \le \limsup_{|x|\to \infty}|x|^{\gamma}\varphi(x) < \infty,
\end{equation}
$0 < \gamma < N$. This   is studied in \cite{LeeNI}  for small $\lambda>0$.  It is shown in \cite[Theorem 3.15 (ii), p. 375]{LeeNI} and \cite[Theorem 3.21 (ii), p. 376]{LeeNI}  that if $\alpha < 2/N$  then
\begin{equation*}
0 < \liminf_{\lambda\to 0}\lambda^{[({1\over \alpha}-{\gamma\over 2})^{-1}]}T_{\max}(\lambda \varphi) \le \limsup_{\lambda\to 0}\lambda^{[({1\over \alpha}-{\gamma\over 2})^{-1}]}T_{\max}(\lambda \varphi) < \infty.
\end{equation*}
These results have been generalized recently in \cite{HisaIshige} replacing  $\varphi\in C_b(\R^N)$ by $\varphi\in L^\infty(\R^N).$ See \cite[Theorem 5.1, p. 128]{HisaIshige} and \cite[Theorem 5.2, p. 130]{HisaIshige}. We notice that refined asymptotic is given in \cite{MY1} for large $\lambda$ and for $\varphi$   not necessarily positive but still continuous and bounded. Other estimates are obtained for the life-span for regular and slowly decaying initial data in \cite{MY,D,TW2}.

A class of initial { data} $\varphi$ where $\varphi$ is { not necessarily  in  $C_b(\R^N)$ or in $L^\infty(\R^N)$ and satisfying either \eqref{boundfmzero} or } \begin{equation}
\label{boundfminfty}
0 < \liminf_{|x|\to 0}|x|^{\gamma}\varphi(x) \le \limsup_{|x|\to 0}|x|^{\gamma}\varphi(x) < \infty,
\end{equation}
has been considered in \cite{TW2}. In fact, the  asymptotic behavior of the life-span for initial data $\varphi\in L^1_{{\rm loc}}(\R^N\setminus\{0\}),$ $|\varphi(x)|\leq c|x|^{-\gamma},$ where $0<\gamma<N$ and $c>0$ a constant,  is studied in \cite{TW2}. It is shown there that for some initial data $\varphi$ singular at the origin or satisfying $\liminf_{|x|\to \infty}\varphi(x):=\varphi_\infty=0$ the situation  is quite different from previously known life-span results. In particular, unlike \cite{GW}, the  limits of $\lambda^{[({1\over \alpha}-{\gamma\over 2})^{-1}]}T_{\max}(\lambda \varphi)$ as $\lambda \to 0$ or as $\lambda\to \infty $ may not exist. Also,   if $\varphi$ satisfies \eqref{boundfminfty}, {  so that} $\varphi$ is singular at the origin, then $T_{\max}(\lambda\varphi)\sim \lambda^{-\left({1\over \alpha}-{\gamma\over 2}\right)^{-1}},$ as $\lambda \to \infty$ instead of $\lambda^{-\alpha}$ if $\varphi$ is regular. See \cite[Corollary 1.13 and Proposition 4.6]{TW2}. It is also proved that if $\varphi(x)=|x|^{-\gamma},\;|x|\leq 1,\; \varphi(x)=0,\; |x|> 1,\; 0<\gamma<N,\; 0<\alpha<2/\gamma,\; (N-2)\alpha<4$ then $\lim_{\lambda\to\infty}\lambda^{\left({1\over \alpha}-{\gamma\over 2}\right)^{-1}}T_{\max}(\lambda \varphi)=C>0.$ This last behavior shows  the impact of the singularity of the initial data on the behavior of the life-span for large $\lambda.$

 The goal of this paper is to   improve and extend the above mentioned results by considering a large class of initial data, including singular, sign changing,  not necessarily polynomially decaying initial data. To carry out this goal, we use three different methods. The first is based on the contraction mapping argument used to prove local existence.   We recently introduced and used this method in \cite{TW2}. Here, we apply it to the nonlinear heat equation and nonlinear Hardy-H\'enon parabolic equations. In a forthcoming paper,   it will be applied  to a variety of evolution equations in order to exhibit the generality of this method (\cite{Slim}).  We know of some cases where the idea behind this method was previously used in other papers (see for example \cite{II})  but to our knowledge, this method has never been presented as such or exploited in a systematic way.
The second method is based on a necessary condition for local existence of non-negative solutions. The third method is based on scaling properties of the equation. Details are given below later in the introduction.

{ We begin with the first method and we consider  the case} where $u_0$ belongs to a Lebesgue space, where we can use the contraction mapping argument  done  in \cite{W2, W3}.  It is well known that the problem \eqref{NLH}-\eqref{NLHini} is locally well-posed
in $L^q(\Omega)$ whenever  $q\geq 1,\; q>q_c$
 where
\begin{equation}
\label{qcritique}
q_c = \frac{N\alpha}{2}.
\end{equation}
See \cite{W2, W3, QS3} and references therein.
For any
$u_0 \in L^q(\Omega)$, we denote by   $T_{\max}(u_0)$ { the  existence time of the maximal (regular) solution}
to \eqref{NLH}-\eqref{NLHini} in $L^q(\Omega)$. {  Our first result on lower bound of the life-span is derived from \cite{W2, W3} using an argument from \cite{TW2}.}
\begin{To}[Initial data in Lebesgue spaces]
\label{ThNLH}
Let $N\geq 1,$ $\alpha>0$  and $q_c$ be given by \eqref{qcritique}. Let $\varphi\in L^q(\Omega)$ with $1 \le q \le \infty,$ $q>q_c$ or $\varphi\in C_0(\Omega).$ Let $u\in C\left([0,T_{\max}(\lambda \varphi)); L^q(\Omega)\right)$ be the maximal classical solution of \eqref{NLH}-\eqref{NLHini} with initial data $u_0=\lambda \varphi,$ $\lambda>0$ (we replace $[0,T_{\max}(\lambda \varphi))$ by $(0,T_{\max}(\lambda \varphi))$ if $q=\infty$).   Then there exists a constant $C= C(\alpha,q)>0$ such that
\begin{equation}
\label{lowerNLH}
T_{\rm{max}}(\lambda\varphi){ \geq}  \frac{C}{(\lambda\|\varphi\|_q)^{(\frac{1}{\alpha} - \frac{N}{2q})^{-1}}},
\end{equation}
for all $\lambda > 0$.
\end{To}
Hereafter, $\|\cdot\|_q$ denotes the norm in the Lebesgue space $L^q(\Omega).$

\begin{rem}
\label{remarksLq}
{\rm $\,$
\begin{itemize}
\item[1)] If $T_{\max}(\varphi,L^q)$ denotes { the  existence time of the maximal (regular) solution}
to \eqref{NLH}-\eqref{NLHini} for  $\varphi\in L^q(\Omega),$ it is known ({see for example \cite{W2} and Proposition \ref{C0nlh} below}) that if $\varphi\in L^p(\Omega)\cap L^q(\Omega)$ we have $T_{\max}(\varphi,L^q)=T_{\max}(\varphi,L^p).$ It follows that if $\varphi\in L^{{\underline{q}}}(\Omega)\cap L^{{\overline{q}}}(\Omega) ,\; 1\leq {\underline{q}}\leq {\overline{q}}\leq \infty,\; q_c< {\underline{q}},$ then the estimate \eqref{lowerNLH} reads
    $$T_{\rm{max}}(\lambda\varphi)\geq C \lambda^{-\left({1\over \alpha}-{N\over 2{\underline{q}}}\right)^{-1}},\; \mbox{ if }\; 0<\lambda<1,$$
    $$T_{\rm{max}}(\lambda\varphi)\geq C \lambda^{-\left({1\over \alpha}-{N\over 2{\overline{q}}}\right)^{-1}},\; \mbox{ if }\; \lambda>1.$$
\item[2)] Theorem \ref{ThNLH} includes many known results on lower bound of the life-span.  {  For example, it }is shown in \cite[Theorem 3.21 (ii), p. 376]{LeeNI}
 that if  $\varphi\in C_b(\R^N),$ $\varphi\geq 0$ and $\limsup_{|x|\to \infty} |x|^{\gamma}\varphi(x)<\infty,$ then for  $N<\gamma<2/\alpha,$ (hence $q_c<1$) we have
$T_{\rm{max}}(\lambda\varphi)\geq C \lambda^{-\left({1\over \alpha}-{N\over 2}\right)^{-1}},$ as  $\lambda \to 0.
$
{  This result is a special case of Theorem \ref{ThNLH}. Indeed, it follows that} $\varphi\in C_0(\RR^N)\cap L^q(\R^N)$ with  $q\geq 1>N/\gamma.$  {  Thus, the estimate  of \cite{LeeNI} follows by taking $q=1$ in \eqref{lowerNLH}. Other examples will be given throughout the paper.}
\item[3)] Theorem \ref{ThNLH} is  valid for the equation $\partial_t u=\Delta u+a(x)|u|^\alpha u,$ with $a\in L^\infty(\Omega).$ Under the additional assumption  $a>0,$  it is shown in \cite[Theorem 3 (i), p. 35]{P} that, for $\varphi \in C_b(\R^N),$
$T_{\rm{max}}(\lambda\varphi)\geq C \lambda^{-\alpha},\; \mbox{ as } \lambda \to \infty.$ For such initial data we may take $q=\infty$ and \eqref{lowerNLH} recover the last estimate.
\item[4)] For positive initial data and bounded domain, estimate \eqref{lowerNLH} is established in \cite[Theorem 3.1, p. 2526]{TF} where it is also assumed that $N\geq 3,\; q>\max(1,q_c).$ See also \cite{PS} for other estimates with $N=3.$ 
\item[5)] Let us consider the nonlinear heat equation with  diffusivity:
\begin{equation}
\label{diffusioncoef}
\partial_t u=\mu\Delta u+|u|^\alpha u,\; u(0)=\varphi\in L^q(\Omega),\; q\geq 1,\; q>q_c,
\end{equation}
on { $(0,T^\mu_{\max}(\varphi))\times\Omega,$ where $\;\alpha>0,\;  \mu>0$ and $T^\mu_{\max}(\varphi)$ denotes the existence time of the maximal solution of \eqref{diffusioncoef} with initial data $\varphi$.} We want to find a lower estimate of  $T^\mu_{\max}(\varphi)$ with respect to $\mu.$ Let
$$v(t,x)=\mu^{-1/\alpha}u(t/\mu,x).$$
Then $v$ satisfies the equation $$\partial_t v=\Delta v+|v|^\alpha v,\; v(0)=\mu^{-1/\alpha}\varphi,$$
on $(0,\mu T^\mu_{\max}(\varphi))\times\Omega=\left(0,T_{\max}(\mu^{-1/\alpha}\varphi)\right)\times\Omega,$ where  $T_{\max}(\mu^{-1/\alpha}\varphi)$ is the existence time of  the maximal solution $v.$ Using \eqref{lowerNLH} we get,
$$T^\mu_{\max}(\varphi)=\mu^{-1}T_{\max}(\mu^{-1/\alpha}\varphi)\geq C\mu^{-1}\left(\mu^{-1/\alpha}\|\varphi\|_q\right)^{-\left({1\over \alpha}-{N\over 2q}\right)^{-1}}.$$
That is,
$$T^\mu_{\max}(\varphi)\geq C\mu^{\left({2q\over N\alpha}-1\right)^{-1}}\|\varphi\|_q^{-\left({1\over \alpha}-{N\over 2q}\right)^{-1}}.$$
 For $q=\infty,$  the right-hand side term does not depend on $\mu$ and we have
$$T^\mu_{\max}(\varphi)\geq C\|\varphi\|_\infty^{-\alpha}.$$ See \cite{MY2,FI,FI2} for related estimates. Note that if $q<\infty$ then $\lim_{\mu\nearrow \infty}T^\mu_{\max}(\varphi)=\infty.$
\end{itemize}}
\end{rem}

Using the same method based on the contraction mapping argument as in \cite{TW2}, we also derive from \cite{W3}  the following lower estimate for the life-span   for the case of finite Borel measure. We denote by $\M$
the set of finite Borel measures on $\Omega$.
\begin{To}[Initial data measure]
\label{lowermeas}
Let $N\geq 1,$ $\alpha>0$  and $q_c$ be given by \eqref{qcritique}. If $\m$ is a finite Borel measure on $\Omega$, i.e. $\m  \in \M$, and if $q_c < 1$, i.e. $\alpha < \frac{2}{N}$, then { the existence time of the maximal solution} for  \eqref{NLH}-\eqref{NLHini} with initial data $u_0=\lambda \m$ satisfies
\begin{equation}
\label{bulambdam3}
T_{\max}(\lambda\m) { \geq}  \frac{C}{(\lambda\|\m\|_\M)^{(\frac{1}{\alpha} - \frac{N}{2})^{-1}}}
\end{equation}
for all $\lambda > 0$,
where $C = C(\alpha)>0$ is a constant.
\end{To}

 We now  estimate from below the life-span of solutions for the nonlinear heat equation \eqref{NLH} in $\R^N$ with decaying initial data $u_0=\lambda\varphi$, which may be singular, without sign restriction and for any $\lambda>0.$  For $\gamma> 0,$ $1\leq q\leq \infty,$ we consider the weighted Lebesgue space
$$L^q_\gamma(\R^N)=\{f : \R^N\to \R,\mbox{mesurable}, |\cdot|^\gamma f\in L^q(\R^N)\},$$ endowed with the  norm
 $$\|f\|_{L^q_\gamma(\R^N)}:=\||\cdot|^\gamma f\|_{q}.$$ {  In Theorem \ref{th3} below, we give a  well-posedness result in weighted Lebesgue spaces for the nonlinear heat equation.  As a consequence,  we obtain}  the following lower bound estimates of the life-span.
 \begin{To}[Initial data in weighted Lebesgue spaces]
\label{lowerLqgamma}
 Let $N\geq 1$ and  $\alpha>0$. If  $\varphi \in L^q_\gamma(\R^N) $, where $0< \gamma<N,\; \gamma<2/\alpha, \; q\in (1,\infty]$ and 
{  \begin{equation}
\label{qcondition}
{1\over q}+{\gamma\over N}<1,\; {N\alpha\over 2q}+{\alpha\gamma\over 2}<1,
 \end{equation}
  then the existence time of   the maximal  solution of \eqref{NLH}-\eqref{NLHini} in $L^q_\gamma(\R^N)$} with $u_0=\lambda\varphi$ satisfies
\begin{equation}
\label{bulambdagamma3b}
T_{\max}(\lambda\varphi)  { \geq} \frac{C}{\left(\lambda\|\varphi\|_{L^q_{{\gamma}}}\right)^{(\frac{1}{\alpha} - \frac{N}{2q}-{\gamma\over 2})^{-1}}},
\end{equation}
for all $\lambda>0,$ where $C = C(\alpha,q,\gamma,N)>0$ is a constant.
\end{To}
{ \begin{rem}{\rm Under the conditions $0< \gamma<N,\; \gamma<2/\alpha,\; 1<q\leq \infty,$ \eqref{qcondition} is equivalent to $$q>{N\over N-\gamma}\; \mbox{and}\; q>{N\alpha\over 2-\gamma\alpha}.$$ }
\end{rem}}
Combining the results of Theorems \ref{lowerLqgamma} and \ref{ThNLH}, we get the following estimates of the {    existence time of  the maximal } solution
to \eqref{NLH}-\eqref{NLHini}.
\begin{cor}
\label{lowerLqLqgamma}
  Let $N\geq 1$,  $\alpha>0,$  $\gamma>0,$  $\gamma\not=N$ and $q_c$ be given by \eqref{qcritique}. Assume that
  { $$\min\left[{N\alpha\over 2},{\gamma\alpha\over 2}\right]<1.$$}
   Let $\varphi \in L^q(\R^N)\cap L^q_\gamma(\R^N) $, where $q\in (1,\infty],\; q>q_c.$ If  $\gamma<N$ we assume further {  \eqref{qcondition}.  Then the existence time of   the maximal  solution of \eqref{NLH}-\eqref{NLHini} in $L^q(\R^N)$} with $u_0=\lambda\varphi$ satisfies
\begin{equation}
\label{bulambdagamma3}
T_{\max}(\lambda\varphi) \geq C\lambda^{-\left(\frac{1}{\alpha} - \frac{1}{2}\min\left({N\over q}+\gamma,N\right)\right)^{-1}},
\end{equation}
for all $0<\lambda\leq 1,$ where $C$ is a positive constant, $C = C(\alpha,q,\gamma,N,\|\varphi\|_{L^q_\gamma\cap L^q})$ if $\gamma<N$ and $C = C(\alpha,q,\gamma,N,\|\varphi\|_1)$ if $\gamma>N.$
\end{cor}

\begin{rem}
{\rm $\,$
\begin{itemize}
\item[1)] For the particular case $q=\infty,$ Corollary \ref{lowerLqLqgamma} includes that of \cite[Theorem 3.21, p. 376]{LeeNI} and \cite[Theorem 5.1, p. 128]{HisaIshige}, where  $\varphi$ is continuous, $\varphi\geq 0,$ $\varphi\in L^\infty(\R^N)\cap L^\infty_\gamma(\R^N),\; \gamma>0.$
The novelty of our estimate {is} that it holds  without any condition on the sign of the initial data. Unlike \cite{LeeNI,HisaIshige}, the case $\gamma=N$ is not considered here.
 \item[2)] Corollary \ref{lowerLqLqgamma}  is totally new if $q<\infty.$
 \item[3)] Obviously, if $\lambda>1,$ \eqref{lowerNLH} is better than \eqref{bulambdagamma3}, which itself holds for all $\lambda>0,$ as shown in the proof.
 \end{itemize}}
\end{rem}

The solution of \eqref{NLH}-\eqref{NLHini} constructed with initial data in $L^q_\gamma(\R^N)$ { is in $C_0(\R^N)$ for $t>0$,} by Proposition  \ref{C0} below. This is well-known to hold also for the solution with initial data in $L^p(\R^N), \; p<\infty.$ So the constructed solution for initial data in $L^p(\R^N)\cap L^q_\gamma(\R^N)$ can be extended to a maximal solution of  \eqref{NLH}-\eqref{NLHini}, $u : (0,T_{\max}) \to C_0(\R^N)$. This maximal existence time is {  equal to that} in $L^p(\R^N)$ or in $L^q_\gamma(\R^N),$ as shown in Proposition  \ref{C0} below. In the following result, which extends Corollary \ref{lowerLqLqgamma} for $0<\gamma<N,$ we give {  a  lower bound estimate of the life-span}  for initial data in $L^p(\R^N)\cap L^q_\gamma(\R^N)$.
\begin{cor}
\label{lowerLpLqgamma}
 Let $N\geq 1,$ $\alpha>0$  and $q_c$ be given by \eqref{qcritique}. If $\varphi \in L^p(\R^N)\cap L^q_\gamma(\R^N) $, where $p>q_c,\; 1\leq p\leq \infty$, $0< \gamma<N,\; \gamma<2/\alpha, \; q\in (1,\infty]$ and { satisfies \eqref{qcondition}, then the existence time of  the maximal solution} of \eqref{NLH}-\eqref{NLHini} with $u_0=\lambda\varphi$ satisfies
\begin{equation}
\label{bulambdagamma3pq}
T_{\max}(\lambda\varphi) \geq C \begin{cases}{\lambda^{-\left(\frac{1}{\alpha} -\frac{N}{2} \max\left(\frac{1}{q}+{\gamma\over N},\frac{1}{p}\right)\right)^{-1}}},\; \mbox{if}\quad\, 0<\lambda\leq 1,\\
  {\lambda^{-\left(\frac{1}{\alpha} - \frac{N}{2} \min(\frac{1}{q}+{\gamma\over N},\frac{1}{p})\right)^{-1}}},\; \mbox{if}\quad\, \lambda>1,
  \end{cases}
\end{equation}
where $C = C(\alpha,p,q,\gamma,N,\|\varphi\|_{L^p\cap L^q_{{\gamma}}})>0$ is a constant.
\end{cor}

We now turn to results based on the second method, which gives  upper-bounds on the life-span  and which  requires positivity. We distinguish the cases when  $\lambda>0$ is large or $\lambda>0$ is small and begin with the case  $\lambda$ is large. By \cite{L}, $T_{\max}(\lambda\varphi) <\infty$ in this case, if $\varphi$ is sufficiently regular. {  By \cite[Theorem 1]{W5}, it follows  that if $\varphi\ge 0$ is either a locally integrable function or a positive Borel measure on $\Omega$,  $\varphi \not\equiv 0$,  then  $T_{\max}(\lambda\varphi) < \infty$ for all sufficiently large $\lambda > 0.$ See section 5. See also \cite[Theorem 2, p. 882]{MW} for $\varphi\in C_0(\Omega)$.}
Our first life-span upper bound is as follows.
\begin{To}
\label{upperLinfty}
Let $N\geq 1$ and $\alpha>0$. Let $\varphi \in L^\infty(\Omega)$, $\varphi \ge 0$ and $\varphi \not\equiv 0$.
It follows that the existence time for the  maximal solution of \eqref{NLH}-\eqref{NLHini} with $u_0=\lambda\varphi$ satisfies { $T_{\max}(\lambda\varphi)<\infty$ for $\lambda>0$ sufficiently large and}
\begin{equation}
\label{upperestcont}
\limsup_{\lambda \to \infty} \lambda^\alpha T_{\max}(\lambda\varphi) \le \frac{1}{\alpha\|\varphi\|_\infty^\alpha}.
\end{equation}
\end{To}
\begin{rem}
{\rm $\,$
\begin{itemize}
\item[1)] Theorem~\ref{ThNLH} and Theorem~\ref{upperLinfty} together show that
\begin{equation}
T_{\max}(\lambda\varphi) \sim (\lambda\|\varphi\|_\infty)^{-\alpha}, \; \lambda \to \infty
\end{equation}
whenever $\varphi \ge 0$, $\varphi \not\equiv 0$, $\varphi \in L^\infty(\Omega)$. This extends the result of \cite[Theorem 3.2 (ii), p. 372]{LeeNI} to $L^\infty$ initial data.
The lower estimate is valid even if $\varphi$ is not necessarily positive.
\item[2)] With the notation of Part 5) of Remark \ref{remarksLq} and using Theorem~\ref{upperLinfty}, we have that for $\varphi \ge 0$, $\varphi \not\equiv 0$, $\varphi \in L^\infty(\Omega),$ the maximal existence time of \eqref{diffusioncoef} satisfies
$$\limsup_{\mu\searrow 0}T^\mu_{\max}(\varphi)\leq {1\over \alpha}\|\varphi\|_\infty^{-\alpha}$$  and hence combined with  Part 5) of Remark \ref{remarksLq} we have $T^\mu_{\max}(\varphi)\sim {1\over \alpha}\|\varphi\|_\infty^{-\alpha},$ as $\mu \to 0.$ It is shown in \cite[Theorem 1, p. 351]{MY2}  that $\lim_{\mu\searrow 0}T^\mu_{\max}(\varphi)={1\over \alpha}\|\varphi\|_\infty^{-\alpha},$  without sign restriction on $\varphi$ but, unlike our case, only for $\Omega $ a bounded domain and assuming also  $\varphi$ a continuous function on $\overline{\Omega}$.
\item[3)] {  Theorem~\ref{upperLinfty} is known for bounded domain and regular initial data, see \cite{Sperb,Sato}. See also \cite[Remark 17.2(i), p. 92]{QS3} for other estimates in bounded domain.}
\end{itemize}}
\end{rem}

We now consider positive initial data {  which are singular}   near the origin, where we restrict ourselves to  the case $\Omega=\R^N.$ We have obtained the following.
\begin{To}
\label{uppernegpower1}
Let $N\geq 1$ and $\alpha>0$. Let $0 < \gamma < N$, $\gamma < \frac{2}{\alpha}$ and  let $\omega \in L^\infty(\R^N)$
be homogeneous of degree $0$, $\omega \ge 0$ and $\omega \not\equiv 0$. Suppose
that $\varphi \in L^1_{loc}(\R^N)$, $\varphi \ge 0$
is such that $\varphi(x) \ge \omega(x)|x|^{-\gamma}$ for $|x| \le \epsilon$, and some $\epsilon > 0$.
It follows that { $T_{\max}(\lambda\varphi)<\infty$ for $\lambda>0$ sufficiently large and}
\begin{equation}
\limsup_{\lambda \to \infty} \lambda^{(\frac{1}{\alpha} - \frac{\gamma}{2})^{-1}} T_{\max}(\lambda\varphi)
\le \frac{1}{(\alpha^{1/\alpha}\|{\rm e}^{\Delta}(\omega|\cdot|^{-\gamma})\|_\infty)^{(\frac{1}{\alpha} - \frac{\gamma}{2})^{-1}}}.
\end{equation}
\end{To}
\begin{rem}
{\rm
$\,$
\begin{itemize}
\item[1)] If $\varphi$ is as in  { Theorem  \ref{uppernegpower1}} such that  $\varphi \in L^\infty_\gamma(\R^N),$ $0 < \gamma < N$, $\gamma < \frac{2}{\alpha}$ then Theorem~\ref{lowerLqgamma} and Theorem~\ref{uppernegpower1} together show that $ T_{\max}(\lambda\varphi)\sim \lambda^{-(\frac{1}{\alpha} - \frac{\gamma}{2})^{-1}},$ as $\lambda\to\infty.$ This extends the result of \cite[Proposition 4.5]{TW2} by removing the condition $(N-2)\alpha<4$, as well as the condition $\liminf_{|x|\to 0}|x|^\gamma\varphi(x)>0.$
\item[2)] If $N < \gamma < \frac{2}{\alpha}$, then there is no local nonnegative solution to \eqref{NLH} with initial value
$\lambda\tilde\varphi$ \ for all $\lambda > 0$,  where  \begin{equation}
\label{phitilde}
\tilde\varphi(x) = \begin{cases}
\omega(x)|x|^{-\gamma}, &|x| \le \epsilon\\
0, 		&|x| > \epsilon.
\end{cases}
\end{equation}
  See \cite{W5,CDNW}.
\end{itemize}}
\end{rem}

We now turn to upper estimates on $T_{\max}(\lambda\varphi)$ as $\lambda \to 0$.
For this we need to assume that $\Omega$ is not bounded, and for simplicity we
consider $\Omega = \R^N$.  Our first result of this type is for measures. Consider $u_0 = \lambda\m$, where $\lambda > 0$ and $\m \in \M$,
the set of finite Borel measures on $\R^N$. We suppose that $\m$ is a positive measure.
To insure that \eqref{NLH} is locally well-posed on $\M$ we assume
$\alpha < \frac{2}{N}$, and this implies (by Fujita's result) that $T_{\max}(\lambda\m) < \infty$ for all $\lambda > 0$.

\begin{To} Let $N\geq 1$ and $\alpha>0$. Suppose $\alpha < \frac{2}{N}$ and let $\m \in \M$ be a positive finite Borel measure on $\R^N$.
It follows that {   $T_{\max}(\lambda\m) < \infty$
for all $\lambda > 0$ and}
\label{uppermeas}
\begin{equation}
\limsup_{\lambda \to 0} \lambda^{({1\over \alpha}-{N\over 2})^{-1}} T_{\max}(\lambda\m)
\le \frac{1}{\left(\alpha^{1/\alpha}(4\pi)^{-N/2}\|\m\|_\M\right)^{({1\over \alpha}-{N\over 2})^{-1}}}.
\end{equation}
\end{To}

\begin{rem}
\label{rem5}
{\rm
Theorem~\ref{uppermeas} includes the case $u_0 = \lambda\varphi$ where $\varphi \in L^1(\R^N)$,
$\varphi \ge 0$, $\varphi \not\equiv 0$ and $\alpha < \frac{2}{N}$.  Indeed, consider the measure $d\m = \varphi dx$
where $dx$ denotes Lebesgue measure. It follows then that Theorem~\ref{ThNLH} and Theorem~\ref{uppermeas} together show that
\begin{equation}
T_{\max}(\lambda\varphi) \sim (\lambda\|\varphi\|_1)^{-({1\over \alpha}-{N\over 2})^{-1}}, \; \lambda \to 0,
\end{equation}
whenever $\varphi \ge 0$, $\varphi \not\equiv 0$, $\varphi \in L^1(\R^N)$.
The lower estimate is valid even if $\varphi$ is not necessarily positive.
}
\end{rem}

We have obtained the following  for positive initial data  having some decay at infinity.
\begin{To}
\label{uppernegpower2}
Let $N\geq 1$ and $\alpha>0$. Let $\varphi \in L^1_{loc}(\R^N)$, $\varphi \ge 0$ and
suppose that
$\varphi(x) \ge \omega(x)|x|^{-\gamma}$ for $|x| \ge R$, for some $R > 0$, where
$\omega \in L^\infty(\R^N)$ is homogeneous of degree $0$, $\omega \ge 0$ and $\omega \not\equiv 0$.
If  $0 < \gamma < N$ and $\gamma < \frac{2}{\alpha}$, then $T_{\max}(\lambda\varphi) < \infty$
for all $\lambda > 0$ and
\begin{equation}
\limsup_{\lambda \to 0} \lambda^{(\frac{1}{\alpha} - \frac{\gamma}{2})^{-1}} T_{\max}(\lambda\varphi)
\le \frac{1}{(\alpha^{1/\alpha}\|{\rm e}^{\Delta}(\omega|\cdot|^{-\gamma})\|_\infty)^{(\frac{1}{\alpha} - \frac{\gamma}{2})^{-1}}}.
\end{equation}
\end{To}

\begin{rem}
{\rm $\,$
\begin{itemize}
\item[1)] If $\varphi$ is as in { Theorem \ref{uppernegpower2}}  such that  $\varphi \in L^\infty_\gamma(\R^N),$ $0 < \gamma < N$, $\gamma < \frac{2}{\alpha}$ then Theorem~\ref{lowerLqgamma} and Theorem~\ref{uppernegpower2} together show that $ T_{\max}(\lambda\varphi)\sim \lambda^{-(\frac{1}{\alpha} - \frac{\gamma}{2})^{-1}},$ as $\lambda\to 0.$ This extends the result of \cite[Proposition 4.5]{TW2} by removing the condition $(N-2)\alpha<4$, as well as the condition $\liminf_{|x|\to \infty}|x|^\gamma\varphi(x)>0.$
    \item[2)] If $N < \gamma < \frac{2}{\alpha}$, then $\tilde{\tilde\varphi} \in L^1(\R^N)$, where
\begin{equation}
\label{phitildetilde}
\tilde{\tilde\varphi}(x) = \begin{cases}
0, &|x| < R\\
\omega(x)|x|^{-\gamma}, 	&|x| \ge R,
\end{cases}
\end{equation}
for some $R>0,$ and so Theorem~\ref{uppermeas} gives an upper life-span bound as $\lambda \to 0$. So for $\varphi$ as in  { Theorem \ref{uppernegpower2}}  with $\gamma>0,\; \gamma\not=N,$ and by comparison argument, Remark \ref{rem5} and the above one  together show that $ T_{\max}(\lambda\varphi)\sim \lambda^{-(\frac{1}{\alpha} - \frac{1}{2}\min(\gamma,N))^{-1}},$ as $\lambda\to 0.$
\item[3)] In the particular case where $\varphi$ is continuous and bounded such that $\liminf_{|x|\to\infty}|x|^{\gamma}\varphi(x)>0,$ {  a }similar result is obtained in \cite[Theorem 3.15 (ii)]{LeeNI}. If $\varphi\in L^\infty(\R^N)$ and is nonnegative satisfying $\varphi(x)\geq (1+|x|)^{-\gamma}$ for almost all $x\in \R^N$ { a} similar result is { also} obtained in  \cite[Theorem 5.2 (ii)]{HisaIshige}. Here $\varphi$ is only $L^1_{loc}(\R^N)$, {  and so} the condition on lower bound on $\varphi$ is imposed only near infinity and we do not require $\liminf_{|x|\to \infty}|x|^\gamma\varphi(x)>0.$ { In fact, by taking for example $\omega(x)=|x_1|/|x|,$  we have $\liminf_{|x|\to \infty}|x|^\gamma\varphi(x)=0.$} We also give an explicit upper bound.
\end{itemize}}
\end{rem}

We now consider upper-bounds of the life-span for sign changing initial data. We  define the sector \begin{equation} \label{dmn}%
\Omega_m= \biggl\{x=(x_1,\; x_2,\; \cdots,\; x_N) \in
\Rd;\; x_1>0,\; x_2>0,\cdots,\; x_m>0\biggr\},
\end{equation}
where $1\leq m\leq N$ is an integer. For $0<\gamma<N$ and integer $1 \le m \le N$, we let $\psi_0\, :\, \Omega_m \rightarrow\, \R$ be given by
\begin{equation} \label{psi0}%
\psi_0(x) = c_{m, \gamma} x_1 \cdots x_m  \vert x \vert^{-\gamma-2m} ,\, \; x\in \Omega_m,
\end{equation}
where
\begin{equation} \label{cstcmg}%
c_{m, \gamma}=\gamma (\gamma+2) \cdots (\gamma+2m-2).
\end{equation}
In \cite{TW2} local well-posedness for $\varphi \in L^1_{loc}(\R^N\setminus\{0\})$, anti-symmetric with respect to $x_1,\; x_2,\;\cdots,\; x_m,$ and $\varphi_{|\Omega_m}$ is in the Banach space
\begin{equation}
\label{spc}
{\mathcal X}= \left\{ \psi \in L^1_{\rm loc}(\Omega_m) ;\quad
\frac{\psi}{\psi_0} \in L^\infty(\Omega_m)\right\},
\end{equation}
have been shown for $0<\alpha<2/(\gamma+m).$ The solution can be extended to maximal solution
$u : \left(0,T_{\max}(\varphi)\right) \to C_0(\R^N)$. Furthermore, there exits a constant $C>0,$ such that
\begin{equation}
\label{lowerm}
\lambda^{[({1\over \alpha}-{\gamma+m\over 2})^{-1}]}T_{\max}(\lambda \varphi) \geq C,
\end{equation}
for all $\lambda>0.$ We denote by $ {\rm e}^{t\Delta_m}$ the heat semigroup on $\Omega_m.$ We have obtained the following for large $\lambda$.
\begin{To}
\label{uppernegpower1signchanging}
Let the positive integer $m$ and the real numbers
$\alpha,\; \gamma$ be such that
$$
1\leq m \leq N,\; 0<\gamma<N, \; 0<\alpha<{2\over \gamma+m}.
$$ Suppose
that $\varphi \in L^1_{loc}(\R^N\setminus\{0\})$, anti-symmetric with respect to $x_1,\; x_2,\;\cdots,\; x_m,$ $\varphi_{|\Omega_m}\in
{\mathcal X},$ $\varphi \ge 0$ in $\Omega_m$
is such that $\varphi(x) \ge \,\omega(x)\psi_0(x)$ for $ x\in \Omega_m\cap\{|x| \le \epsilon\}$, for some $\epsilon > 0,$  where
$\omega \in L^\infty(\R^N)$ is homogeneous of degree $0$, anti-symmetric with respect to $x_1,\; x_2,\;\cdots,\; x_m,$  $\omega \ge 0$ on $\Omega_m$ and $\omega \not\equiv 0$.
It follows that { $T_{\max}(\lambda\varphi)<\infty$ for $\lambda>0$ sufficiently large and}
\begin{equation*}
\limsup_{\lambda \to \infty} \lambda^{(\frac{1}{\alpha} - \frac{\gamma+m}{2})^{-1}} T_{\max}(\lambda\varphi)
\le \frac{1}{(\alpha^{1/\alpha}\|{\rm e}^{\Delta_m}(\omega\psi_0)\|_\infty)^{(\frac{1}{\alpha} - \frac{\gamma+m}{2})^{-1}}}.
\end{equation*}
\end{To}

We have obtained the following for small $\lambda$.
\begin{To}
\label{uppernegpower2signchanging}
Let the positive integer $m$ and the real numbers
$\alpha,\; \gamma$ be such that
$$
1\leq m \leq N,\; 0<\gamma<N, \; 0<\alpha<{2\over \gamma+m}.
$$ Suppose
that $\varphi \in L^1_{loc}(\R^N\setminus\{0\})$, anti-symmetric with respect to $x_1,\; x_2,\;\cdots,\; x_m,$ $\varphi_{|\Omega_m}\in
{\mathcal X},$ $\varphi \ge 0$ in $\Omega_m$
is such that $\varphi(x) \ge \,\omega(x)\psi_0(x)$ for $ x\in \Omega_m\cap\{|x| \geq R\}$, for some $R> 0,$ where
$\omega \in L^\infty(\R^N)$ is homogeneous of degree $0$, anti-symmetric with respect to $x_1,\; x_2,\;\cdots,\; x_m,$  $\omega \ge 0$ on $\Omega_m$ and $\omega \not\equiv 0$.
It follows that $T_{\max}(\lambda\varphi) < \infty$
for all $\lambda > 0$ and
\begin{equation*}
\limsup_{\lambda \to 0} \lambda^{(\frac{1}{\alpha} - \frac{\gamma+m}{2})^{-1}} T_{\max}(\lambda\varphi)
\le \frac{1}{(\alpha^{1/\alpha}\|{\rm e}^{\Delta_m}(\omega\psi_0)\|_\infty)^{(\frac{1}{\alpha} - \frac{\gamma+m}{2})^{-1}}}.
\end{equation*}
\end{To}

\begin{rem}
{\rm $\,$
\begin{itemize}
\item[1)] Theorems   \ref{uppernegpower1signchanging} and \ref{uppernegpower2signchanging} improve the results of \cite[Theorem 1.10, Proposition  4.5]{TW2} by removing the condition $(N-2)\alpha<4.$ Also the conditions on the $\liminf_{|x|\to \infty}{|x|^{\gamma+m}\over x_1x_2\cdots x_m}\varphi(x)>0$ or on the $\liminf_{|x|\to 0}{|x|^{\gamma+m}\over x_1x_2\cdots x_m}\varphi(x)>0$ are not required here.
\item[2)] Theorem  \ref{uppernegpower1signchanging} (respectively Theorem \ref{uppernegpower2signchanging})  together with \eqref{lowerm} show that $$T_{\max}(\lambda \varphi)\sim \lambda^{-({1\over \alpha}-{\gamma+m\over 2})^{-1}},$$
as $\lambda\to \infty$ (respectively as $\lambda\to 0$).
\end{itemize}}
\end{rem}

The proofs of { the known results cited above} are based on  careful constructions of  super and sub-solutions, comparison and Kaplan's arguments. See, for example,  \cite{Sato,Li,YY,Y,OY} and some references therein.   In the case of decaying initial data, the results are derived via  a careful analysis of the asymptotic in the $L^\infty$-norm of the solutions to the linear heat equation on $\R^N$ with initial { data} having specific orders of decay at space infinity as well as  Kaplan's arguments and comparison principles, see \cite{LeeNI}. This method, \cite{LeeNI}, has been used in many papers in the last three decades, see for example \cite{Li,ZW,ZW2,Cao,Cao2} and references therein. Most of the results require that $\lambda$ be either sufficiently large or sufficiently small and initial data are positive and regular.  Also, some scaling arguments are applied to derive life-span estimates, such as in \cite{GW,D}.

It interesting to compare the two methods used to prove our  results above.    The proof of lower bounds as already mentioned,  is based on the contraction mapping argument which gives local well-posedness of solutions (as in \cite{TW2}).   Consequently,  it does not require any positivity condition or maximum principle. To prove the upper estimates, we use a necessary condition  for local existence of non-negative solutions established in \cite{W5} (see Proposition \ref{lincd} below), combined with the maximum principle, continuity properties of the heat semigroup and scaling argument.   For these estimates, positivity is required.
There is a certain unity  in these  two  methods. On the one hand, the contraction mapping argument gives a sufficient
condition on $T > 0$ for the existence of a solution on the interval $[0,T]$ for some initial value $u_0.$ This condition takes the form of an inequality involving both $T$ and $u_0.$ This condition must fail for $T = T_{\max},$ which implies that the opposite inequality must hold.   When this inequality is applied to initial values of the form
$u_0 = \lambda\varphi,$ this results in a lower life-span estimate. On the other hand, inequality in \cite[Theorem 1]{W5} gives a necessary condition on $T > 0$ for the existence of a (positive) solution on the interval $[0,T],$ for some initial value $u_0 \ge 0.$  This condition must hold for all $T < T_{\max}.$ Moreover, this condition is stable under limits, and so must hold in the case $T = T_{\max}.$ When the resulting inequality is applied to initial values of the form
$u_0 = \lambda\varphi,$ an upper life-span estimate is obtained. We note that the lower estimates for $T_{\max}(\lambda\varphi)$ do not in and of themselves prove finite time blowup, while the upper estimates do so.

Our results based on scaling, the third approach in this paper, on the one hand use ideas introduced in \cite{D}, and on the other  hand comparison arguments. In particular, we give life-span estimates for an initial value of the form 
\begin{equation}
\label{phi3}
\Phi(x) = \begin{cases}
\omega(x)|x|^{-\gamma_1}, &|x| \le 1\\
\omega(x)|x|^{-\gamma_2}, &|x| > 1.
\end{cases}
\end{equation}
where $0 < \gamma_1, \gamma_2 < N$ and $\gamma_1, \gamma_2 < \frac{2}{\alpha}$ ($\gamma_1 \neq \gamma_2$)
and $\omega \in L^\infty(\R^N)$ is homogeneous of degree $0$, $\omega \ge 0$, $\omega \not \equiv 0$. See Corollary \ref{resulphi3} below. We show, in particular, the impact of the singularity on the life-span for $\lambda$ large  and the impact of the decay at infinity on the life-span for $\lambda$ small.

The rest of this paper is organized as follows.  In Section 2, we consider the standard nonlinear heat equation and prove Theorems \ref{ThNLH} and \ref{lowermeas}.  In Section 3, we prove new estimates for the heat kernel in weighted Lebesgue spaces, see Proposition \ref{smoothingeffectorl-lebeg} below. Section 4 is devoted  to the case of slowly decaying initial data and the proofs of Theorem \ref{lowerLqgamma} and Corollaries \ref{lowerLqLqgamma} and \ref{lowerLpLqgamma}. The upper estimates, Theorems \ref{upperLinfty}--\ref{uppernegpower2signchanging}, are proved  in Section 5.  In Section 6, we establish life-span estimates via nonlinear scaling.  In the appendix, we give some estimates of the life-span for Hardy-H\'enon parabolic equations.

Throughout  the paper, $C$ will be a positive constant which may vary from line to line. For positive functions $f$ and $g,$  we say that $f(x)\sim g(x)$ as $x\to x_0$ if there exists two positive constants $C_1$ and $C_2$ such that $C_1g(x)\leq f(x)\leq C_2g(x)$ in a neighborhood of $x_0.$

\section{Lower bounds for initial data in Lebesgue spaces}
\setcounter{equation}{0}

We consider the integral equation corresponding to the problem \eqref{NLH}-\eqref{NLHini}
\begin{equation}
\label{NLHint}
u(t) = {\rm e}^{t\Delta}u_0 + \int_0^t {\rm
e}^{(t-\sigma)\Delta} \big[|u(\sigma)|^\alpha u(\sigma)\big]
d\sigma,
\end{equation}
where ${\rm e}^{t\Delta}$ is the heat semigroup on $\Omega$.  It is known that the integral
kernel corresponding to ${\rm e}^{t\Delta}$ is bounded by the Gauss kernel for the
heat semigroup on $\R^N$.  Hence the $L^q - L^r$ smoothing inequalities are
independent of $\Omega$, i.e.
\begin{equation}
\label{heatsmouth}
\|{\rm e}^{t\Delta}u_0\|_{L^r(\Omega)} \le (4\pi t)^{-\frac{N}{2}(\frac{1}{q} - \frac{1}{r})}\|u_0\|_{L^q(\Omega)}
\end{equation}
whenever $1 \le q \le r \le \infty$.

We recall for future use that in the case of $\Omega = \R^N$
\begin{equation}
\label{hsgdil1}
D_\tau{\rm e}^{t\Delta} = {\rm e}^{(t/\tau^2)\Delta}D_\tau
\end{equation}
where $D_\tau$ is the dilation operator $D_\tau f(x) = f(\tau x)$.
In particular
\begin{equation}
\label{hsgdil2}
D_{\sqrt t}{\rm e}^{t\Delta} = {\rm e}^{\Delta}D_{\sqrt t}.
\end{equation}

In this section the goal is to establish lower bounds for the life-span of solutions
as an immediate consequence of the fixed point argument used to prove well-posedness of
\eqref{NLHint} in certain Banach spaces.  While this argument is well-known, in order to show the
applications to life-span, it is more convenient to recall some of the details.

To this end we recall the value
\begin{equation*}
q_c = \frac{N\alpha}{2},
\end{equation*}
and we require that $q$ and $r$ satisfy the following conditions:
\begin{equation}
\label{supcrit}
q > q_c
\end{equation}
and
\begin{equation}
\label{compatible}
1 \le \frac{r}{\alpha+1} \le q \le r.
\end{equation}
Note that given $q > q_c$ and $q \ge 1$, one can always choose $r = q(\alpha + 1)$.
Also, if $q > q_c$ and $q \ge \alpha + 1$, one can choose
$r = q$.  Furthermore, in all cases above, we have $r \ge \alpha + 1$.
Finally, if $q = \infty$, then necessarily $r = \infty$.
We set
\begin{equation}
\label{beta}
\beta = \frac{N}{2}\Big(\frac{1}{q} - \frac{1}{r}\Big).
\end{equation}

We next define the space of curves in which we seek a solution
to \eqref{NLHint}, i.e. the space in which we carry
out the contraction mapping argument.  For a fixed $M > 0$ and $T > 0$, (and $q$, $r$ and $\beta$
as above), we set
\begin{equation}
\label{Yspace}
Y_{M,T}^{q,r} = \{u \in C((0,T]; L^r(\Omega)) :
\sup_{t\in(0,T]} t^\beta\|u(t)\|_r \le M\}.
\end{equation}
With the distance
\begin{equation}
\label{metric}
d(u,v) = d_{M,T}^{q,r}(u,v) = \sup_{t \in (0,T]}t^\beta\|u(t) - v(t)\|_r
\end{equation}
the space $Y_{M,T}^{q,r}$ is a complete metric space.

To carry out the fixed point argument, we let $u_0 \in \Dd$
and suppose that there exists $K > 0$ such that
\begin{equation}
\label{initcd}
\sup_{t \in (0,T]}t^\beta\|e^{t\Delta}u_0\|_r \le K.
\end{equation}
This condition includes implicitly the condition that  $e^{t\Delta}u_0$ be well-defined and
in $L^r(\Omega)$ for $t > 0$. Recall that if $u_0 \ge 0,$ then  $e^{t\Delta}u_0$ indeed is well-defined,
but perhaps infinite.  In order for \eqref{initcd}
to hold, it suffices for example that $\sup_{t \in (0,T]}t^\beta\|e^{t\Delta}|u_0|\|_r \le K$,
since $|e^{t\Delta}u_0| \le e^{t\Delta}|u_0|$.
We define the iterative operator by
\begin{equation}
\label{iteratn}
\F_{u_0} u(t) = {\rm e}^{t\Delta}u_0 + \int_0^t {\rm
e}^{(t-\sigma)\Delta} \big[|u(\sigma)|^\alpha u(\sigma)\big]
d\sigma.
\end{equation}

The following theorem is well-known.  Since we are particularly interested here in the contraction
mapping property, we sketch that part of the proof.

\begin{To}
\label{fixpt1}
Let $N\geq 1,$ $\alpha > 0,$  $1 \le q \le \infty$ and $q > q_c$.
There is a constant $C = C(\alpha,q)>0$ such that if
$K > 0$, $M > 0$ and $T > 0$ satisfy
\begin{equation}
\label{fixptcd}
K + CT^{1 - \frac{N\alpha}{2q}}M^{\alpha + 1} \le M,
\end{equation}
and  if $u_0 \in \Dd$ satisfies \eqref{initcd} for some $r \ge q$ with $1 \le \frac{r}{\alpha+1} \le q \le r$,
then $\F_{u_0}$ is a strict contraction on $Y_{M,T}^{q,r}$ and so has a unique fixed point
$u = \F_{u_0} u \in Y_{M,T}^{q,r}$.  This solution $u$ of \eqref{NLHint} is a classical solution of \eqref{NLH} on $(0,T]$.

Furthermore, if $u_0 \in L^q(\Omega)$ and $q < \infty$, then this fixed point has the property that $u \in C([0,T];L^q(\Omega))$
with $u(0) = u_0$.
\end{To}

\begin{rem}
{\rm
Of course the sufficient condition \eqref{fixptcd} can be taken as
\begin{equation}
\label{fixptcd2}
\sup_{t \in (0,T]}t^\beta\|e^{t\Delta}u_0\|_r + CT^{1 - \frac{N\alpha}{2q}}M^{\alpha + 1} \le M,
\end{equation}
i.e. taking equality in \eqref{initcd}}
\end{rem}

\begin{proof}
We first consider when the space $Y_{M,T}^{q,r}$ is preserved by the iterative operator $\F_{u_0}$.
Thus we suppose $u \in Y_{M,T}^{q,r}$, and we estimate $\F_{u_0}u(t)$ as follows.
\begin{eqnarray*}
t^\beta \|\F_{u_0}u(t)\|_r &\le& t^\beta \|{\rm e}^{t\Delta}u_0\|_r + t^\beta \int_0^t \|{\rm
e}^{(t-\sigma)\Delta} \big[|u(\sigma)|^\alpha u(\sigma)\big]\|_r d\sigma\\
&\le& K +
 t^\beta \int_0^t [4\pi(t-\sigma)] ^{-\frac{N\alpha}{2r}} \||u(\sigma)|^\alpha u(\sigma)\|_{{r\over \alpha+1}} d\sigma\\
 &=& K +
 (4\pi)^{-\frac{N\alpha}{2r}}t^\beta \int_0^t (t-\sigma)^{-\frac{N\alpha}{2r}} \|u(\sigma)\|_r^{\alpha+1} d\sigma\\
  &\le& K +
(4\pi)^{-\frac{N\alpha}{2r}}t^\beta \left(\int_0^t (t-\sigma)^{-\frac{N\alpha}{2r}} \sigma^{-\beta(\alpha+1)} d\sigma\right) M^{\alpha+1}\\
  &\le& K +
(4\pi)^{-\frac{N\alpha}{2r}}t^{1-\frac{N\alpha}{2q}}\left( \int_0^1(1-\sigma)^{-\frac{N\alpha}{2r}} \sigma^{-\beta(\alpha+1)} d\sigma\right) M^{\alpha+1}
\\
  &\le& K +
(4\pi)^{-\frac{N\alpha}{2r}} \left(\int_0^1(1-\sigma)^{-\frac{N\alpha}{2r}} \sigma^{-\beta(\alpha+1)} d\sigma\right)
T^{1-\frac{N\alpha}{2q}}M^{\alpha+1}\\&\le& K +2(\alpha+1)
(4\pi)^{-\frac{N\alpha}{2r}} \left(\int_0^1(1-\sigma)^{-\frac{N\alpha}{2r}} \sigma^{-\beta(\alpha+1)} d\sigma\right)
T^{1-\frac{N\alpha}{2q}}M^{\alpha+1}.
\end{eqnarray*}
It follows that if \eqref{fixptcd} holds, where
\begin{equation*}
C = C(\alpha, q, r)  = 2(\alpha + 1)(4\pi)^{-\frac{N\alpha}{2r}} \int_0^1(1-\sigma)^{-\frac{N\alpha}{2r}}\sigma^{-\beta(\alpha+1)} d\sigma,
\end{equation*}
then $Y_{M,T}^{q,r}$ is stable by $\F_{u_0}.$

Next we show that $\F_{u_0}$ is a strict contraction on $Y_{M,T}^{q,r}$. We estimate as follows.
\begin{eqnarray*}
t^\beta \|\F_{u_0} u(t) - \F_{u_0} v(t)\|_r &\le&
 t^\beta \int_0^t \|{\rm
e}^{(t-\sigma)\Delta} \big[|u(\sigma)|^\alpha u(\sigma) - |v(\sigma)|^\alpha v(\sigma)\big]\|_r d\sigma\\
&\le&
 (4\pi)^{-\frac{N\alpha}{2r}}t^\beta \int_0^t (t-\sigma) ^{-\frac{N\alpha}{2r}}
  \|\big[|u(\sigma)|^\alpha u(\sigma) - |v(\sigma)|^\alpha v(\sigma)\big]\|_{{r\over \alpha+1}} d\sigma\\
  &\le&
  (\alpha + 1)(4\pi)^{-\frac{N\alpha}{2r}}t^\beta \int_0^t (t-\sigma) ^{-\frac{N\alpha}{2r}}
  \|[u(\sigma) - v(\sigma)]\big[|u(\sigma)|^\alpha +|v(\sigma)|^\alpha\big]\|_{{r\over \alpha+1}} d\sigma\\
 &\le&
 (\alpha + 1)(4\pi)^{-\frac{N\alpha}{2r}}t^\beta \int_0^t (t-\sigma)^{-\frac{N\alpha}{2r}}
 \|u(\sigma) - v(\sigma)\|_r\big[\|u(\sigma)\|_r^\alpha+\|v(\sigma)\|_r^\alpha\big]
d\sigma\\
  &\le&
 2(\alpha + 1)(4\pi)^{-\frac{N\alpha}{2r}}t^\beta \left(\int_0^t (t-\sigma)^{-\frac{N\alpha}{2r}} \sigma^{-\beta(\alpha+1)} d\sigma\right) M^\alpha d_{M,T}^{q,r}(u,v)\\
  &\le&
 2(\alpha + 1)(4\pi)^{-\frac{N\alpha}{2r}}t^{1-\frac{N\alpha}{2q}} \left(\int_0^1(1-\sigma)^{-\frac{N\alpha}{2r}} \sigma^{-\beta(\alpha+1)} d\sigma\right)M^\alpha d_{M,T}^{q,r}(u,v)
\\
  &\le&
2(\alpha + 1)(4\pi)^{-\frac{N\alpha}{2r}} \left(\int_0^1(1-\sigma)^{-\frac{N\alpha}{2r}} \sigma^{-\beta(\alpha+1)} d\sigma\right)
T^{1-\frac{N\alpha}{2q}}M^\alpha d_{M,T}^{q,r}(u,v).
\end{eqnarray*}
Thus
\begin{equation*}
d_{M,T}^{q,r}(\F_{u_0} u,\F_{u_0} v) \le 2(\alpha + 1)(4\pi)^{-\frac{N\alpha}{2r}} \left(\int_0^1(1-\sigma)^{-\frac{N\alpha}{2r}} \sigma^{-\beta(\alpha+1)} d\sigma\right)
T^{1-\frac{N\alpha}{2q}}M^\alpha d_{M,T}^{q,r}(u,v).
\end{equation*}
It follows that if \eqref{fixptcd} holds then $\F_{u_0}$ is a strict contraction on $Y_{M,T}^{q,r}$.

The only difficulty is that $C$ potentially depends on $r$ as well as $q$ and $\alpha$.  To rectify this, one can replace $C(\alpha, q, r)$ by
$C(\alpha, q) = \max_{q \le r \le q(\alpha + 1)}C(\alpha, q, r)$ and the result holds with $C = C(\alpha, q)$.
\end{proof}

As is well-known, Theorem~\ref{fixpt1} is used to show that the integral equation \eqref{NLHint} is locally well-posed
on $L^q(\Omega)$.  In particular, if $u_0 \in L^q(\Omega)$ the resulting solution given by the fixed point
argument can be extended to a unique maximal solution on an interval $[0, T_{\max}(u_0))$.  We will not
belabor this point further.

We have also the following.
\begin{prop}
\label{C0nlh} Let $N\geq 1,$ $\alpha > 0,$  $1 \le q < \infty$ and $q > q_c$. Let  $T_{\max}(\varphi,L^q)$ denotes the existence time of the  maximal solution of \eqref{NLHint} with initial data $\varphi\in L^q(\Omega).$  Then the following hold.
\begin{itemize}
\item[(i)] $u(t)\in C_0(\Omega)$ for $t\in \left(0,T_{\max}(\varphi,L^q)\right).$ 
\item[(ii)] If $\varphi\in L^q(\Omega)\cap C_0(\Omega)$ then  $T_{\max}(\varphi,L^q)=T_{\max}(\varphi,C_0(\Omega)),$ the existence time of the  maximal solution of \eqref{NLHint} with initial data $\varphi\in C_0(\Omega).$
\item[(iii)] If $\varphi\in L^q(\Omega)\cap L^p(\Omega)$ with $q_c< p\leq\infty$ then  $T_{\max}(\varphi,L^q)=T_{\max}(\varphi,L^p),$ the existence time of the  maximal solution of \eqref{NLHint} with initial data $\varphi\in L^p(\Omega).$
\end{itemize}
\end{prop}
\begin{proof}
(i) By iterative argument, as in \cite{BTW}, $u(t)\in L^r(\Omega)$ for $q\leq r\leq \infty.$ It is known that $e^{t\Delta} :  L^q(\Omega)\to C_0(\Omega),$  is bounded for $t>0$ and $1\leq q<\infty.$ See \cite{QS3,Davies}. Hence, by 
\eqref{NLHint}, $u(t)\in C_0(\Omega).$ 

(ii) By (i) we have $T_{\max}(\varphi,L^q)\leq T_{\max}(\varphi,C_0(\Omega)).$  Using \eqref{NLHint} and \eqref{heatsmouth}, we have
 \begin{eqnarray*}
 \|u(t)\|_{q}&\leq& \|e^{t \Delta}\varphi\|_{{q}}+   \displaystyle\int_0^t \||u(\sigma)|^{\alpha}u(\sigma)\|_{q} d\sigma\\
 &\leq& \|\varphi\|_{{q}}+   \displaystyle\int_0^t \|u(\sigma)\|^{\alpha}_\infty \|u(\sigma)\|_{q}d\sigma.
 \end{eqnarray*}
 By Gronwall's inequality, we get
 $$\|u(t)\|_{q}\leq \|\varphi\|_{{q}}e^{ \int_0^t \|u(\sigma)\|^{\alpha}_\infty d\sigma}.$$
Hence $u$ can not blow up in $L^q(\Omega)$ before it blows up in $C_0(\Omega)$. That is  $T_{\max}(\varphi,C_0(\Omega))\leq T_{\max}(\varphi,L^q).$  
 
 (iii) Let $\varepsilon\in (0, \min(T_{\max}(\varphi,L^q), T_{\max}(\varphi,L^p))).$ By (i) we have $u(\varepsilon)\in C_0(\Omega).$ Using (ii) we  have if $p<\infty,$
 $$T_{\max}(u(\varepsilon),L^q)=T_{\max}(u(\varepsilon),C_0(\Omega))=T_{\max}(u(\varepsilon),L^p).$$ That is $T_{\max}(\varphi,L^q)-\varepsilon=T_{\max}(\varphi,L^p)-\varepsilon.$ If $p=\infty,$ then $q<p$ hence (i)-(ii) hold and $T_{\max}(\varphi,L^q)-\varepsilon=T_{\max}(u(\varepsilon),C_0(\Omega))=T_{\max}(u(\varepsilon),L^\infty)=T_{\max}(\varphi,L^\infty)-\varepsilon.$ Hence we get the result.
  \end{proof}

As a first application of Theorem~\ref{fixpt1} to life-span estimates, we prove Theorem \ref{ThNLH}.
\begin{proof}[Proof of Theorem \ref{ThNLH}] Consider  $u_0 = \lambda\varphi$, where $\lambda > 0$ and $\varphi \in L^q(\Omega).$
 The key observation is that
if  $T_{\max}(\lambda\varphi) <\infty$, it is impossible to carry out the fixed point argument
on the interval $[0,T_{\max}(\lambda\varphi)]$ with initial value $u_0 = \lambda\varphi$.
Hence, by \eqref{fixptcd2}
\begin{equation*}
\sup_{t \in (0,T_{\max}(\lambda\varphi)]}t^\beta\|e^{t\Delta}u_0\|_r
+ CT_{\max}(\lambda\varphi)^{1 - \frac{N\alpha}{2q}}M^{\alpha + 1} > M,
\end{equation*}
for all $M > 0$.
 Recall that $(4\pi t)^\beta\|e^{t\Delta}u_0\|_r \le \|u_0\|_q$
by the $L^q-L^r$ smoothing properties of the heat semigroup \eqref{heatsmouth}, so that
\begin{equation*}
(4\pi)^\beta\lambda\|\varphi\|_q + CT_{\max}(\lambda\varphi)^{1 - \frac{N\alpha}{2q}}M^{\alpha + 1} > M,
\end{equation*}
for all $M > 0$.  In particular, if we set $M = 2(4\pi)^\beta\lambda\|\varphi\|_q$, this gives
\begin{equation*}
CT_{\max}(\lambda\varphi)^{1 - \frac{N\alpha}{2q}}[\lambda\|\varphi\|_q]^\alpha > 1.
\end{equation*}
Thus we have proved Theorem \ref{ThNLH}.
\end{proof}

As a second application, consider $u_0 = \lambda\m$, where $\lambda > 0$ and $\m \in \M$,
the set of finite Borel measures on $\Omega$.  For example, $\m$ could be a point mass.
In order to apply Theorem~\ref{fixpt1}, we observe first that

\begin{equation*}
|{\rm e}^{t\Delta}\m| \le (4\pi t)^{-\frac{N}{2}}\int_{\R^N}{\rm e}^{-\frac{|x-y|^2}{4t}}d|\m|(y),
\end{equation*}
where, by abuse of notation, $\m$ denotes both the measure on $\Omega$ and its natural extension
to $\R^N$, and $|\m|$ is the total variation of $\m$. Hence
\begin{eqnarray*}
\|{\rm e}^{t\Delta}\m\|_\infty &\le& (4\pi t)^{-\frac{N}{2}}\|\m\|_\M,\\
\|{\rm e}^{t\Delta}\m\|_1 &\le& \|\m\|_\M ;\\
\end{eqnarray*}
and so by interpolation
\begin{equation}
\label{measLr}
\|{\rm e}^{t\Delta}\m\|_r\le (4\pi t)^{-\frac{N}{2}(1 - \frac{1}{r})}\|\m\|_\M,
\end{equation}
for all $1 \le r \le \infty$.

Theorem~\ref{fixpt1} thus implies that \eqref{NLHint} is locally well-posed on $\M$
if $q_c < 1$.  Simply take $q = 1$ and $r = \alpha + 1$.  (This of course is well-known.)

\begin{proof}[Proof of Theorem \ref{lowermeas}] To obtain a life-span estimate, we again note that if the maximal existence time is finite, i.e. $T_{\max}(\lambda\m) < \infty$,
then \eqref{fixptcd2} can not hold
with $u_0 = \lambda\m$ and $T = T_{\max}(\lambda\m)$.  Hence, also using \eqref{measLr}, for $\beta={N\over 2}(1-{1\over \alpha+1}),$
we must have
\begin{equation*}
(4\pi)^\beta\lambda\|\m\|_\M + CT_{\max}(\lambda\m)^{1 - \frac{N\alpha}{2}}M^{\alpha + 1} > M,
\end{equation*}
for all $M > 0$.  As above, we take $M = 2(4\pi)^\beta\lambda\|\m\|_\M$, which gives the lower estimate
\begin{equation*}
CT_{\max}(\lambda\m)^{\frac{1}{\alpha} - \frac{N}{2}}\lambda\|\m\|_\M> 1.
\end{equation*}
This completes the proof of Theorem \ref{lowermeas}.
\end{proof}

As a third application of Theorem~\ref{fixpt1} to life-span estimates we consider
$u_0 = \lambda\varphi$ where $\lambda>0,$ $\varphi \in L^1_{loc}(\R^N)$ and $|\varphi| \le |\cdot|^{-\gamma}$
for some $0 < \gamma < N$.  We recall that if $\frac{N}{\gamma} < r$, then
\begin{equation}
\|{\rm e}^{t\Delta}|\cdot|^{-\gamma}\|_r
= t^{-\frac{\gamma}{2}+ \frac{N}{2r}}\|{\rm e}^{\Delta}|\cdot|^{-\gamma}\|_r
\end{equation}
for all $t > 0$.
This follows from a scaling argument.
For convenience, we set
\begin{equation}
\label{scest}
L = \|{\rm e}^{\Delta}|\cdot|^{-\gamma}\|_r.
\end{equation}
Hence if $|\varphi| \le |\cdot|^{-\gamma}$, then
\begin{equation}
\|{\rm e}^{t\Delta}(\lambda\varphi)\|_r \le \lambda\|{\rm e}^{t\Delta}|\cdot|^{-\gamma}\|_r
= L\lambda t^{-\frac{\gamma}{2}+ \frac{N}{2r}}.
\end{equation}
We next set $q = \frac{N}{\gamma}$, so that $1 < q < r$, and we may choose $r$ so that \eqref{compatible} holds.
Also, $\beta = \frac{N}{2}(\frac{1}{q} - \frac{1}{r}) = \frac{\gamma}{2}- \frac{N}{2r}$.  Theorem~\ref{fixpt1}
clearly shows that \eqref{NLHint} is locally well-posed for initial values bounded by a multiple of
$|x|^{-\gamma}$ with $0 < \gamma < N$ and $\frac{N}{\gamma} > q_c$, i.e. $\gamma < \frac{2}{\alpha}$.
This of course is known.  See \cite[Theorem 2.8]{D}, and also \cite[Theorem 2.3]{TW2}.

As for a life-span estimate, if $u_0 = \lambda\varphi$ where $|\varphi(x)|  \le |x|^{-\gamma}$,
then the existence time of the solution, $T_{\max}( \lambda\varphi)$, if it is finite, must verify
\begin{equation*}
\lambda L + CT_{\max}(\lambda\varphi)^{1 - \frac{\gamma\alpha}{2}}M^{\alpha + 1} > M
\end{equation*}
for all $M > 0$, where $L$ is given by \eqref{scest}.  For $M = 2\lambda L$, this gives
\begin{equation*}
CT_{\max}(\lambda\varphi)^{1 - \frac{\gamma\alpha}{2}}(\lambda L)^\alpha > 1.
\end{equation*}
In other words, we have the following result.

\begin{cor}
\label{negpower}
Let $0 < \gamma < N$ and $\gamma < \frac{2}{\alpha}$.  Suppose $\varphi \in L^1_{loc}(\R^N)$
is such that $|\varphi(x)| \le |x|^{-\gamma}$.  It follows that
\begin{equation}
\label{estimateseqdif}
T_{\max}(\lambda\varphi) { \geq}
 \frac{C}{(\lambda L)^{(\frac{1}{\alpha} - \frac{\gamma}{2})^{-1}}},\; \lambda>0,
\end{equation}
where $L$ is given by \eqref{scest} and $C$ depends only on $\alpha$ and $\gamma$.
\end{cor}

This last result recovers \cite[Theorem 2.6(ii)]{TW2} in the case $m = 0$, by a different but related
method: the contraction mapping argument is formulated differently.  It does not seem possible
that the contraction mapping argument used in the proof of Theorem~\ref{fixpt1} can be used
to recover \cite[Theorems 2.3 and 2.6]{TW2} in the case $1 \le m \le N$.  Indeed, that is the point
of the paper \cite{TW2}. Note also that Theorem \ref{lowerLqgamma} gives also the result but here the constant at the right-hand side is explicit.

\section{Estimates for the heat kernel in weighted spaces}
\setcounter{equation}{0}
 In this section we prove the following heat kernel estimates. For simplicity, the space $L^p(\R^N),$ will be denoted by $L^p.$ We recall that the norm in $L^p(\R^N), \|\cdot\|_{L^p(\R^N)}$ is denoted by $\|\cdot\|_p.$
 \begin{prop}
 \label{smoothingeffectorl-lebeg}
 Let $N\geq 1,$ $0\leq\gamma\leq \mu<N,$  $q_1\in(1,\infty]$ and $q_2\in(1,\infty]$  satisfy $$0\leq\frac{1}{q_2}{<}  \frac{\mu-\gamma}{N}+\frac{1}{q_1} \leq \frac{\mu}{N}+\frac{1}{q_1}<1.$$
Then  there exists a  constant $C>0$ depending on $N,\gamma,\mu,\;q_1$ and $q_2$ such that
\begin{equation}\label{99Lebegq}
\left \||.|^{\gamma}e^{t \Delta}u\right\|_{{q_2}}\leq C t^{-\frac{N}{2}\left(\frac{1}{q_1}-\frac{1}{q_2}\right)-\frac{\mu-\gamma}{2}}
\left\||.|^{\mu}u\right\|_{{q_1}},\; t>0, \; |.|^{\mu}u \in L^{q_1}.
\end{equation}
\end{prop}
\begin{rem}
\label{possibleequalityinthemiddle}
{$\,$\rm
\begin{itemize}
\item[1)]  The estimate \eqref{99Lebegq} is well-known for $\mu=\gamma=0,$ that is \eqref{heatsmouth} (see for example \cite{W2}). For the case  $\gamma=0,$ $0<\mu<N,$ \eqref{99Lebegq} is established in \cite{BTW}.  Estimate \eqref{99Lebegq} is known for $0<\gamma\leq\mu<N,$ $0\leq {1\over q_2}\leq {1\over q_1}<{1\over q_1}+{\mu\over N}<1$ in \cite{Tsutsui,CIT}. See also  \cite{KhaiTri} for the case  $q_1=q_2=\infty.$ It follows by  \cite{Tsutsui,CIT} that \eqref{99Lebegq} holds for $0<\gamma=\mu<N,$ $q_1=q_2=q\in (1,\infty],\; {1\over q}+{\mu\over N}<1$.

    \item[2)] The power of $t$ in \eqref{99Lebegq}  is optimal. This can be shown by scaling argument as in \cite{BTW}. In fact, for $t>0,$ we have
    $$e^\Delta u = D_{\sqrt t} e^{t \Delta} D_\frac{1}{\sqrt t} u$$
     for all $u \in \mathcal S'(\RR^N).$ Also  $$\||\cdot|^\gamma D_{\sqrt t} f\|_{r}=t^{-\frac{N}{2r}-{\gamma \over 2}} \||\cdot|^\gamma f\|_{r}$$
     for all $|\cdot|^\gamma f \in L^r, r\geq 1$. Writing \eqref{99Lebegq} for $t=1$ as follows
    $$  \left \||.|^{\gamma}e^{\Delta}u\right\|_{{q_2}}=\left \||.|^{\gamma}D_{\sqrt t} e^{t \Delta} D_{1/\sqrt t} u\right\|_{{q_2}}\leq C
\left\||.|^{\mu}u\right\|_{{q_1}}.$$
Setting $D_{1/\sqrt t} u=v$ that is $u=D_{\sqrt t}v,$ we get
$$ \left \||.|^{\gamma}D_{\sqrt t} e^{t \Delta} v\right\|_{{q_2}}\leq C
\left\||.|^{\mu}D_{\sqrt t}v\right\|_{{q_1}}.$$
That is
$$t^{-{N\over 2q_2}-{\gamma\over 2}} \left \||.|^{\gamma} e^{t \Delta} v\right\|_{{q_2}}\leq C
t^{-{N\over 2q_1}-{\mu\over 2}} \left\||.|^{\mu}v\right\|_{{q_1}}.$$ This gives \eqref{99Lebegq} for all $t>0.$
    \item[3)] The fact that $\gamma\leq \mu$ is necessary. This follows by translation argument. See also \cite{DenapoliDrelichman}. We take, in \eqref{99Lebegq},
$t=1,\; u=G_1(\cdot-\tau x_0),\; \tau>0,\; x_0\in\R^N,|x_0|=1.$  In fact, we have that
$$e^{\Delta}G_1(\cdot-\tau x_0)=G_1\star G_1(\cdot-\tau x_0)=G_2(\cdot-\tau x_0).$$ On the other hand,
$$\left \||\cdot|^{\gamma}G_2(\cdot-\tau x_0)\right\|_{{q_2}}=\left \||\cdot+\tau x_0|^{\gamma}G_2\right\|_{{q_2}}=\tau^{\gamma}\left \|\left|{\cdot\over \tau}+ x_0\right|^{\gamma}G_2\right\|_{{q_2}}$$ and
$$\left \||\cdot|^{\mu}G_1(\cdot-\tau x_0)\right\|_{{q_1}}=\tau^{\mu}\left \|\left|{\cdot\over \tau}+ x_0\right|^{\mu}G_1\right\|_{{q_1}}.$$ Hence, \eqref{99Lebegq}  reads,
$$\tau^{-(\mu-\gamma)}\left \|\left|{\cdot\over \tau}+ x_0\right|^{\gamma}G_2\right\|_{{q_2}}\leq C \left \|\left|{\cdot\over \tau}+ x_0\right|^{\mu}G_1\right\|_{{q_1}}.$$
Then we let $\tau\to\infty,$ since $\left \|\left|{\cdot\over \tau}+ x_0\right|^{\mu}G_1\right\|_{{q_1}}\to \left \|G_1\right\|_{{q_1}} <\infty$ and $\left \|\left|{\cdot\over \tau}+ x_0\right|^{\gamma}G_2\right\|_{{q_2}}\to \left \|G_2\right\|_{{q_2}}>0,$  to deduce that $\gamma\leq \mu$  if $q_2,\; q_1<\infty.$
\item[4)] Our estimate is different from that of \cite{Tsutsui,CIT} since we do not require $q_1\leq q_2$ if $\gamma<\mu.$ In fact, since the condition $\gamma\leq \mu,$ is necessary by the above remark, all that we require  is that the power of $t$  in  \eqref{99Lebegq} is negative.  \end{itemize}}
\end{rem}
To prove Proposition \ref{smoothingeffectorl-lebeg}, we establish the following estimates for the heat kernel in weighted Lorentz spaces{. Since the cases $0=\gamma<\mu<N,$ (\cite{BTW}) and $0<\gamma=\mu<N$ (\cite{Tsutsui,CIT}) are known, we only give the proof for $0<\gamma<\mu<N.$}
\begin{prop}
 \label{smoothingeffectorl-orl}
 Let $N\geq 1,$ $0{<}\gamma{<} \mu<N,$ $1\leq q\leq \infty,$ $q_1\in(1,\infty]$ and $q_2\in(1,\infty]$  satisfy $$0\leq\frac{1}{q_2}{<}    \frac{\mu-\gamma}{N}+\frac{1}{q_1}{<}  \frac{\mu}{N}+\frac{1}{q_1}<1.$$
Then  there exists a  constant $C>0$ depending on $N,\gamma,\mu,\; q,\;q_1$ and $q_2$ such that
\begin{equation}\label{99Lorentzq}
\left \||.|^{\gamma}e^{t \Delta}u\right\|_{L^{q_2,q}}\leq C t^{-\frac{N}{2}\left(\frac{1}{q_1}-\frac{1}{q_2}\right)-\frac{\mu-\gamma}{2}}  \left\||.|^{\mu}u\right\|_{L^{q_1,\infty}},\; t>0, \; |.|^{\mu}u \in L^{q_1,\infty},
\end{equation}
with if $q_2=\infty$ then $q=\infty.$
\end{prop}
\begin{rem}
 {$\,$\rm A Young's inequality is proved in \cite[Theorem 3.1, p. 201]{Kerman} for weighted Lebesgue  spaces where it is assumed also $q_1,q_2<\infty.$  We do not use this here and we provide a simpler proof for our case as a  convolution with a Gaussian. See \cite{Tayachi} for \eqref{99Lorentzq} with $\gamma=0<\mu<N,$  $0\leq\frac{1}{q_2}{<} \frac{\mu}{N}+\frac{1}{q_1}<1.$}
 \end{rem}
\begin{proof}[Proof of Proposition \ref{smoothingeffectorl-orl}] From the embedding $L^{q_2,1} \hookrightarrow L^{q_2,q},\; q\geq 1,\, q_2<\infty,$ it is sufficient to  give the proof for $q=1.$ Since {$\gamma> 0,$} then  by the inequality
$|x|^\gamma\leq C(|y|^\gamma+|x-y|^\gamma),$ we write
\begin{equation}
||.|^{\gamma}e^{t \Delta}u|=||.|^{\gamma}(G_t\star u)| \leq  C G_t\star (|.|^{\gamma}|u|)+C(|.|^{\gamma}G_t)\star |u|,\; t>0\label{keyinequality}.
\end{equation}

Let $\gamma< \mu<N,$  $q_1\in (1,\infty)$ and $q_2\in (1,\infty)$ be such that
$${1\over q_2}<{\mu-\gamma\over N}+{1\over q_1}<{\mu\over N}+{1\over q_1}<1.$$ Set 
$${1\over p_1}=1+{1\over q_2}-{1\over p_2},$$ with $${1\over p_2}= {\mu-\gamma\over N}+{1\over q_1}.$$
Since $\gamma<N,$ then $p_1\in {(1},\infty)$ and satisfies
$${\gamma\over N}<1-\left({\mu-\gamma\over N}+{1\over q_1}\right)<{1\over p_1} {<1}.$$
Let us introduce the numbers ${\tilde{p}}_1,\; {\tilde{p}}_2$ defined by
$${1\over {\tilde{p}}_1}={1\over p_1}-{\gamma\over N},\; {1\over {\tilde{p_2}}}={\mu\over N}+{1\over {q_1}}.$$
We have
$${1\over {\tilde{p}}_1}=1+{1\over q_2}-{1\over {\tilde{p}}_2},$$ 
Since $0<\gamma<\mu,$ and by the conditions on $q_1,\; q_2,$ we have that
$$0<{1\over p_2}<{1\over {\tilde{p}}_2}<1,\; 0<{1\over {\tilde{p}}_1}<1,\; 0<{1\over p_1}+{1\over p_2}-1={1\over p_1}-1+{1\over p_2}{<} {1\over p_2}<1.$$

Using generalized Young's inequality, see \cite[Theorem  2.6, p. 137]{ONeil} or \cite{Grafakos,Lemarie},    we deduce that
\begin{eqnarray}
\||.|^{\gamma}e^{t \Delta}u\|_{L^{q_2,1}} &\leq & C\|G_t\star|.|^{\gamma}|u|\|_{L^{q_2,1}}+C\||.|^{\gamma}G_t\star |u|\|_{L^{q_2,1}}\nonumber\\ &\leq&
C\|G_t\|_{L^{p_1,1}}\||.|^{\gamma}|u|\|_{L^{p_2,\infty}}+C\||.|^{\gamma}G_t\|_{L^{{\tilde{p_1}},1}}\||u|\|_{L^{{\tilde{p_2}},\infty}}\nonumber\\ \label{estimate12} &:=& CI_1+CI_2,\end{eqnarray}
where
$$ 1+ {1\over q_2}={1\over p_1}+{1\over p_2}={1\over {\tilde{p_1}}}+{1\over {\tilde{p_2}}},\; 1<q_2,p_2,{\tilde{p_2}}<\infty,\; 1<p_1,\;{\tilde{p_1}}<\infty.$$

We first estimate $I_1.$ Using the generalized H\"older  inequality, see \cite[Theorem 3.4, p. 141]{ONeil} or \cite{Lemarie,Grafakos}),
we get
$$I_1\leq C\|G_t\|_{L^{p_1,1}}\||.|^{-(\mu-\gamma)}\|_{L^{{N\over \mu-\gamma},\infty}}\||.|^{\mu}u\|_{L^{q_1,\infty}},\;0<\mu-\gamma<N,\; {1\over p_2}={\mu-\gamma\over N}+{1\over q_1}<1.$$ 
Since $G_t(x)=t^{-\frac{N}{2}}G_1(x/\sqrt{t})=t^{-\frac{N}{2}}(4\pi)^{-\frac{N}{2}}e^{-{|x|^2\over 4t}}\in L^{p_1,1},$ we deduce from \cite{Grafakos} that
$$\|G_t\|_{L^{p_1,1}}=t^{-\frac{N}{2}}\|D_{1/\sqrt{t}}G_1\|_{L^{p_1,1}}=t^{-\frac{N}{2}}t^{\frac{N}{2p_1}}\|G_1\|_{L^{p_1,1}}=Ct^{-\frac{N}{2}\left(1-\frac{1}{p_1}
\right)},$$ with
$$1-{1\over p_1}={1\over p_2}-{1\over q_2}={1\over q_1}-{1\over q_2}+{\mu-\gamma\over N}>0.$$
Then, we deduce that
\begin{equation}
\label{estimationI1}
I_1\leq C t^{-{N\over 2}\left({1\over q_1}-{1\over q_2}\right)-{\mu-\gamma\over 2}}\||.|^{\mu}u\|_{L^{q_1,\infty}}.
\end{equation}

We now estimate $I_2.$ Using the generalized H\"older  inequality,
we get
$$I_2\leq C\||.|^{\gamma}G_t\|_{L^{{\tilde{p_1}},1}}\||.|^{-\mu}\|_{L^{{N\over \mu},\infty}}\||.|^{\mu}u\|_{L^{{q_1},\infty}},$$ $$\;0<\mu<N,\; {1\over {\tilde{p_2}}}={\mu\over N}+{1\over {q_1}}<1{\color{blue}.}$$

Since $|x|^\gamma G_t(x)=t^{-\frac{N}{2}+\frac{\gamma}{2}}|x/\sqrt{t}|^\gamma G_1(x/\sqrt{t})=t^{-\frac{N}{2}}(4\pi)^{-\frac{N}{2}}|x|^\gamma e^{-{|x|^2\over 4t}}\in L^{ {\tilde{p_1}},1},$ we deduce from \cite{Grafakos} that
$$\||.|^{\gamma}G_t\|_{L^{{\tilde{p_1}},1}}=t^{-\frac{N}{2}+\frac{\gamma}{2}}\|D_{1/\sqrt{t}}(|.|^{\gamma}G_1)\|_{L^{{\tilde{p_1}},1}}=
t^{-\frac{N}{2}+\frac{\gamma}{2}}t^{\frac{N}{2{\tilde{p_1}}}}\||.|^{\gamma}G_1\|_{L^{{\tilde{p_1}},1}}=Ct^{-\frac{N}{2}\left(1-\frac{1}{p_1}
\right)}.$$
Then, we deduce that
\begin{equation}
\label{estimationI2}
I_2\leq C t^{-{N\over 2}\left({1\over q_1}-{1\over q_2}\right)-{\mu-\gamma\over 2}}\||.|^{\mu}u\|_{L^{q_1,\infty}}.
\end{equation}
Plugging \eqref{estimationI1} and \eqref{estimationI2} in \eqref{estimate12} we get \eqref{99Lorentzq}.

If $q_2\in (1,\infty)$ and $q_1=\infty,$ hence  $q_2>N/(\mu-\gamma)$, { the above calculations for estimating $I_1,\; I_2$ hold using the generalized H\"older  inequality in \cite[Theorem 3.4, p. 141]{ONeil} (see also \cite[Proposition 2.3 a), p. 19]{Lemarie}).} 

If $q_2=q=\infty,$ the proof follows by using the generalized Young inequality, \cite[Theorem 3.6, p. 141]{ONeil} (see also \cite[Proposition 2.4 b), p. 20]{Lemarie}) as follows
$$\||.|^{\gamma}e^{t \Delta}u\|_{\infty} \leq
C\|G_t\|_{L^{p_1,1}}\||.|^{\gamma}|u|\|_{L^{p_2,\infty}}+C\||.|^{\gamma}G_t\|_{L^{{\tilde{p_1}},1}}\||u|\|_{L^{{\tilde{p_2}},\infty}},$$
with $$1-\left({\mu-\gamma\over N}+{1\over q_1}\right)={1\over p_1}\in (\gamma/N,1)$$ and by similar calculations as above. This completes the proof.\end{proof}

We now give the proof of Proposition \ref{smoothingeffectorl-lebeg}.
\begin{proof}[Proof of Proposition \ref{smoothingeffectorl-lebeg}] The proof follows by taking $q=q_2$ in \eqref{99Lorentzq} and using the fact that
$$ \left\||.|^{\mu}u\right\|_{L^{q_1,\infty}}\leq C \left\||.|^{\mu}u\right\|_{L^{q_1,q_1}}=C\left\||.|^{\mu}u\right\|_{L^{q_1}}.$$
\end{proof}
\section{Lower bounds for slowly decaying initial data}
\setcounter{equation}{0}
\label{NLHweight}
In this section we apply Proposition \ref{smoothingeffectorl-lebeg} in order to show local well-posedness in weighted Lebesgue spaces for the nonlinear heat equation \eqref{NLHint}. This allows us to obtain more precise estimates for the lower bound of the life-span in relation with the weight.
For $\gamma\geq 0,$ $1\leq q\leq \infty,$ we consider the weighted Lebesgue space
$$L^q_\gamma(\R^N)=\{f : \R^N\to \R,\mbox{mesurable}, |\cdot|^\gamma f\in L^q(\R^N)\}.$$ Endowed with the  norm
 $$\|f\|_{L^q_\gamma}:=\||\cdot|^\gamma f\|_{L^q},$$ $L^q_\gamma(\R^N)$ is a Banach space. Clearly, if $0<\gamma<N,$ $1\leq q\leq \infty,$ and ${1\over q}+{\gamma\over N}<1,$  using the H\"older inequality, we have $L^q_\gamma(\R^N)\subset \Sd.$ Also, for $u_0\in L^q_\gamma(\R^N),\; 0<\gamma<N,\; {N\over N-\gamma}<q<\infty,$ we have $\lim_{t\to 0}\|e^{t \Delta}u_0-u_0\|_{L^q_\gamma(\R^N)}=0.$ This follows as for the standard $L^q(\R^N)$ case, that is $\gamma=0.$

We are interested in the local well-posedness for the nonlinear heat equation \eqref{NLHint} in $L^q_\gamma(\R^N).$ We consider initial data $u_0\in L^q_\gamma(\R^N)$ where $q,\; \gamma$ satisfy $$0<\gamma<N,\; \gamma<{2\over \alpha}.$$ $${1\over q}+{\gamma\over N}<1,\; {N\alpha\over 2q}+{\alpha\gamma\over 2}<1.$$  The critical exponent in the weighted Lebesgue space $L^q_\gamma(\R^N)$ is given by \begin{equation}
\label{qcritiquegamma}
q_c(\gamma)={N\alpha\over 2-\gamma\alpha}.
\end{equation}
The value of the critical exponent $q_c(\gamma)$ can be explained by scaling argument. In fact, if $u$ is a solution of the equation \eqref{NLH}, with $\Omega=\R^N,$ then for any $\mu>0,$ $u_\mu$ is also a solution of  \eqref{NLH}, where  $$u_{\mu}(t,x)=\mu^{2\over \alpha}u(\mu^2t,\mu x).$$
We have  $\|u_\mu(t)\|_{L^q_\gamma}=\mu^{{2\over \alpha}-{\gamma}-{N\over q}}\|u(t)\|_{L^q_\gamma},$
and on initial data $u(0)=u_0$ we have $$\|\mu^{2\over \alpha}u_0(\mu\cdot)\|_{L^q_\gamma}=\mu^{{2\over \alpha}-{\gamma}
-{N\over q}}\|u_0\|_{L^q_\gamma}.$$ The only weighted Lebesgue exponent (obviously if its exponent is greater than $1$) for which the norm is invariant under these dilations is   $${N\over q_c(\gamma)}={2\over \alpha}-{\gamma}.$$ Hence $q_c(\gamma)$ is given by \eqref{qcritiquegamma}.
We have the following local well-posedness result.
\begin{To}[Local well-posedness in $L^q_\gamma$]
\label{th3} Let $N\geq 1$ be an integer, $\alpha>0$ and $\gamma$
such that
\begin{equation}
\label{e03} 0<\gamma<N,\; \gamma<2/\alpha. \end{equation} Let $q_c(\gamma)$ be given by
\eqref{qcritiquegamma}. Then we have the following.
\begin{itemize}
\item[(i)] If $\gamma(\alpha+1)<N$ and $q$ is such that
$$q >\frac{N(\alpha+1)}{N-\gamma(\alpha+1)}, \quad q> q_c(\gamma)\quad \mbox{and } \quad q\leq\infty,$$
then equation \eqref{NLHint} is locally well-posed in $L^q_\gamma(\RR^N).$ More precisely, given $u_0\in L^q_\gamma(\RR^N)$, then there exist $T>0$ and a unique solution $u\in   C\big([0, T];L^q_\gamma(\mathbb R^N))$ of
\eqref{NLHint} (we replace $[0,T]$ by $(0,T]$ if $q=\infty$ {  and  $u$ satisfies  $\lim_{t\to 0}\|u(t)-e^{t \Delta}u_0\|_{L^\infty_\gamma(\R^N)}=0$}). Moreover,
 $u$ can be extended to a maximal interval $[0, T_{\max})$ such that
either $T_{\max}= \infty$ or  $T_{\max}< \infty$ and
$\displaystyle\lim_{t\to T_{\max}}\|u(t)\|_{L^q_\gamma} = \infty.$
\item[(ii)] Assume that $q>q_c(\gamma)$ with ${N\over N-\gamma} < q\leq\infty$. It follows that equation \eqref{NLHint} is locally well-posed in
$L^q_\gamma(\RR^N)$ as in part (i) except that uniqueness is guaranteed only among  functions $u\in  C\big([0, T];L_\gamma^q(\RR^N)) $ which also verify
 $t^{{N\over 2}({1\over q}-{1\over r})+{\nu\alpha\over 2}}\|u(t)\|_{L^r_\nu}$ is bounded on $(0,T]$, where $r=(\alpha+1)q,\; \nu(\alpha+1)=\gamma,$ (we replace $[0,T]$ by $(0,T]$ if $q=\infty$ {  and  $u$ satisfies  $\lim_{t\to 0}\|u(t)-e^{t \Delta}u_0\|_{L^\infty_\gamma(\R^N)}=0$}).
  Moreover,
 $u$ can be extended to a maximal interval $[0, T_{\max})$ such that
either $T_{\max}= \infty$ or  $T_{\max}< \infty$ and
$\displaystyle\lim_{t\to T_{\max}}\|u(t)\|_{L^q_\gamma}= \infty.$  Furthermore,
\begin{equation}
\label{VitesseInf} \|u(t)\|_{L^q_\gamma}\geq C\left(T_{\max}-t\right)^{{N\over
2q}-{2-\gamma\alpha\over 2\alpha}},\; \forall\; t\in [0, T_{\max}),
\end{equation}
where $C$ is a positive constant.
\end{itemize}
\end{To}
\begin{proof}
(i) Let us define the maps
\begin{equation*}
    K_t(u)=e^{t \Delta}\left(|u|^\alpha u\right),\; t>0.
\end{equation*}
Using the following inequality, which follows by the H\"older inequality,
\begin{equation}
\label{ineqwithHolder}
\left\||.|^{\nu(\alpha+1)}\big( |u|^ \alpha
 u-|v|^\alpha v\big)\right\|_{\frac{p}{\alpha+1}}\leq C \left(\|u\|^\alpha_{L^p_\nu}+\|v\|^\alpha_{L^p_\nu} \right)  \left\|
 u- v\right\|_{L^{p}_\nu},\; p\geq \alpha+1,\; \nu\geq 0,
\end{equation}
and Proposition \ref{smoothingeffectorl-lebeg} that is $e^{t \Delta}: L^{q\over \alpha+1}_{(\alpha+1)\gamma}\rightarrow L^q_\gamma$ is bounded for each $t>0,$ we have that  $K_t: L^q_\gamma\longrightarrow L^q_\gamma$ is locally Lipschitz with
\begin{eqnarray*}
 \|K_t(u)-K_t(v)\|_{L^q_\gamma} &\leq & Ct^{-{N\over 2}({\alpha+1\over q}-{1\over q})-{\alpha\gamma\over 2}}\left\||u|^\alpha u-|v|^\alpha v\right\|_{L^{q\over \alpha+1}_{(\alpha+1)\gamma}}\\  &\leq & C t^{-\frac{N \alpha}{2 q}-{\alpha\gamma\over 2}} (\|u\|_{L^q_\gamma}^\alpha+\|v\|_{L^q_\gamma}^\alpha)\|u-v\|_{L^q_\gamma} \\
 &\leq & 2 CM^\alpha t^{-\frac{N \alpha}{2 q}-{\alpha\gamma\over 2}} \|u-v\|_{L^q_\gamma},
\end{eqnarray*}
for $\|u\|_{L^q_\gamma}\leq M \mbox{ and } \;\|v\|_{L^q_\gamma}\leq M.$
We have also,  that $t^{-\frac{N \alpha}{2 q}-{\alpha\gamma\over 2}}\in L^1_{loc}(0,\infty),$ since $q>q_c(\gamma).$  Obviously $t\mapsto \|K_t(0)\|_\infty=0
\in L^1_{loc}(0,\infty),$ also $e^{s \Delta}K_t=K_{t+s}$ for $s,\;t>0.$ Then the proof follows by \cite[Theorem 1, p. 279]{W2}.

(ii)  We consider $r$ and $\nu>0$  such that $\nu<\gamma,\; \nu(\alpha+1)<N,$ $r>q.$ Hence we have
  $${1\over r}<{\alpha+1\over r}
  +{\nu\alpha\over N}<{\alpha+1\over r}+ {\nu(\alpha+1)\over N}<1,\; {1\over r}<{1\over q}+{\gamma-\nu\over N}<{1\over q}+{\gamma\over N}<1.
  $$
  The choice of $r,\; \nu$ is to guaranties that the maps $e^{t \Delta}: L^{r\over \alpha+1}_{(\alpha+1)\nu}\rightarrow L^r_\nu$ and 
   $e^{t \Delta}: L^{q}_{\gamma}\rightarrow L^r_\nu$    are bounded  so that we may apply Proposition \ref{smoothingeffectorl-lebeg}. In order that $e^{t \Delta}: L^{r\over \alpha+1}_{(\alpha+1)\nu}\rightarrow L^q_\gamma$ is bounded,  we choose for simplicity,
  $$r=(\alpha+1)q,\; \nu(\alpha+1)=\gamma,$$
 (If $q=\infty$ we have $r=\infty$), and we may apply \cite[Lemma 2.1]{CIT} to get that $e^{t \Delta}: L^{r\over \alpha+1}_{(\alpha+1)\nu}=L^q_\gamma \rightarrow L^q_\gamma$ is bounded. With this choice, the conditions on $r$ and $\nu$ are satisfied, since  $$\frac{1}{q}+{\gamma\over N}<1.$$ Define
 \begin{equation}
 \label{betagamma}
  \beta(\nu) = \frac{N}{2q} - \frac{N}{2r}+{\alpha\nu\over 2}.
    \end{equation}
    We choose $K>0,\; T>0,\; M>0$ such that
\begin{equation}
\label{equKMT}
K+CM^{\alpha+1}T^{1-\frac{N\alpha}{2q}-{\gamma\alpha\over 2}}\leq M,
\end{equation}
where $C$ is a positive constant.  We will show that  there exists a
unique solution $u$ of  (\ref{NLHint}) such that $u\in C\left([0,T];L^q_\gamma(\R^N)\right)\cap C\left((0,T];L^r_\nu(\R^N)\right)$  with $$\|u\|=\max\left[\sup_{t\in[0, T]}\|u(t)\|_{L^q_\gamma},\;\sup_{t\in(0, T]}t^{\beta(\nu)}\|u(t)\|_{L^r_\nu}\right]\leq M.$$

The proof is based on a contraction mapping argument in the set
   $$ Y_{M,T}^{q,\gamma} = \{u \in C\left([0,T];L^q_\gamma(\R^N)\right)\cap C((0,T]; L^r_\nu) :
\|u\|\leq M\}.$$  Endowed with the metric $d(u,v)=\|u-v\|,$  $Y_{M,T}^{q,\gamma}$ is a nonempty complete metric space.
We note that for  $u_0\in L^q_\gamma$ we have
$$ \|{\rm e}^{t\Delta}u_0\|_{L^r_\nu} \le Ct^{-{N\over 2}({1\over q}-{1\over r})-{\gamma-\nu\over 2}}\|u_0\|_{L^q_\gamma}=Ct^{-{N\over 2}({1\over q}-{1\over r})-{\nu\alpha\over 2}}\|u_0\|_{L^q_\gamma}=Ct^{-\beta(\nu)}\|u_0\|_{L^q_\gamma}.$$

We will show that  $\mathcal{F}_{u_0}$ defined in \eqref{iteratn} is a strict contraction on $Y_{M,T}^{q,\gamma}$.  The condition on the initial data $\|u_0\|_{L^q_\gamma}\leq K$
will implies that $t^\beta \|{\rm e}^{t\Delta}u_0\|_{L^r_\nu}\leq K.$ We have
 \begin{eqnarray*}
t^{\beta(\nu)} \|\F_{u_0}u(t)\|_{L^r_\nu} &\le& t^{\beta(\nu)} \|{\rm e}^{t\Delta}u_0\|_{L^r_\nu} + t^{\beta(\nu)} \int_0^t \|{\rm
e}^{(t-\sigma)\Delta} \big[|u(\sigma)|^\alpha u(\sigma)\big]\|_{L^r_\nu} d\sigma\\
&\le& K +C
 t^{\beta(\nu)} \int_0^t (t-\sigma)^{-\frac{N\alpha}{2r}-{\nu(\alpha+1)-\nu\over 2}} \||\cdot|^{\nu(\alpha+1)}|u(\sigma)|^\alpha u(\sigma)\|_{r/(\alpha+1)} d\sigma\\
 &=& K +C
t^{\beta(\nu)} \int_0^t (t-\sigma)^{-\frac{N\alpha}{2r}-{\nu\alpha\over 2}} \|u(\sigma)\|_{L^r_\nu}^{\alpha+1} d\sigma\\
  &\le& K +
CM^{\alpha+1}t^{\beta(\nu)} \int_0^t (t-\sigma)^{-\frac{N\alpha}{2r}-{\nu\alpha\over 2}} \sigma^{-{\beta(\nu)}(\alpha+1)} d\sigma \\
  &\le& K +
CM^{\alpha+1}t^{1-\frac{N\alpha}{2q}-{\gamma\alpha\over 2}} \int_0^1(1-\sigma)^{-\frac{N\alpha}{2r}-{\nu\alpha\over 2}} \sigma^{-{\beta(\nu)}(\alpha+1)} d\sigma
\\
  &\le& K +
CM^{\alpha+1}T^{1-\frac{N\alpha}{2q}-{\gamma\alpha\over 2}}\int_0^1(1-\sigma)^{-\frac{N\alpha}{2r}-{\nu\alpha\over 2}} \sigma^{-{\beta(\nu)}(\alpha+1)} d\sigma.
\end{eqnarray*}
By the hypotheses and the fact that $q<r$ and $\nu<\gamma$ we have
$$\frac{N\alpha}{2r}+{\nu\alpha\over 2}<\frac{N\alpha}{2q}+{\gamma\alpha\over 2}<1,\; \beta(\nu)(\alpha+1)={N\alpha\over 2q}+{\alpha\gamma\over 2}<1.$$

We estimate in $L^q_\gamma$  as follows,
 \begin{eqnarray*}
 \|\F_{u_0}u(t)\|_{L^q_\gamma} &\le&  \|{\rm e}^{t\Delta}u_0\|_{L^q_\gamma} +  \int_0^t \|{\rm
e}^{(t-\sigma)\Delta} \big[|u(\sigma)|^\alpha u(\sigma)\big]\|_{L^q_\gamma} d\sigma\\
&\le& K +C \int_0^t \||\cdot|^{\nu(\alpha+1)}|u(\sigma)|^\alpha u(\sigma)\|_{r/(\alpha+1)} d\sigma\\
 &=& K +C \int_0^t  \|u(\sigma)\|_{L^r_\nu}^{\alpha+1} d\sigma\\
  &\le& K +
CM^{\alpha+1} \int_0^t \sigma^{-\beta(\nu)(\alpha+1)} d\sigma \\
  &\le& K +
CM^{\alpha+1}t^{1-\frac{N\alpha}{2q}-{\gamma\alpha\over 2}} \int_0^1\sigma^{-\beta(\nu)(\alpha+1)} d\sigma
\\
  &\le& K +
CM^{\alpha+1}T^{1-\frac{N\alpha}{2q}-{\gamma\alpha\over 2}}\int_0^1 \sigma^{-\beta(\nu)(\alpha+1)} d\sigma.
\\
\end{eqnarray*}
The condition \eqref{equKMT} implies that the space $Y_{M,T}^{q,r}$ is preserved by the iterative operator $\F_{u_0}$.
We show similarly  the contraction.  The proof of the other parts follows as in \cite{BTW}. So we omit the details. This completes the proof of the theorem.
\end{proof}
We note that uniqueness in Part (ii) of  Theorem \ref{th3} holds in  $u\in  C\big([0, T];L_\gamma^q(\RR^N))\cap C\big((0, T];L_\nu^r(\RR^N)).$
This follows by similar argument  as in \cite{BenSlimene}. We will not belabor this point further.

\begin{proof}[Proof of Theorem \ref{lowerLqgamma}] Consider  $u_0 = \lambda\varphi$, where $\lambda > 0$ and $\varphi \in L^q_\gamma.$
If  $T_{\max}(\lambda\varphi) <\infty$, it is impossible to carry out the fixed point argument
on the interval $[0,T_{\max}(\lambda\varphi)]$ with initial value $u_0 = \lambda\varphi$.
Hence, by \eqref{equKMT}
\begin{equation*}
K
+ CT_{\max}(\lambda\varphi)^{1 - \frac{N\alpha}{2q}-{\gamma\alpha\over 2}}M^{\alpha + 1} > M,
\end{equation*}
for all $M > K$.
 Letting $K=\|u_0\|_{L^q_\gamma}=\lambda\|\varphi\|_{L^q_\gamma}$, so that
\begin{equation*}
\lambda\|\varphi\|_{L^q_\gamma} + CT_{\max}(\lambda\varphi)^{1 - \frac{N\alpha}{2q}-{\gamma\alpha\over 2}}M^{\alpha + 1} > M,
\end{equation*}
for all $M > \lambda\|\varphi\|_{L^q_\gamma}$.  In particular, if we set $M = 2\lambda\|\varphi\|_{L^q_\gamma}$, this gives
\begin{equation*}
CT_{\max}(\lambda\varphi)^{1 - \frac{N\alpha}{2q}-{\gamma\alpha\over 2}}[\lambda\|\varphi\|_{L^q_\gamma}]^\alpha > 1.
\end{equation*}
Thus we have proved Theorem \ref{lowerLqgamma}.
\end{proof}
\begin{rem}{\rm If $u_0\in L^q_\gamma$ with $\gamma>0,\; q\leq\infty$ are as in Theorem \ref{th3} then writing
\begin{eqnarray*}|u_0|&=&|u_0\mathbf{1}_{\{|x|\leq 1\}}+u_0 \mathbf{1}_{\{|x|> 1\}}|\\&\leq&|x|^{-\gamma}\mathbf{1}_{\{|x|\leq 1\}}\left(|x|^{\gamma}|u_0|\mathbf{1}_{\{|x|\leq 1\}}\right)+|x|^{-\gamma}\mathbf{1}_{\{|x|> 1\}}\left(|x|^{\gamma}|u_0|\right)\\&\leq&|x|^{-\gamma}\mathbf{1}_{\{|x|\leq 1\}}\left(|x|^{\gamma}|u_0|\mathbf{1}_{\{|x|\leq 1\}}\right)+|x|^{\gamma}|u_0|,
 \end{eqnarray*}
 where $\mathbf{1}_A$ is the indicator function of a subset $A$ of $\R^N,$ we see by the H\"older inequality   that $u_0\in L^r+L^s$ with ${1\over r}={1\over q}+{\gamma\over N},\; s=q$ if $q<\infty$ and $r\geq 1,\; {N\alpha\over 2}<r<{N\over \gamma}<s<\infty$ if $q=\infty$. Then the local well-posedness is proved in \cite[Theorem 2.8]{D}. The fixed point argument used in \cite{D} seems not to give an explicit lower bound  estimate of life span, as $T<1,$ the minimal local existence time, is required in the proof there and the constants in particular in \cite[Inequality (2.10)]{D}, seems to depend on $T$.}
\end{rem}

The construction of solutions to \eqref{NLHint}  with initial data in the intersection of two metric spaces follows by well-known argument. See also the proof of \cite[Proposition 3.2, p. 126]{BTW}. We have  the following result for the existence time of the  maximal solution.
\begin{prop}
\label{C0} Let $N\geq 1$ be an integer, $\alpha>0$ and 
$ 0<\gamma<N,\; \gamma<2/\alpha.$ Let $q_c(\gamma)$ be given by
\eqref{qcritiquegamma}.  Let $q>q_c(\gamma),$ ${N\over N-\gamma} < q<\infty$ and $T_{\max}(\varphi,L^q_\gamma)$ denotes the existence time of the  maximal solution of \eqref{NLHint} with initial data $\varphi\in L^q_\gamma.$  Then the following hold.
\begin{itemize}
\item[(i)] $u(t)\in C_0(\R^N)\cap L^\infty_\gamma$ for $t\in \left(0,T_{\max}(\varphi,L^q_\gamma)\right).$ 
\item[(ii)] If $\varphi\in L^q_\gamma\cap C_0(\R^N)$ then  $T_{\max}(\varphi,L^q_\gamma)=T_{\max}(\varphi,C_0(\R^N)),$ the existence time of the  maximal solution of \eqref{NLHint} with initial data $\varphi\in C_0(\R^N).$
\item[(iii)] If $\varphi\in L^q_\gamma\cap L^p_\gamma$ with $p>q_c(\gamma),$ ${N\over N-\gamma} < p<\infty$ then  $T_{\max}(\varphi,L^q_\gamma)=T_{\max}(\varphi,L^p_\gamma),$ the existence time of the  maximal solution of \eqref{NLHint} with initial data $\varphi\in L^p_\gamma.$
\item[(iv)] If $\varphi\in L^q_\gamma\cap L^p$ with $q_c<p \leq \infty,$ then  $T_{\max}(\varphi,L^q_\gamma)=T_{\max}(\varphi,L^p),$ the existence time of the  maximal solution of \eqref{NLHint} with initial data $\varphi\in L^p.$

\end{itemize}
\end{prop}
\begin{proof}
(i) Let  $\varphi\in L^q_\gamma(\RR^N),$ $q>q_c(\gamma)$ and $ q>{N\over N-\gamma}.$ Let $r=(\alpha+1)q,\; \nu(\alpha+1)=\gamma$ and $\beta(\nu)$ be given by \eqref{betagamma}. Let $p$ be such that $r<p\leq \infty.$ Hence $p>q$ and
$$0\leq {1\over p}<{\alpha+1\over r}<{\gamma\over N}+{\alpha+1\over r}<1,\; {1\over p}<{1\over q}<{\gamma\over N}+{1\over q}<1.$$ For $0<T<T_{\max}(\varphi,L^q_\gamma),$ we have
\begin{eqnarray*}
    \|u(t)\|_{L^p_\gamma}&\leq& \|e^{t \Delta}\varphi\|_{{L^p_\gamma}}+  C \displaystyle\int_0^t (t-\sigma)^{-\frac{N }{2 }(\frac{\alpha+1}{r}-\frac{1}{p})} \|u(\sigma)\|^{\alpha+1}_{L^r_\nu} d\sigma\\
    &\leq&Ct^{-\frac{N}{2}(\frac{1}{q}-\frac{1}{p})}\|\varphi\|_{L^q_\gamma}+  C  t^{1-\frac{N }{2 }(\frac{\alpha+1}{r}-\frac{1}{p})-{\beta(\nu)(\alpha+1)}} \displaystyle\sup_{s\in(0,T]}\left(s^{\beta(\nu)(\alpha+1)}\|u(s)\|_{L^r_\nu}^{\alpha+1}\right)\times
    \\&&\int_0^1 (1-\sigma)^{-\frac{N }{2 }(\frac{\alpha+1}{r}-\frac{1}{p})} \sigma^{-\beta(\nu)(\alpha+1)}d\sigma\\
    &\leq&C t^{-\frac{N}{2}(\frac{1}{q}-\frac{1}{p})}\|\varphi\|_{L^q_\gamma}+   M^{\alpha+1}Ct^{1-\frac{N }{2 }(\frac{\alpha+1}{q}-\frac{1}{p})-{\gamma\alpha\over 2}}\int_0^1 (1-\sigma)^{-\frac{N }{2 }(\frac{\alpha+1}{r}-\frac{1}{p})} \sigma^{-\beta(\nu)(\alpha+1)}d\sigma.
   \end{eqnarray*}
 Since $r>q>q_c(\gamma),$  it follows that if $$\frac{\alpha+1}{r}-\frac{2}{N}<\frac{1}{p}<\frac{1}{r},$$
 then $u(t)$ is in $L^p_\gamma$ for all $t\in \big(0,T_{\max}(\varphi, L^q_\gamma)\big).$
 The result for general $p>q$ follows by iteration. Hence $u(t)$ is in $L^\infty_\gamma,$ for $t\in \big(0,T_{\max}(\varphi, L^q_\gamma)\big).$ Then $u(t)\in L^r+L^s$ for $r\geq 1,\; {N\alpha\over 2}<r<{N\over \gamma}<s<\infty.$ Hence by \cite[Theorem 2.8]{D} { $u(t)\in L^p$ for $s\leq p\leq\infty.$} Then it follows that $u(t)\in C_0(\R^N),$ for $t\in \big(0,T_{\max}(\varphi, L^q_\gamma)\big)$.
 
 (ii) By (i) we have $T_{\max}(\varphi,L^q_\gamma)\leq T_{\max}(\varphi,C_0(\R^N)).$  Using \eqref{NLHint}, we have
 \begin{eqnarray*}
 \|u(t)\|_{L^q_\gamma}&\leq& \|e^{t \Delta}\varphi\|_{{L^q_\gamma}}+  C \displaystyle\int_0^t \||u(\sigma)|^{\alpha}u(\sigma)\|_{L^q_\gamma} d\sigma\\
 &\leq&C \|\varphi\|_{{L^q_\gamma}}+  C \displaystyle\int_0^t \|u(\sigma)\|^{\alpha}_\infty \|u(\sigma)\|_{L^q_\gamma}d\sigma
 \end{eqnarray*}
 By Gronwall's inequality, we get
 $$\|u(t)\|_{L^q_\gamma}\leq C\|\varphi\|_{{L^q_\gamma}}e^{C \int_0^t \|u(\sigma)\|^{\alpha}_\infty d\sigma}.$$
Hence $u$ can not blow up in $L^q_\gamma$ before it blows up in $C_0(\R^N)$. That is  $T_{\max}(\varphi,C_0(\R^N))\leq T_{\max}(\varphi,L^q_\gamma).$  
 
 (iii) Let $\varepsilon\in (0, \min(T_{\max}(\varphi,L^q_\gamma), T_{\max}(\varphi,L^p_\gamma))).$ By (i) we have $u(\varepsilon)\in C_0(\R^N).$ Using (ii) we  have
 $$T_{\max}(u(\varepsilon),L^q_\gamma)=T_{\max}(u(\varepsilon),C_0(\R^N))=T_{\max}(u(\varepsilon),L^p_\gamma).$$ That is $T_{\max}(\varphi,L^q_\gamma)-\varepsilon=T_{\max}(\varphi,L^p_\gamma)-\varepsilon,$ hence we get the result.
 
 (iv) Follows similarly as (iii). This completes the proof of Proposition  \ref{C0}.
 \end{proof}

\begin{proof}[Proof of Corollary \ref{lowerLqLqgamma}]
 Since $T_{\max}(\varphi,L^q) ,$ the maximal existence time  in $L^q,$ is equal to  $T_{\max}(\varphi,L^q\cap L^q_\gamma)$ the maximal existence time  in $L^q\cap L^q_\gamma,$ we deduce that
\begin{itemize}
\item[1)] If  $N\alpha<2$  we discuss the two cases
\begin{enumerate}
\item[(i)] $\gamma<N$ hence $\gamma< 2/\alpha$ and we have
$${1\over q}+{\gamma\over N}<1,\; {N\alpha\over 2q}+{\gamma\alpha\over 2}={N\alpha\over 2}\left({1\over q}+{\gamma\over N}\right)<{N\alpha\over 2}<1,$$ hence, we apply Theorem  \ref{lowerLqgamma} to get $$ T_{\max}(\lambda \varphi)\geq C \lambda^{-\left({1\over \alpha}-{N\over 2q}-{\gamma\over 2}\right)^{-1}}.$$
    \item[(ii)] $\gamma>N$ then $\varphi\in L^1.$ In fact, we write $\varphi=\varphi \mathbf{1}_{\{|x|\leq 1\}}+\varphi \mathbf{1}_{\{|x|> 1\}}.$  On one hand, since $\varphi\in L^q,\; q>1,$ hence $\varphi \mathbf{1}_{\{|x|\leq 1\}}\in L^1.$ On the other hand, by the H\"older inequality,
\begin{eqnarray*}
\|\varphi \mathbf{1}_{\{|x|> 1\}}\|_1& = &\|\varphi |x|^\gamma|x|^{-\gamma}\mathbf{1}_{\{|x|> 1\}}\|_1\\
&\leq  & \|\varphi |x|^\gamma\|_q\|\||x|^{-\gamma}\mathbf{1}_{\{|x|> 1\}}\|_{q'}\\
&= &\|\varphi\|_{L^q_\gamma}\||x|^{-\gamma}\mathbf{1}_{\{|x|> 1\}}\|_{q'}<\infty,
\end{eqnarray*}
 since $\gamma q'\geq \gamma>N$ and since $\varphi\in L^q_\gamma,$ that is $\varphi \mathbf{1}_{\{|x|> 1\}}\in L^1.$ Hence both results give that $\varphi\in L^1.$     We  may then apply Theorem \ref{ThNLH} in $L^1,$ using $1>{N\alpha\over 2}$ to get $T_{\max}(\lambda\varphi)\geq C \lambda^{-\left({1\over \alpha}-{N\over 2}\right)^{-1}}.$
\end{enumerate}
\item[2)] If $N\alpha >2$ then we assume $\gamma<2/ \alpha $ and we have only one case, $\gamma<N.$ Hence since $q>q_c(\gamma)$ we apply Theorem  \ref{lowerLqgamma}  to get $ T_{\max}(\lambda \varphi)\geq C \lambda^{-\left({1\over \alpha}-{N\over 2q}-{\gamma\over 2}\right)^{-1}}.$
    \end{itemize}
    In the all cases we have
    $ T_{\max}(\lambda \varphi)\geq C \lambda^{-\left({1\over \alpha}-{1\over 2}\min({N\over q}+\gamma,N)\right)^{-1}}.$
    This completes the proof of Corollary \ref{lowerLqLqgamma}.
\end{proof}
\begin{proof}[Proof of Corollary \ref{lowerLpLqgamma}]
 Since $\varphi\in L^p\cap L^q_\gamma,$ then the existence time of the maximal solution  is the same to that in $L^p$ and to that in  $L^q_\gamma.$ By the H\"older inequality $\varphi\in L^r_\tau$ for $ {1\over r}={\theta\over p}+{1-\theta\over q},\; \tau=(1-\theta)\gamma,\; \theta \in [0,1].$ The existence time of the maximal solution is also the same in $L^r_\tau.$  The function
$x\in \left(0,2/(N\alpha)\right)\to \lambda^{-\left(\frac{1}{\alpha} - \frac{N}{2}x\right)^{-1}}$ is increasing for $\lambda\in (0,1)$ and decreasing for $\lambda\in (1,\infty).$ Letting $x=x_\theta={1\over r}+{\tau\over N}={\theta}\left({1\over p}-{1\over q}-\gamma\right)+{1\over q}+{\gamma\over N},$ we have that $\max_{\theta\in[0,1]} x_\theta=\max({1\over q}+{\gamma\over N},{1\over p})=\max(x_0,x_1)$ and $\min_{\theta\in[0,1]} x_\theta=\min({1\over q}+{\gamma\over N},{1\over p})=\min(x_0,x_1).$ 

Using Theorem \ref{lowerLqgamma}  if $\theta=0$ or Theorem \ref{ThNLH} if $\theta=1,$ we have that
$$T_{\max}(\lambda\varphi) \geq C\left(\lambda\|\varphi\|_{L^r_{{\tau}}}\right)^{-\left(\frac{1}{\alpha} - \frac{N}{2}\left[{1\over r}+{\tau\over N}\right]\right)^{-1}}\geq C\left(\lambda\|\varphi\|_{L^p\cap L^q_{{\gamma}}}\right)^{-\left(\frac{1}{\alpha} - \frac{N}{2}\left[{1\over r}+{\tau\over N}\right]\right)^{-1}}.$$

The result follows then by taking in the last inequality $\max_{\theta\in[0,1]} x_\theta,$  for $\lambda\in (0,1)$ and $\min_{\theta\in[0,1]} x_\theta$ for $\lambda\in (1,\infty).$ This completes the proof of the Corollary.
\end{proof}
\begin{example}
{\rm Let $0<\gamma<N$ and $\gamma<2/\alpha.$ Let $\tilde\varphi$ be given by \eqref{phitilde}. Then $\tilde\varphi\in L^p\cap L^\infty_\gamma$ with ${N\alpha\over 2}<p<{N\over \gamma}.$ By Corollary \ref{lowerLpLqgamma} we have
$$T_{\max}(\lambda\varphi) \geq C \begin{cases}{\lambda^{-\left(\frac{1}{\alpha} -\frac{N}{2p} \right)^{-1}}},\; \mbox{if}\quad\, 0<\lambda\leq 1,\\
  {\lambda^{-\left(\frac{1}{\alpha} - \frac{\gamma}{2}\right)^{-1}}},\; \mbox{if}\quad\, \lambda>1.
  \end{cases}$$
}
\end{example}

\section{Upper bounds for nonnegative solutions}
\setcounter{equation}{0}
\label{ubnnsol}
In this section we exploit a well-known necessary condition for the existence
of a nonnegative solution to \eqref{NLHint}.
More precisely, if $u$ is a nonnegative solution of the integral equation \eqref{NLHint} on $(0,T) \times \Omega$
then
\begin{equation}
\label{neccd}
\alpha t ({\rm e}^{t\Delta}u_0)^\alpha \le 1,
\end{equation}
for all $t \in (0,T]$, where $u_0 \ge 0$ can be either a locally integrable function or a positive Borel measure on $\Omega$.
See  \cite[Theorem 1]{W5}.

We let $T_{\max}(u_0)$ denote the maximal existence time of a nonnegative solution of \eqref{NLHint}, and so
 $0 \le T_{\max}(u_0) \le \infty$.  Indeed, there are three possibilities, all of which can be realized:
there is no local nonnegative solution with initial value $u_0$, there is at least one local solution
on some interval $(0,T)$, but no global solution, i.e. on $(0,\infty)$, or there is indeed a global solution.
In the case $u_0 = \lambda\varphi$, then \eqref{neccd} becomes
\begin{equation}
\label{neccdlam}
\alpha \lambda^\alpha t ({\rm e}^{t\Delta}\varphi)^\alpha \le 1,
\end{equation}
for all $t \in (0,T]$.  If $\varphi \ge 0$, $\varphi \not\equiv 0$, this implies that $T_{\max}(\lambda\varphi) < \infty$ for all sufficiently large $\lambda > 0$
and that
\begin{equation}
\label{limintfytmax}
\lim_{\lambda \to \infty}T_{\max}(\lambda\varphi) = 0.
\end{equation}
Indeed, given any $t > 0$, \eqref{neccdlam} can not be true for sufficiently large $\lambda > 0$,
and so $t \ge T_{\max}(\lambda\varphi)$ for sufficiently large $\lambda > 0$. This shows the first statements of Theorems \ref{upperLinfty}, \ref{uppernegpower1} and \ref{uppernegpower1signchanging}.

It is important to realize that $T_{\max}(u_0)$ as just defined, i.e. the maximal existence time of
a nonnegative solution, is not necessarily the same as the maximal existence time of a regular nonnegative
solution.  Indeed, in some cases, a nonnegative solution can be continued after blowup. See \cite{BarasC} for example.  However,
any upper bound on $T_{\max}(u_0)$ is also an upper bound on the maximal existence time of a regular nonnegative
solution.

\begin{prop}
\label{lincd} Let $u_0 \ge 0$ be either a locally integrable function  or a positive
Borel measure on $\Omega$, and let $T_{\max}(u_0)$ denote the maximal existence time of a nonnegative solution of \eqref{NLHint}.
If $0 < T_{\max}(u_0) < \infty$, then
\begin{equation}
\label{neccdmax}
\alpha T_{\max}(u_0) ({\rm e}^{T_{\max}(u_0)\Delta}u_0)^\alpha \le 1.
\end{equation}
In particular, if $u_0 = \lambda\varphi$, then
\begin{equation}
\label{neccdmaxlam}
\alpha \lambda^\alpha T_{\max}(\lambda\varphi) \|{\rm e}^{T_{\max}(\lambda\varphi)\Delta}\varphi\|_\infty^\alpha \le 1.
\end{equation}
\end{prop}

\begin{proof}
Inequality \eqref{neccd} is true for all $0 < t < T_{\max}(u_0)$.  Hence it is true for $t = T_{\max}(u_0)$.
\end{proof}

We now give the proofs of the upper bounds.
\begin{proof}[Proof of Theorem \ref{upperLinfty}]
Since $T_{\max}(\lambda\varphi) \to 0$ as $\lambda \to \infty$, it suffices by \eqref{neccdmaxlam} to observe
that if $\varphi \in L^\infty(\Omega)$, then $\|{\rm e}^{t\Delta}\varphi\|_\infty \to \|\varphi\|_\infty$ as $t \to 0$.
Indeed, $\|{\rm e}^{t\Delta}\varphi\|_\infty \le \|\varphi\|_\infty$ so
$\limsup_{t\to 0}\|{\rm e}^{t\Delta}\varphi\|_\infty \le \|\varphi\|_\infty$.
On the other hand,  ${\rm e}^{t\Delta}\varphi \to \varphi$ weak* as $t \to 0$, so
$\|\varphi\|_\infty \le \liminf_{t\to 0}\|{\rm e}^{t\Delta}\varphi\|_\infty $.
\end{proof}
\begin{proof}[Proof of Theorem \ref{uppernegpower1}] In order to estimate $T_{\max}(\lambda\varphi)$ from above, it suffices to estimate
$T_{\max}(\lambda\tilde\varphi)$ from above, where $\tilde\varphi$ is defined in \eqref{phitilde}.
Indeed, since $0 \le \tilde\varphi(x) \le \varphi(x)$
it follows that $T_{\max}(\lambda\varphi) \le T_{\max}(\lambda\tilde\varphi)$.

To find an upper estimate on $T_{\max}(\lambda\tilde\varphi)$  as $\lambda \to \infty$, it suffices by
Proposition~\ref{lincd} to determine the behavior of $\|{\rm e}^{t\Delta}\tilde\varphi\|_\infty$
as $t \to 0$.  Let  $D_\tau$ be the dilation operator $D_\tau f(x) = f(\tau x)$.  We have
\begin{equation}
\label{gammascale}
\|{\rm e}^{t\Delta}\tilde\varphi\|_\infty = \|D_{\sqrt t}{\rm e}^{t\Delta}\tilde\varphi\|_\infty
=\|{\rm e}^{\Delta}D_{\sqrt t}\tilde\varphi\|_\infty
= t^{-\frac{\gamma}{2}}\|{\rm e}^{\Delta}[t^{\frac{\gamma}{2}}D_{\sqrt t}\tilde\varphi]\|_\infty.
\end{equation}
Since $t^{\frac{\gamma}{2}}D_{\sqrt t}\tilde\varphi \to \omega|\cdot|^{-\gamma}$ in ${\mathcal D'}(\R^N)$
as $t \to 0$, it follows by \cite[Proposition 3.8 (i), page 1123]{CDW2003} that
\begin{equation*}
t^{\frac{\gamma}{2}}\|{\rm e}^{t\Delta}\tilde\varphi\|_\infty \to \|{\rm e}^{\Delta}(\omega|\cdot|^{-\gamma})\|_\infty,
\end{equation*}
as $t \to 0$.
Since by \eqref{limintfytmax}, $T_{\max}(\lambda\varphi) \to 0$ as $\lambda \to \infty$, this along with \eqref{neccdmaxlam} implies
\begin{equation*}
\limsup_{\lambda \to \infty} \lambda^\alpha T_{\max}(\lambda\varphi)^{1 - \frac{\alpha\gamma}{2}}
\le \frac{1}{\alpha\|{\rm e}^{\Delta}(\omega|\cdot|^{-\gamma})\|_\infty^\alpha},
\end{equation*}
which is the desired result.
\end{proof}
\begin{proof}[Proof of Theorem \ref{uppermeas}] Applying \eqref{neccdmaxlam} we see that
\begin{equation*}
\alpha \lambda^\alpha T_{\max}(\lambda\m) \|{\rm e}^{T_{\max}(\lambda\m)\Delta}\m\|_\infty^\alpha \le 1.
\end{equation*}
Furthermore,
\begin{equation*}
\|{\rm e}^{t\Delta}\m\|_\infty = \|D_{\sqrt t}{\rm e}^{t\Delta}\m\|_\infty
=\|{\rm e}^{\Delta}D_{\sqrt t}\m\|_\infty
= t^{-\frac{N}{2}}\|{\rm e}^{\Delta}[t^{\frac{N}{2}}D_{\sqrt t}\m]\|_\infty
\end{equation*}
so that
\begin{equation}
\label{neccdmaxlamdelta}
\alpha \lambda^\alpha T_{\max}(\lambda\m)^{1 - \frac{N\alpha}{2}}
 \|{\rm e}^{\Delta}[T_{\max}(\lambda\m)^\frac{N}{2}D_{\sqrt {T_{\max}(\lambda\m)}}\m]\|_\infty^\alpha \le 1.
\end{equation}
The result follows since $T_{\max}(\lambda\m) \to \infty$
as $\lambda \to 0$ (by continuous dependence or Theorem \ref{lowermeas}) and
 $t^{\frac{N}{2}}D_{\sqrt t} \m \to \|\m\|_\M \delta$ as $t \to \infty$. In fact,  $\mu_t=t^{N/2}D_{\sqrt t} \m$ is the measure defined by
 $$\int_{\R^N}f(x)d\mu_t(x)dx=\int_{\R^N}D_{1/\sqrt t}f(x)d\m(x),\; f\in C_0(\R^N).$$ We have $D_{1/\sqrt t}f\to f(0)$  as $t\to \infty$ $\m$ a.e. Since $\m$ is finite, then by the dominated convergence theorem $\int_{\R^N}f(x)d\mu_t(x)dx\to \|\m\|_\M \delta$ as $t\to \infty$ for every $f$ in $C_0(\R^N).$ Then $t^{\frac{N}{2}}D_{\sqrt t} \m \to \|\m\|_\M \delta$ as $t \to \infty$ in the dual space $(C_0(\R^N))'$. We know that ${\rm e}^{\Delta} : L^1(\R^N)\to C_0(\R^N)$ is a continuous operator then, by duality,  ${\rm e}^{\Delta} : (C_0(\R^N))'\to (L^1(\R^N))'=L^\infty(\R^N),$ is a continuous operator. Hence $\|{\rm e}^{\Delta}[T_{\max}(\lambda\m)^\frac{N}{2}D_{\sqrt {T_{\max}(\lambda\m)}}\m]\|_\infty$ converges to $\|\m\|_\M\|{\rm e}^{\Delta}\delta\|_\infty=\|\m\|_\M(4\pi)^{-N/2}$ as $\lambda \to 0.$ This along with \eqref{neccdmaxlamdelta} implies
\begin{equation*}
\limsup_{\lambda \to \infty} \lambda^\alpha T_{\max}(\lambda\m)^{1 - \frac{N\alpha}{2}}
\le \frac{1}{\left(\alpha^{1/\alpha}(4\pi)^{-N/2}\|\m\|_\M\right)^\alpha}.
\end{equation*}
This gives the desired result.
\end{proof}
\begin{proof}[Proof of Theorem \ref{uppernegpower2}]
If $\varphi$ is too singular, it may happen that $T_{\max}(\lambda\varphi) = 0$, i.e. there is no local
nonnegative solution with initial value $\lambda\varphi$. This is not a problem, since we will be obtaining upper bounds.

Since $\varphi \ge \tilde{\tilde\varphi}$,  where $\tilde{\tilde\varphi}$ is defined in \eqref{phitildetilde}, it suffices to estimate $T_{\max}(\lambda\tilde{\tilde\varphi})$.
The calculation in \eqref{gammascale} gives
\begin{equation*}
\|{\rm e}^{t\Delta}\tilde{\tilde\varphi}\|_\infty
= t^{-\frac{\gamma}{2}}\|{\rm e}^{\Delta}[t^{\frac{\gamma}{2}}D_{\sqrt t}\tilde{\tilde\varphi}]\|_\infty.
\end{equation*}
Moreover, $t^{\frac{\gamma}{2}}D_{\sqrt t}\tilde{\tilde\varphi} \to \omega|\cdot|^{-\gamma}$
as $t \to \infty$ in ${\mathcal D'}(\R^N).$ It follows, by \cite[Proposition 3.8 (i), page 1123]{CDW2003}  that
\begin{equation*}
t^{\frac{\gamma}{2}}\|{\rm e}^{t\Delta}\tilde{\tilde\varphi}\|_\infty
= \|{\rm e}^{\Delta}[t^{\frac{\gamma}{2}}D_{\sqrt t}\tilde{\tilde\varphi}]\|_\infty
\to \|{\rm e}^{\Delta}[\omega|\cdot|^{-\gamma}]\|_\infty,
\end{equation*}
as $t \to \infty$.
If $0 < t < T_{\max}(\lambda\tilde{\tilde\varphi})$, then by \eqref{neccdlam} we must have
\begin{equation}
\label{gammascale3}
\alpha \lambda^\alpha t^{1 - \frac{\alpha\gamma}{2}} [t^{\frac{\gamma}{2}}\|{\rm e}^{t\Delta}\tilde{\tilde\varphi}\|_\infty ]^\alpha \le 1.
\end{equation}
It follows that if $\gamma < \frac{2}{\alpha}$, then $T_{\max}(\lambda\tilde{\tilde\varphi}) < \infty$
for all $\lambda > 0$. This is of course a consequence of Fujita's result (including the limiting case)
if $\alpha \le \frac{2}{N}$. (See for example \cite{W4}.) If $\alpha > \frac{2}{N}$, i.e. $q_c > 1$, this is a consequence of the
more general result \cite[Theorem 1.7]{TW2} in the case $m = 0$.  (See also \cite[Theorem 3.2(i)]{LeeNI} and
\cite[Theorem 2]{SW}.)
In these cases, putting $t = T_{\max}(\lambda\tilde{\tilde\varphi})$  in \eqref{gammascale3}
and letting $\lambda \to 0$, we obtain
\begin{equation*}
\limsup_{\lambda \to 0} \lambda^{(\frac{1}{\alpha} - \frac{\gamma}{2})^{-1}} T_{\max}(\lambda\tilde{\tilde\varphi})
\le \frac{1}{(\alpha^{1/\alpha}\|{\rm e}^{\Delta}(\omega|\cdot|^{-\gamma})\|_\infty)^{(\frac{1}{\alpha} - \frac{\gamma}{2})^{-1}}}.
\end{equation*}
\end{proof}
\begin{proof}[Proof of Theorem \ref{uppernegpower1signchanging}] The proof follows similarly as that of Theorem \ref{uppernegpower1}, replacing $\R^N$ by $\Omega_m,$ $|x|^{-\gamma}$ by $\psi_0$ hence $\gamma$ by $\gamma+m,$ and using \cite[Proposition 4.1 (ii), p. 359]{MTW}, for the convergence.
\end{proof}

\begin{proof}[Proof of Theorem \ref{uppernegpower2signchanging}]  The proof follows similarly as that of Theorem \ref{uppernegpower2} by replacing $\R^N$ by $\Omega_m,$ $|x|^{-\gamma}$ by $\psi_0$, hence $\gamma$ by $\gamma+m,$  using \cite[Theorem 1.7]{TW2} for the blow up of the solution and \cite[Proposition 4.1 (ii), p. 359]{MTW}, for the convergence.
\end{proof}
\section{Life-span estimates via nonlinear scaling}
\setcounter{equation}{0}
\label{nlsc}
In this section we show how certain scaling arguments can give
upper (and lower) life-span bounds for solutions of \eqref{NLHint} on $\R^N$.  Similar arguments can be used on sectors of $\R^N$.
The previous section likewise used scaling arguments, but only in regard to
properties of ${\rm e}^{t\Delta}\varphi$.  In this section, we use nonlinear
scaling arguments, which can then be adapted to other equations which are
scale invariant.  Some of the results in this section are the
same as in the previous section, but obtained by a different method.

We begin with some observations in a general context.
We consider an evolution partial differential equation defined either on $\R^N$ or on some
domain $\Omega$ which is a cone, i.e. if $x \in \Omega$ then $\mu x \in \Omega$ for all $\mu > 0$.
We also suppose that the set of solutions of the evolution equation is invariant under the transformation
\begin{equation}
\label{eqinvt}
u_\mu(t,x) = \mu^\sigma u(\mu^2 t, \mu x).
\end{equation}
In other words, $u$ is a solution if and only if $u_\mu$ is a solution for all $\mu > 0$.
If $u$ has initial value $u(0,\cdot) = u_0$, then $u_\mu$ has initial value
$u_\mu(0,\cdot) = \mu^\sigma u_0(\mu \cdot) =  \mu^\sigma D_\mu u_0 \equiv u_{\mu,0}.$
It is clear that
\begin{equation*}
T_{\max}(u_{\mu,0}) =  T_{\max}(\mu^\sigma D_\mu u_0) = \frac{1}{\mu^2}T_{\max}(u_0)
\end{equation*}
If $u_0 = \lambda \varphi$,  it follows that $u_{\mu,0} = \lambda\mu^\sigma D_\mu\varphi$,
so that
\begin{equation*}
\mu^{-2}T_{\max}(\lambda \varphi) = T_{\max}(\lambda\mu^\sigma D_\mu\varphi).
\end{equation*}
Now let us suppose that $\varphi$ has certain properties with respect to a scaling
different from that of the equation, for example $\mu^\gamma D_\mu\varphi$
where $\gamma \neq \sigma$.  If so, we  may set
\begin{equation}
\label{lambdamu}
\lambda = \mu^{\gamma - \sigma},
\end{equation}
hence $\mu = \lambda^{\frac{1}{\gamma - \sigma}},
 \quad \mu^{-2} = \lambda^{\frac{2}{\sigma - \gamma}},$ so that
\begin{equation}
\label{lifespanscale}
\lambda^{\frac{2}{\sigma - \gamma}}T_{\max}(\lambda \varphi) = T_{\max}(\mu^\gamma D_\mu\varphi).
\end{equation}

In the simplest case, $\mu^\gamma D_\mu\varphi \equiv \varphi$, i.e. $\varphi$ is  homogeneous
of degree $-\gamma$,  we have therefore the following formal proposition.

\begin{prop}
Let $\Omega \subset \R^N$ be a domain which is also a cone.  Suppose that the solutions of
an evolution equation (the set of trajectories of a dynamical system over $\Omega$) are invariant under the transformation \eqref{eqinvt}.
If $\varphi \in L^1_{loc}(\Omega)$ or $\varphi \in \M(\Omega)$ is homogeneous of degree $-\gamma$,
where $\gamma \neq \sigma$,  then
\begin{equation*}
\lambda^{\frac{2}{\sigma - \gamma}}T_{\max}(\lambda \varphi) = T_{\max}(\varphi)
\end{equation*}
for all $\lambda > 0$.
\end{prop}

In the case of the nonlinear heat equation, $\sigma = \frac{2}{\alpha}$ and so \eqref{lambdamu} and \eqref{lifespanscale} become
\begin{equation}
\label{lifespanscaleheat}
\lambda = \mu^{\gamma - \frac{2}{\alpha}}, \quad
\lambda^{(\frac{1}{\alpha} - \frac{\gamma}{2})^{-1}}T_{\max}(\lambda \varphi) = T_{\max}(\mu^\gamma D_\mu\varphi).
\end{equation}
We immediately deduce the following.
\begin{cor}
\label{invttmax}
Let $T_{\max}(u_0)$ denote the maximal solution to \eqref{NLHint} on $\R^N$ with initial value $u_0$.
\begin{enumerate}
\item[(i)] If $\alpha < \frac{2}{N}$, then
\begin{equation*}
\lambda^{(\frac{1}{\alpha} - \frac{N}{2})^{-1}}T_{\max}(\lambda\delta) = T_{\max}(\delta)
\end{equation*}
for all $\lambda > 0$.
\item[(ii)] If $0 < \gamma < N$ and $\gamma < \frac{2}{\alpha}$, and if $\psi(x) = \omega(x)|x|^{-\gamma}$
where $\omega \in L^\infty(\R^N)$ is homogeneous of degree $0$,
then
\begin{equation*}
\lambda^{(\frac{1}{\alpha} - \frac{\gamma}{2})^{-1}} T_{\max}(\lambda\psi) = T_{\max}(\psi)
\end{equation*}
for all $\lambda > 0$.
\end{enumerate}
\end{cor}

There are two other possibilities which allow us to obtain life-span estimates.
On the one hand, it could be that $\mu^\gamma D_\mu\varphi$ has a limit as
$\mu \to 0$ or as $\mu \to \infty$, possibly along a subsequence.  If
one can control $T_{\max}(\mu^\gamma D_\mu\varphi)$ as this limit
is attained, one obtains a corresponding life-span estimate from \eqref{lifespanscale}.
This procedure was introduced in the paper \cite{D}. For results of this type, we refer the reader to
\cite[Theorems 1.3, 1.4, 1.5]{D} and \cite[Theorems 1.9, 1.10, 1.12, Corollary 1.13, Propositions 4.5, 4.6]{TW2}.
As these latter results show, one can have different life-span behaviors along different
subsequences, either as $\lambda \to 0$ or as $\lambda \to \infty$.
We recall that all of these results depend on delicate continuity properties of the
blowup time.

The other approach uses comparison.  As a first, and simple, example,
we have the following immediate consequence of Corollary~\ref{invttmax}.

\begin{cor}
If $0 < \gamma < N$ and $\gamma < \frac{2}{\alpha}$, and if $\varphi \in L^1_{loc}(\R^N)$, $|\varphi(x)| \le \omega(x)|x|^{-\gamma}$
where $\omega \in L^\infty(\R^N)$, $\omega \ge 0$ is homogeneous of degree $0$, then
\begin{equation*}
\lambda^{(\frac{1}{\alpha} - \frac{\gamma}{2})^{-1}} T_{\max}(\lambda\varphi) \ge T_{\max}(\omega|\cdot|^{-\gamma})
\end{equation*}
for all $\lambda > 0$.
\end{cor}

This is essentially the same as Corollary~\ref{negpower}. Theorem \ref{lowerLqgamma} gives also the result but here the constant at the right-hand side is explicit.
\begin{proof}
The absolute value of the solution with initial value $\lambda\varphi$
is bounded above by the solution with initial value $\lambda\omega|\cdot|^{-\gamma}$. We then apply
the second assertion of Corollary~\ref{invttmax}.
\end{proof}

We have the following for the function $\tilde\varphi$  given by \eqref{phitilde}.
\begin{cor}
\label{resulphi1} Let $N\geq 1,\; \alpha>0,$ $0<\gamma<N,\; \gamma<{2\over \alpha},$  $\omega \in L^\infty(\R^N)$ is homogeneous of degree $0$, $\omega \ge 0$, $\omega \not \equiv 0$ and $\tilde\varphi$ be given by \eqref{phitilde}. Then the following hold.
\begin{itemize}
\item[(i)] There exists $T_1\in [T_{\max}(\omega|\cdot|^{-\gamma}),T_{\max}(\tilde\varphi))$ such that $$
\lim_{\lambda \to \infty}\lambda^{(\frac{1}{\alpha} - \frac{\gamma}{2})^{-1}}T_{\max}(\lambda \tilde\varphi)=T_1.$$
\item[(ii)]  $\lim_{\lambda \to 0}\lambda^{\left({1\over \alpha}-{\gamma\over 2}\right)^{-1}}T_{\max}(\lambda \tilde\varphi)=\infty.$
\end{itemize}
\end{cor}
\begin{proof}
(i) The function $\mu \to \mu^{\gamma} D_\mu\tilde\varphi$ is decreasing on $(0, \infty)$, and
\begin{equation*}
\lim_{\mu \to 0}\mu^{\gamma} D_\mu\tilde\varphi = \omega|\cdot|^{-\gamma}, \quad
\lim_{\mu \to \infty}\mu^{\gamma} D_\mu\tilde\varphi = 0,
\end{equation*}
where the first limit is realized in $L^{q_1}(\R^N) + L^{q_2}(\R^N)$ and the second in $L^{q_1}(\R^N)$ whenever
$0 \le \frac{N}{q_2} < \gamma <  \frac{N}{q_1} \le N$.
Consequently
\begin{equation*}
\tilde\varphi \le \mu^{\gamma} D_\mu\tilde\varphi \le  \omega|\cdot|^{-\gamma}, \; \forall \mu \le 1.
\end{equation*}
Applying \eqref{lifespanscaleheat}, and since $\gamma<2/\alpha,$ we conclude that
$\lambda \to \lambda^{(\frac{1}{\alpha} - \frac{\gamma}{2})^{-1}}T_{\max}(\lambda \tilde\varphi) $
is decreasing on $(0,\infty)$ and
\begin{equation*}
T_{\max}(\omega|\cdot|^{-\gamma}) \le
\lambda^{(\frac{1}{\alpha} - \frac{\gamma}{2})^{-1}}T_{\max}(\lambda \tilde\varphi)
\le T_{\max}(\tilde\varphi), \; \forall \lambda \ge 1.
\end{equation*}
The existence of the limit $T_1$ follows by  monotonicity.

(ii) We have that  $\tilde\varphi\in L^q$ for $q\geq 1,$ ${N\alpha\over 2}<q<{N\over \gamma}.$  Hence, by Theorem \ref{ThNLH}
    $$T_{\max}(\lambda \tilde\varphi)\geq C \lambda^{-\left({1\over \alpha}-{N\over 2q}\right)^{-1}}.$$ Then
    $$\lambda^{\left({1\over \alpha}-{\gamma\over 2}\right)^{-1}}T_{\max}(\lambda \tilde\varphi)\geq C\lambda^{\left({1\over \alpha}-{\gamma\over 2}\right)^{-1}-\left({1\over \alpha}-{N\over 2q}\right)^{-1}},$$ that is
    $$\lim_{\lambda \to 0}\lambda^{\left({1\over \alpha}-{\gamma\over 2}\right)^{-1}}T_{\max}(\lambda \tilde\varphi)=\infty.$$

For the second assertion, we may also use the continuous dependence in $L^q(\R^N)$, where $\gamma < \frac{N}{q} < \frac{2}{\alpha}$.
In fact, since $\lim_{\mu \to \infty}\mu^{\gamma} D_\mu\tilde\varphi = 0$, we know that
$T_{\max}(\mu^{\gamma} D_\mu\tilde\varphi) \to \infty$ as $\mu \to \infty$; so that by
\eqref{lifespanscaleheat}  we have
\begin{equation*}
\lambda^{(\frac{1}{\alpha} - \frac{\gamma}{2})^{-1}}T_{\max}(\lambda \tilde\varphi)
\to \infty, {\rm as}\,  \lambda \to 0.
\end{equation*}
\end{proof}
\begin{rem}{\rm $\,$
\begin{itemize}
\item[1)] It is natural to conjection that $\lambda^{(\frac{1}{\alpha} - \frac{\gamma}{2})^{-1}}T_{\max}(\lambda \tilde\varphi)
\to T_{\max}(\omega|\cdot|^{-\gamma}), {\rm as}\,  \lambda \to \infty.$ This holds in particular for $(N-2)\alpha<4,$ by continuous dependence of the maximal time of existence.
\item[2)] We remark that the upper bound on $\lambda^{(\frac{1}{\alpha} - \frac{\gamma}{2})^{-1}}T_{\max}(\lambda \tilde\varphi)$
for large $\lambda > 0$ is of the same order as given in Theorem~\ref{uppernegpower1}.
As for small $\lambda > 0$, if $\alpha < \frac{2}{N}$, then Theorem~\ref{lowermeas} gives the stronger
estimate $\lambda^{(\frac{1}{\alpha} - \frac{N}{2})^{-1}}T_{\max}(\lambda \tilde\varphi) \ge c > 0$
for all $\lambda > 0$, and is  of the same order as given in Theorem~\ref{uppermeas} and Remark \ref{rem5}.  However, Part (ii)
improves the estimate of Corollary~\ref{negpower} in the case $\alpha = \frac{2}{N}$, where
$T_{\max}(\lambda \tilde\varphi) < \infty$ for all $\lambda > 0$. If $\alpha > \frac{2}{N}$,
then $T_{\max}(\lambda \tilde\varphi) = \infty$ for sufficiently small $\lambda > 0,$ since
$\tilde\varphi \in L^{q_c}(\R^N)$, see \cite{W4}.
\end{itemize}}
\end{rem}

For the function $\tilde{\tilde\varphi},$ we have the following.
\begin{cor}
\label{resulphi2}
Let $N\geq 1,$ $\alpha>0, \; 0<\gamma<N,\; \gamma<{2\over \alpha},$  $\omega \in L^\infty(\R^N)$ is homogeneous of degree $0$, $\omega \ge 0$, $\omega \not \equiv 0$ and $\tilde{\tilde\varphi}$ be given by \eqref{phitildetilde}. Then the following hold.
\begin{itemize}
\item[(i)] There exists $T_2\in [T_{\max}(\omega|\cdot|^{-\gamma}),T_{\max}(\tilde{\tilde\varphi}))$ such that $$
\lim_{\lambda \to 0}\lambda^{(\frac{1}{\alpha} - \frac{\gamma}{2})^{-1}}T_{\max}(\lambda \tilde{\tilde\varphi})=T_2.$$
\item[(ii)] $\lim_{\lambda \to \infty}\lambda^{\left({1\over \alpha}-{\gamma\over 2}\right)^{-1}}T_{\max}(\lambda \tilde{\tilde\varphi})=\infty.$
\end{itemize}
\end{cor}
\begin{proof}
 (i) The function $\mu \to \mu^{\gamma} D_\mu\tilde{\tilde\varphi}$ is increasing on $(0, \infty)$, and
\begin{equation*}
\lim_{\mu \to 0}\mu^{\gamma} D_\mu\tilde{\tilde\varphi} = 0, \quad
\lim_{\mu \to \infty}\mu^{\gamma} D_\mu\tilde{\tilde\varphi} = \omega|\cdot|^{-\gamma},
\end{equation*}
where the first limit is realized in $L^{q_2}(\R^N)$ and the second in $L^{q_1}(\R^N) + L^{q_2}(\R^N)$ whenever
$0 \le \frac{N}{q_2} < \gamma <  \frac{N}{q_1} \le N$.
Consequently
\begin{equation*}
\tilde{\tilde\varphi} \le \mu^{\gamma} D_\mu\tilde{\tilde\varphi} \le  \omega|\cdot|^{-\gamma}, \; \forall \mu \ge 1.
\end{equation*}
Applying \eqref{lifespanscaleheat}  we conclude that
$\lambda \to \lambda^{(\frac{1}{\alpha} - \frac{\gamma}{2})^{-1}}T_{\max}(\lambda \tilde{\tilde\varphi}) $
is increasing on $(0,\infty)$ and
\begin{equation*}
T_{\max}(\omega|\cdot|^{-\gamma}) \le
\lambda^{(\frac{1}{\alpha} - \frac{\gamma}{2})^{-1}}T_{\max}(\lambda \tilde{\tilde\varphi})
\le T_{\max}(\tilde{\tilde\varphi}), \; \forall \lambda \le 1.
\end{equation*}

(ii)  By continuous dependence in $L^\infty(\R^N)$, we know that
$T_{\max}(\mu^{\gamma} D_\mu\tilde{\tilde\varphi}) \to \infty$ as $\mu \to 0$; so that by
\eqref{lifespanscaleheat}  we have
\begin{equation*}
\lambda^{(\frac{1}{\alpha} - \frac{\gamma}{2})^{-1}}T_{\max}(\lambda \tilde{\tilde\varphi})
\to \infty, {\rm as}\,  \lambda \to \infty.
\end{equation*}
For the second assertion, we may also use the following argument. We have that  $\tilde{\tilde\varphi}\in  L^\infty_{\gamma}\cap L^\infty .$ That is, by Corollary \ref{lowerLqLqgamma}
    $$T_{\max}(\lambda \tilde{\tilde\varphi})\geq C \lambda^{-\left({1\over \alpha}-{\gamma\over 2}\right)^{-1}}$$ and by Theorem \ref{ThNLH}
    $$T_{\max}(\lambda \tilde{\tilde\varphi})\geq C \lambda^{-\left({1\over \alpha}\right)^{-1}}.$$ Then
    $$\lambda^{\left({1\over \alpha}-{\gamma\over 2}\right)^{-1}}T_{\max}(\lambda \tilde{\tilde\varphi})\geq
    \max\left(C,C\lambda^{\left({1\over \alpha}-{\gamma\over 2}\right)^{-1}-\left({1\over \alpha}\right)^{-1}}\right).$$  Hence
    $$\lim_{\lambda \to\infty}\lambda^{\left({1\over \alpha}-{\gamma\over 2}\right)^{-1}}T_{\max}(\lambda \tilde{\tilde\varphi})=\infty.$$
\end{proof}
\begin{rem}
{\rm $\,$
\begin{itemize}
\item[1)] It is natural to conjection that $\lambda^{(\frac{1}{\alpha} - \frac{\gamma}{2})^{-1}}T_{\max}(\lambda \tilde{\tilde\varphi})
\to T_{\max}(\omega|\cdot|^{-\gamma}), {\rm as}\,  \lambda \to 0.$
\item[2)] We remark that $\tilde{\tilde\varphi} \in L^\infty(\R^N)$, and Theorem \ref{ThNLH} and Theorem \ref{upperLinfty}
 give the precise order of magnitude of $T_{\max}(\lambda \tilde{\tilde\varphi}) $ as $\lambda \to \infty$.
 \end{itemize}}
 \end{rem}

Next, we consider the  function $\Phi$ given by  \eqref{phi3}. We have the following.
\begin{cor}
\label{resulphi3}
Let $N\geq 1,\; \alpha>0,$ $0 < \gamma_1, \gamma_2 < N$ and $\gamma_1, \gamma_2 < \frac{2}{\alpha}$ ($\gamma_1 \neq \gamma_2$).
 Let $\omega \in L^\infty(\R^N)$ is homogeneous of degree $0$, $\omega \ge 0$, $\omega \not \equiv 0$ and let $\Phi$ be defined by \eqref{phi3}. Then we have the following.
\begin{itemize}
\item[(i)] There exists $T_3\in (T_{\max}(\Phi), T_{\max}(\omega|\cdot|^{-\gamma_1})]$ such that $$\lim_{\lambda \to \infty}\lambda^{(\frac{1}{\alpha} - \frac{\gamma_1}{2})^{-1}}T_{\max}(\lambda \Phi)=T_3$$
and if $\alpha < \frac{4}{N-2}$, or $\gamma_1 > \gamma_2$ then $$T_3=T_{\max}(\omega|\cdot|^{-\gamma_1}).$$

\item[(ii)] There exists ${\tilde T}_3\in (T_{\max}(\Phi), T_{\max}(\omega|\cdot|^{-\gamma_2})]$ such that $$\lim_{\lambda \to 0}\lambda^{(\frac{1}{\alpha} - \frac{\gamma_2}{2})^{-1}}T_{\max}(\lambda \Phi)={\tilde T}_3$$
and if $\alpha < \frac{4}{N-2}$, or $\gamma_1 > \gamma_2$ then $${\tilde T}_3=T_{\max}(\omega|\cdot|^{-\gamma_2}).$$
\item[(iii)] If $\gamma_1 < \gamma_2,$ then
$\lim_{\lambda \to 0}\lambda^{(\frac{1}{\alpha} - \frac{\gamma_1}{2})^{-1}}T_{\max}(\lambda \Phi) =\lim_{\lambda \to \infty}\lambda^{(\frac{1}{\alpha} - \frac{\gamma_2}{2})^{-1}}T_{\max}(\lambda \Phi)= \infty.$
\end{itemize}
\end{cor}
\begin{rem}
{\rm $\,$ Corollary \ref{resulphi3} shows that the asymptotic behavior of the life-span as $\lambda\to \infty$ is determined by the singularity of the initial data and when $\lambda\to 0$ it is determined by the decay rate at infinity of the initial value.}
\end{rem}
\begin{proof}[Proof of Corollary \ref{resulphi3}]$\,$

{\bf 1) Analysis of $T_{\max}(\lambda \Phi)$ in the case $\gamma_1 < \gamma_2$.}
In this case $\Phi = \omega\min[|\cdot|^{-\gamma_1},|\cdot|^{-\gamma_2}]$,
so $\Phi \le \omega|\cdot|^{-\gamma_1}$ and $\Phi \le \omega|\cdot|^{-\gamma_2}$.
Hence
\begin{equation}
\label{bd1}
\mu^{\gamma_1} D_\mu\Phi \le \omega|\cdot|^{-\gamma_1}
\end{equation}
and
\begin{equation}
\label{bd2}
\mu^{\gamma_2} D_\mu\Phi \le \omega|\cdot|^{-\gamma_2}
\end{equation}
 for all $\mu > 0$.

We claim that
\begin{itemize}
\item The function $\mu \to \mu^{\gamma_1} D_\mu\Phi$ is decreasing on $(0, \infty)$, and
\begin{equation*}
\lim_{\mu \to 0}\mu^{\gamma_1} D_\mu\Phi = \omega|\cdot|^{-\gamma_1}, \quad
\lim_{\mu \to \infty}\mu^{\gamma_1} D_\mu\Phi = 0,
\end{equation*}
where the limits are in $L^{q_1}(\R^N) + L^{q_2}(\R^N)$ whenever
$0 \le \frac{N}{q_2} < \gamma_1 < \frac{N}{q_1} \le N$, by \eqref{bd1}.
\item The function $\mu \to \mu^{\gamma_2} D_\mu\Phi$ is increasing on $(0, \infty)$, and
\begin{equation*}
\lim_{\mu \to 0}\mu^{\gamma_2} D_\mu\Phi = 0, \quad
\lim_{\mu \to \infty}\mu^{\gamma_2} D_\mu\Phi = \omega|\cdot|^{-\gamma_2},
\end{equation*}
where the limits are in $L^{q_1}(\R^N) + L^{q_2}(\R^N)$ whenever
$0 \le \frac{N}{q_2} < \gamma_2 < \frac{N}{q_1} \le N$, by \eqref{bd2}.
\end{itemize}

{\it Proof of the  claim}.
Let $0 < \mu < \nu < \infty$.  In particular, $\mu^{\gamma_1 - \gamma_2} > \nu^{\gamma_1 - \gamma_2}$.
We have
\begin{eqnarray*}
\mu^{\gamma_1}\Phi(\mu x)
&=& \begin{cases}
\omega(x)|x|^{-\gamma_1}, &|x| \le \frac{1}{\nu}\\
\omega(x)|x|^{-\gamma_1}, &\frac{1}{\nu} \le |x| \le \frac{1}{\mu}\\
\mu^{\gamma_1 - \gamma_2}\omega(x)|x|^{-\gamma_2}, &|x| \ge \frac{1}{\mu}.
\end{cases}
\quad =\begin{cases}
\omega(x)|x|^{-\gamma_1}, &|x| \le \frac{1}{\nu}\\
(\frac{1}{|x|})^{\gamma_1 - \gamma_2}\omega(x)|x|^{-\gamma_2}, &\frac{1}{\nu} \le |x| \le \frac{1}{\mu}\\
\mu^{\gamma_1 - \gamma_2}\omega(x)|x|^{-\gamma_2}, &|x| \ge \frac{1}{\mu}.
\end{cases}\\
&\ge& \begin{cases}
\omega(x)|x|^{-\gamma_1}, &|x| \le \frac{1}{\nu}\\
\nu^{\gamma_1 - \gamma_2}\omega(x)|x|^{-\gamma_2}, &\frac{1}{\nu} \le |x| \le \frac{1}{\mu}\\
\nu^{\gamma_1 - \gamma_2}\omega(x)|x|^{-\gamma_2}, &|x| \ge \frac{1}{\mu}.
\end{cases}
\quad = \nu^{\gamma_1}\Phi(\nu x).
\end{eqnarray*}
Also,
\begin{eqnarray*}
\mu^{\gamma_2}\Phi(\mu x)
&=& \begin{cases}
\mu^{\gamma_2 - \gamma_1}\omega(x)|x|^{-\gamma_1}, &|x| \le \frac{1}{\nu}\\
\mu^{\gamma_2 - \gamma_1}\omega(x)|x|^{-\gamma_1}, &\frac{1}{\nu} \le |x| \le \frac{1}{\mu}\\
\omega(x)|x|^{-\gamma_2}, &|x| \ge \frac{1}{\mu}.
\end{cases}
\quad = \begin{cases}
\mu^{\gamma_2 - \gamma_1}\omega(x)|x|^{-\gamma_1}, &|x| \le \frac{1}{\nu}\\
\mu^{\gamma_2 - \gamma_1}|x|^{\gamma_2 - \gamma_1}\omega(x)|x|^{-\gamma_2}, &\frac{1}{\nu} \le |x| \le \frac{1}{\mu}\\
\omega(x)|x|^{-\gamma_2}, &|x| \ge \frac{1}{\mu}.
\end{cases}\\
&\le& \begin{cases}
\nu^{\gamma_2 - \gamma_1}\omega(x)|x|^{-\gamma_1}, &|x| \le \frac{1}{\nu}\\
\omega(x)|x|^{-\gamma_2}, &\frac{1}{\nu} \le |x| \le \frac{1}{\mu}\\
\omega(x)|x|^{-\gamma_2}, &|x| \ge \frac{1}{\mu}.
\end{cases}
\quad = \nu^{\gamma_1}\Phi(\nu x).
\end{eqnarray*}

It follows that
\begin{itemize}
\item The function $\mu \to T_{\max}(\mu^{\gamma_1} D_\mu\Phi)$ is increasing on $(0, \infty)$, and
\begin{equation*}
T_{\max}(\Phi) \ge \lim_{\mu \to 0}T_{\max}(\mu^{\gamma_1} D_\mu\Phi)
 \ge T_{\max}(\omega|\cdot|^{-\gamma_1}),
\end{equation*}
and if $\alpha < \frac{4}{N-2}$, then
\begin{equation*}
\lim_{\mu \to 0}T_{\max}(\mu^{\gamma_1} D_\mu\Phi)
= T_{\max}(\omega|\cdot|^{-\gamma_1}).
\end{equation*}
Also
\begin{equation*}
\lim_{\mu \to \infty}T_{\max}(\mu^{\gamma_1} D_\mu\Phi) = \infty.
\end{equation*}

\item The function $\mu \to T_{\max}(\mu^{\gamma_2} D_\mu\Phi)$ is decreasing on $(0, \infty)$, and
\begin{equation*}
T_{\max}(\Phi) \ge \lim_{\mu \to \infty}T_{\max}(\mu^{\gamma_2} D_\mu\Phi)
\ge  T_{\max}(\omega|\cdot|^{-\gamma_2}),
\end{equation*}
and if $\alpha < \frac{4}{N-2}$, then
\begin{equation*}
\lim_{\mu \to \infty}T_{\max}(\mu^{\gamma_2} D_\mu\Phi)
= T_{\max}(\omega|\cdot|^{-\gamma_2}).
\end{equation*}
Also
\begin{equation*}
\lim_{\mu \to 0}T_{\max}(\mu^{\gamma_2} D_\mu\Phi) = \infty.
\end{equation*}
\end{itemize}

Next, applying \eqref{lifespanscaleheat}, we have
\begin{equation}
\lambda^{(\frac{1}{\alpha} - \frac{\gamma_1}{2})^{-1}}T_{\max}(\lambda \Phi)
=  T_{\max}(\mu^{\gamma_1} D_\mu\Phi ), \quad \lambda = \mu^{\gamma_1 - \frac{2}{\alpha}},
\end{equation}
\begin{equation}
\lambda^{(\frac{1}{\alpha} - \frac{\gamma_2}{2})^{-1}}T_{\max}(\lambda \Phi)
=  T_{\max}(\mu^{\gamma_2} D_\mu\Phi ), \quad \lambda = \mu^{\gamma_2 - \frac{2}{\alpha}}.
\end{equation}
This completes the proof of  (i)-(ii) if $\gamma_1<\gamma_2$ and (iii).

{\bf 2) Analysis of $T_{\max}(\lambda \Phi)$ in the case $\gamma_1 > \gamma_2$.}

In this case $\Phi = \omega\max[|\cdot|^{-\gamma_1},|\cdot|^{-\gamma_2}]$,
so $\Phi \ge \omega|\cdot|^{-\gamma_1}$,  $\Phi \ge \omega|\cdot|^{-\gamma_2}$
and $\Phi \le \omega(|\cdot|^{-\gamma_1} + |\cdot|^{-\gamma_2})$.
Hence
\begin{equation}
\label{bd3}
\omega|\cdot|^{-\gamma_1} \le \mu^{\gamma_1} D_\mu\Phi
 \le \omega(|\cdot|^{-\gamma_1} + \mu^{\gamma_1 - \gamma_2}|\cdot|^{-\gamma_2})
\end{equation}
and
\begin{equation}
\label{bd4}
\omega|\cdot|^{-\gamma_2} \le \mu^{\gamma_2} D_\mu\Phi
 \le \omega(\mu^{\gamma_2 - \gamma_1}|\cdot|^{-\gamma_1} + |\cdot|^{-\gamma_2})
\end{equation}
 for all $\mu > 0$.

 We claim that
\begin{itemize}
\item The function $\mu \to \mu^{\gamma_1} D_\mu\Phi$ is increasing on $(0, \infty)$, and
\begin{equation*}
\lim_{\mu \to 0}\mu^{\gamma_1} D_\mu\Phi = \omega|\cdot|^{-\gamma_1}, \quad
\lim_{\mu \to \infty}\mu^{\gamma_1} D_\mu\Phi = \infty,
\end{equation*}
where the limits are in $L^{q_1}(\R^N) + L^{q_2}(\R^N)$ whenever
$0 \le \frac{N}{q_2} < \gamma_2 < \gamma_1 < \frac{N}{q_1} \le N$, by \eqref{bd3}.

\item The function $\mu \to \mu^{\gamma_2} D_\mu\Phi$ is decreasing on $(0, \infty)$, and
\begin{equation*}
\lim_{\mu \to 0}\mu^{\gamma_2} D_\mu\Phi = \infty, \quad
\lim_{\mu \to \infty}\mu^{\gamma_2} D_\mu\Phi = \omega|\cdot|^{-\gamma_2},
\end{equation*}
where the limits are in $L^{q_1}(\R^N) + L^{q_2}(\R^N)$ whenever
$0 \le \frac{N}{q_2} < \gamma_2 < \gamma_1 < \frac{N}{q_1} \le N$, by \eqref{bd4}.
\end{itemize}

{\it Proof of the  claim}.
Let $0 < \mu < \nu < \infty$, so that $\mu^{\gamma_2 - \gamma_1} > \nu^{\gamma_2 - \gamma_1}$.
We have

\begin{eqnarray*}
\mu^{\gamma_1}\Phi(\mu x)
&=& \begin{cases}
\omega(x)|x|^{-\gamma_1}, &|x| \le \frac{1}{\nu}\\
\omega(x)|x|^{-\gamma_1}, &\frac{1}{\nu} \le |x| \le \frac{1}{\mu}\\
\mu^{\gamma_1 - \gamma_2}\omega(x)|x|^{-\gamma_2}, &|x| \ge \frac{1}{\mu}.
\end{cases}
\quad = \begin{cases}
\omega(x)|x|^{-\gamma_1}, &|x| \le \frac{1}{\nu}\\
(\frac{1}{|x|})^{\gamma_1 - \gamma_2}\omega(x)|x|^{-\gamma_2}, &\frac{1}{\nu} \le |x| \le \frac{1}{\mu}\\
\mu^{\gamma_1 - \gamma_2}\omega(x)|x|^{-\gamma_2}, &|x| \ge \frac{1}{\mu}.
\end{cases}\\
&\le& \begin{cases}
\omega(x)|x|^{-\gamma_1}, &|x| \le \frac{1}{\nu}\\
\nu^{\gamma_1 - \gamma_2}\omega(x)|x|^{-\gamma_2}, &\frac{1}{\nu} \le |x| \le \frac{1}{\mu}\\
\nu^{\gamma_1 - \gamma_2}\omega(x)|x|^{-\gamma_2}, &|x| \ge \frac{1}{\mu}.
\end{cases}
\quad = \nu^{\gamma_1}\Phi(\nu x).
\end{eqnarray*}
Also,
\begin{eqnarray*}
\mu^{\gamma_2}\Phi(\mu x)
&=& \begin{cases}
\mu^{\gamma_2 - \gamma_1}\omega(x)|x|^{-\gamma_1}, &|x| \le \frac{1}{\nu}\\
\mu^{\gamma_2 - \gamma_1}\omega(x)|x|^{-\gamma_1}, &\frac{1}{\nu} \le |x| \le \frac{1}{\mu}\\
\omega(x)|x|^{-\gamma_2}, &|x| \ge \frac{1}{\mu}.
\end{cases}
\quad = \begin{cases}
\mu^{\gamma_2 - \gamma_1}\omega(x)|x|^{-\gamma_1}, &|x| \le \frac{1}{\nu}\\
\mu^{\gamma_2 - \gamma_1}|x|^{\gamma_2 - \gamma_1}\omega(x)|x|^{-\gamma_2}, &\frac{1}{\nu} \le |x| \le \frac{1}{\mu}\\
\omega(x)|x|^{-\gamma_2}, &|x| \ge \frac{1}{\mu}.
\end{cases}\\
&\ge& \begin{cases}
\nu^{\gamma_2 - \gamma_1}\omega(x)|x|^{-\gamma_1}, &|x| \le \frac{1}{\nu}\\
\omega(x)|x|^{-\gamma_2}, &\frac{1}{\nu} \le |x| \le \frac{1}{\mu}\\
\omega(x)|x|^{-\gamma_2}, &|x| \ge \frac{1}{\mu}.
\end{cases}
\quad = \nu^{\gamma_1}\Phi(\nu x).
\end{eqnarray*}
It follows that
\begin{itemize}
\item The function $\mu \to T_{\max}(\mu^{\gamma_1} D_\mu\Phi)$ is decreasing on $(0, \infty)$, and
\begin{equation*}
\lim_{\mu \to 0}T_{\max}(\mu^{\gamma_1} D_\mu\Phi)
= T_{\max}(\omega|\cdot|^{-\gamma_1}).
\end{equation*}
\item The function $\mu \to T_{\max}(\mu^{\gamma_2} D_\mu\Phi)$ is increasing on $(0, \infty)$, and
\begin{equation*}
\lim_{\mu \to \infty}T_{\max}(\mu^{\gamma_2} D_\mu\Phi)
= T_{\max}(\omega|\cdot|^{-\gamma_2}).
\end{equation*}
\end{itemize}
Next, applying \eqref{lifespanscaleheat}, we have
\begin{equation*}
\lambda^{(\frac{1}{\alpha} - \frac{\gamma_1}{2})^{-1}}T_{\max}(\lambda \Phi)
=  T_{\max}(\mu^{\gamma_1} D_\mu\Phi ), \quad \lambda = \mu^{\gamma_1 - \frac{2}{\alpha}},
\end{equation*}
\begin{equation*}
\lambda^{(\frac{1}{\alpha} - \frac{\gamma_2}{2})^{-1}}T_{\max}(\lambda \Phi)
=  T_{\max}(\mu^{\gamma_2} D_\mu\Phi ), \quad \lambda = \mu^{\gamma_2 - \frac{2}{\alpha}}.
\end{equation*}
To show  that we reach the above exact limits we use the following.

{\it Observation.} If $\phi_k \to \phi$ and assume we are in a situation of continuous dependence,
then we know $\liminf_{k \to \infty}T_{\max}(\phi_k) \ge T_{\max}(\phi)$.
Suppose also that $\phi \le \phi_k$ for all $k$, hence $T_{\max}(\phi_k) \le T_{\max}(\phi)$ for all $k$,
so $\limsup_{k \to \infty}T_{\max}(\phi_k) \le T_{\max}(\phi)$.  Hence
$\lim_{k \to \infty}T_{\max}(\phi_k) = T_{\max}(\phi)$.

This allows us to obtain the above exact limits  for the case $\gamma_1>\gamma_2.$
This completes the proof of  (i)-(ii) if $\gamma_1>\gamma_2$.

We may also show  (iii) as follows. The function $\Phi$ verifies:
     if $\gamma_1<\gamma_2$ then $\Phi\in L^q,$ $$ {\gamma_1\over N}<{1\over q}<{\gamma_2\over N}<\min(1,{2\over N\alpha}).$$
    Hence, by Theorem \ref{ThNLH}
    $$T_{\max}(\lambda \Phi)\geq C \lambda^{-\left({1\over \alpha}-{N\over 2q}\right)^{-1}}.$$ Then
    $$\lim_{\lambda\to 0}\lambda^{\left({1\over \alpha}-{\gamma_1\over 2}\right)^{-1}}T_{\max}(\lambda \Phi)= \infty= \lim_{\lambda\to \infty}\lambda^{\left({1\over \alpha}-{\gamma_2\over 2}\right)^{-1}}T_{\max}(\lambda \Phi).$$ This completes the proof of Corollary \ref{resulphi3}.
\end{proof}
\begin{appendix}
\section{Nonlinear Hardy parabolic equations}
\label{hh}
\setcounter{equation}{0}
Our purpose in the appendix is to estimate the life-span of solutions for the nonlinear Hardy-H\'enon  parabolic equations
\begin{equation}
\label{HHNLH}
\partial_t u=\Delta u+|\cdot|^{l}|u|^\alpha u,
\end{equation}
 $u=u(t,x)\in \RR,\; t>0,\; x\in \RR^N,\; N\geq 1,\; \alpha>0,\; -\min(2,N)<l$ and
with initial value
\begin{equation}
\label{HHNLHinitial}u(0)=u_0.
\end{equation}
A mild solution of the problem \eqref{HHNLH}-\eqref{HHNLHinitial} is a solution
of the integral equation
\begin{equation}
\label{intHHNLH}
u(t)= e^{t\Delta}u_0+\displaystyle\int_0^te^{(t-\sigma)\Delta}\left(|\cdot|^{l}|u(\sigma)|^\alpha
u(\sigma)\right)d\sigma,
 \end{equation}
and it is in this form that we consider problem \eqref{HHNLH}-\eqref{HHNLHinitial}.
 
In this first part of the appendix  we consider the case $l<0,$ that is the Hardy case. The problem \eqref{intHHNLH} is well-posed in $C([0,T];L^q(\R^N))\cap C((0,T];L^r(\R^N)),\; T>0,$ for $u_0\in L^q(\R^N),$  $1<q < \infty,$ $q>q_c(l)$ or $u_0\in C_0(\RR^N),$
where \begin{equation}\label{qcritiqueHH}q_c(l)={N\alpha\over
2+l},\end{equation}
and $r>q$ satisfies
\begin{equation}
\label{supercritHH}
  \frac{1}{q(\alpha+1)} + \frac{l}{N(\alpha+1)} < \frac{1}{r} < \frac{N + l}{N(\alpha+1)}.
\end{equation} See \cite[Theorem 1.1, p. 117]{BTW} and \cite{BenSlimene}. This solution can be extended to a maximal solution defined on $[0,T_{\max}(u_0)).$ We have obtained the following.
\begin{To}[The nonlinear Hardy parabolic equations]
\label{ThHH}
Let $N\geq 1,$ $-\min(2,N)<l<0,$ $\alpha>0,$   and $q_c(l)$ be given by \eqref{qcritiqueHH}. Let $\varphi\in L^q(\R^N)$ with $1<q < \infty,$ $q>q_c(l)$ or $\varphi\in C_0(\RR^N)$ and  $r>q$ satisfies \eqref{supercritHH}. Let $u\in C\left([0,T_{\max}(\lambda \varphi)); L^q(\R^N)\right)\cap C((0,T_{\max}(\lambda \varphi));L^r(\R^N))$ be the maximal solution of \eqref{intHHNLH}  with initial data $u_0=\lambda \varphi,$ $\lambda>0.$ Then there exists a constant $C=C(N,l,\alpha,q)>0$ such that
\begin{equation}
\label{lowerHH}
T_{\rm{max}}(\lambda\varphi) { \geq} C\left(\lambda \|\varphi\|_q\right)^{-\left({2+l\over 2\alpha}-{N\over 2q}\right)^{-1}},
\end{equation}
for all $\lambda>0.$
\end{To}
 \begin{proof}[Proof of Theorem \ref{ThHH}]    For  $q>q_c(l),$ let $r>q$ satisfying \eqref{supercritHH}. Let us define
$$
\beta(l)={N\over 2}\left({1\over q}-{1\over r}\right).$$
We note that $r$ depends on $l,$ hence $\beta$ also. The  well-posedness results for \eqref{HHNLH} has been obtained in \cite{BTW,BenSlimene}.
We now give the proof of \eqref{lowerHH}. Let $-\min(2,N)<l<0,\; \alpha>0,$ $\lambda>0,$ $K>0$ and $\varphi\in C_0(\mathbb R^N),$ $\|\varphi\|_\infty\leq K$ or $\varphi\in L^q,\; q>1,\; q>q_c(l)$   such that $\|\varphi\|_q\leq K.$   Let $u\in C\big([0, T_{\rm{max}}(\lambda\varphi));L^q(\RR^N) \big)\cap C\big((0,T_{\rm{max}}(\lambda\varphi));L^r(\RR^N) \big)$ with  $r>q$ satisfying \eqref{supercritHH}, be the maximal solution of \eqref{HHNLH} on $[0,T_{\rm{max}}(\lambda\varphi))$. It is proved in \cite[Inequalities (3.5), (3.6), p. 124]{BTW} that for $K,\; T,\; M>0$ such that
$$K+CT^{1-{N\alpha\over 2q}+{l\over 2}}M^{\alpha+1}\leq M$$ the solution $u$ of \eqref{HHNLH} is defined on $[0,T]$ and verifies $\max\left[\sup_{t\in (0,T]}t^{\beta(l)}\|u(t)\|_r,\sup_{[0,T]}\|u(t)\|_q\right]\leq M.$ Here $C$ is a positive constant. Then for $T_{\rm{max}}(\varphi)$ we should have
 $$K + C\Big(T_{\rm{max}}(\varphi)\Big)^{1 - \frac{N\alpha}{2q}+{l\over 2}}M^{\alpha + 1} > M,$$
 for all $M>K.$ That is it must be
$$\lambda K + C\Big(T_{\rm{max}}(\lambda\varphi)\Big)^{1 - \frac{N\alpha}{2q}+{l\over 2}}M^{\alpha + 1} > M, $$ for all $M>\lambda K.$ If we set $M=2\lambda K,$ we get
$$ 2^{\alpha+1}K^\alpha CT_{\rm{max}}(\lambda\varphi)^{1 - \frac{N\alpha}{2q}+{l\over 2}}\lambda^{\alpha} >1.$$
Then taking $K=\|\varphi\|_q,$ we get that there exists $C=C(N,\alpha,l,q)>0$ (since $ r$ itself depends on $q$) such that \eqref{lowerHH} holds. This completes the proof of the Theorem.
 \end{proof}

Using similar argument developed to prove  Theorem \ref{ThHH}, we derive the same result for the equation
\begin{equation}
\label{Hardy-Henonax}
\partial_t u=\Delta u+a(x)|u|^\alpha u,
\end{equation}
where $a(x)|\cdot|^{-l}$ is in $L^\infty(\R^N).$ In particular, we may take $a$ regular near the origin. Then we have the following.
\begin{cor}
\label{corHH}
Let $N\geq 1,$ $-\min(2,N)<l<0,$ $\alpha>0,$   and $q_c(l)$ be given by \eqref{qcritiqueHH}. Let $\varphi\in L^q(\R^N)$ with $1< q \leq \infty,$ $q>q_c(l)$ or $\varphi\in C_0(\RR^N)$ and  $r>q$ satisfies \eqref{supercritHH}. Let $u\in C\left([0,T_{\max}(\lambda \varphi)); L^q(\R^N)\right)\cap C((0,T_{\max}(\lambda \varphi));L^r(\R^N))$ be the maximal mild solution of \eqref{Hardy-Henonax} such that  $a(x)|\cdot|^{-l}\in L^\infty(\R^N)$ and with initial data $u_0=\lambda \varphi,$ $\lambda>0,$ constructed by \cite[Theorem 1.1, p. 117 ]{BTW} and \cite[p. 142]{BTW} (we replace $[0,T_{\max}(\lambda \varphi))$ by $(0,T_{\max}(\lambda \varphi))$ if $q=\infty$). Then \eqref{lowerHH} holds for all $\lambda>0.$
\end{cor}

 Corollary \ref{corHH} includes many known results. We will compare our results with those of \cite{P}. For this we restrict ourselves to the case where $a$ is positive and H\"older continuous as assumed in \cite{P}. Also,  it is supposed in \cite{P} that $\varphi\in C_b(\R^N),$ $\varphi\geq 0.$ For $\lambda>0$ small,   two classes of initial data are considered in  \cite{P}.

The first class is  for $\varphi$ dominated by a Gaussian. It is shown in \cite[Theorem 1 (i), p. 33]{P} that  if $0<\alpha<(2+l)/N,$ then
$$T_{\rm{max}}(\lambda\varphi)\geq C \lambda^{-\left({1\over \alpha}-{N\over 2}\right)^{-1}},\; \mbox{ as } \lambda \to 0.$$
For this class $\varphi\in L^q(\R^N)$ for all $q\geq 1.$ Since $a\in L^\infty(\R^N)$ because $l<0$ and $a$ is regular, then we may use Theorem \ref{ThNLH} (which is valid for such $a$ as noted before) and apply  \eqref{lowerNLH} with $q=1$ since $q_c<1,$ and then recover the result of \cite{P}.

 The second class considered in \cite{P} is  for $\varphi$ such that there exist constants $c_1,\; c_2>0$ and $c_1\leq \varphi\leq c_2.$ The estimates, as stated in  \cite[Theorem 2 (i)-(a), (ii)-(a), (iii)-(a),  pp. 33-34]{P}, reads
$$T_{\rm{max}}(\lambda\varphi)\geq C \lambda^{-{2\alpha\over 2+l}},\; \mbox{ as } \lambda \to 0.$$
Here $\varphi \in L^\infty(\R^N),$ the previous estimate is the same as \eqref{lowerHH} with $q=\infty.$ We then recover the results of \cite{P}.

\section{Nonlinear H\'enon parabolic equations}
\setcounter{equation}{0}
In this part of the appendix we study a nonlinear heat equation with a spatially growing variable coefficient. We consider the equation \eqref{HHNLH} for $l>0$ and with the initial condition \eqref{HHNLHinitial}. Local well-posedness is known in $C(\R^N)\cap L^\infty\cap L^\infty_{l/\alpha},$  (see \cite{W,Majdoub}). Recently local well-posedness is established   in $L^q_{s}$ for some $q\geq \alpha+1$ and $s$ satisfying some conditions (see \cite{CIT}). Not much is known about this equation in comparison with the standard nonlinear heat equation, that is the case $l=0$. In particular, the life-span is only known for small lambda and rapidly decaying positive initial data, see \cite{P}.  Note that the blowup may hold at the origin as it may also not hold at the origin. See \cite{GSouplet,GLS1,GLS2,GS,FT}. To show lower-bound estimates of the life-span, we establish local well-posedness results. Using Proposition \ref{smoothingeffectorl-lebeg}, we prove local existence for \eqref{HHNLH} in $ L^q_\gamma$  for  $$\gamma={l\over \alpha}<N$$
and  $q$ is such that
\begin{equation}
\label{qvalues13}
 q>q_c={N\alpha\over 2},\quad  \quad {N\over N-\gamma} < q\leq\infty.
\end{equation}
This value of $\gamma$ is inspired by \cite{W}.

We note that for $0<\gamma=l/\alpha<N,$ $$q_c>{N\over N-\gamma}\Leftrightarrow q_F:={N\alpha\over 2+l}>1,$$
where $q_F$ is the Fujita exponent for the equation \eqref{HHNLH} (see \cite{Qi}).
We have the following local well-posedness result.
\begin{To}[Local well-posedness in $L^q_\gamma$]
\label{localNLHGC} Let $N\geq 1$ be an integer, $\alpha>0$ and $l>0$ be
such that
\begin{equation}
\label{e03b} \gamma:={l\over \alpha}<N. \end{equation} Let $q_c$ be given by
\eqref{qvalues13}. Then we have the following.
\begin{itemize}
\item[(i)] If  $q$ is such that
$$q >\frac{N(\alpha+1)}{N-\gamma}, \quad q> q_c\quad \mbox{ and } \quad q\leq\infty,$$
then equation \eqref{intHHNLH} is locally well-posed in $L^q_\gamma(\RR^N).$ More precisely, given $u_0\in L^q_\gamma(\RR^N)$, then there exist $T>0$ and a unique solution of \eqref{intHHNLH} $u\in C\big([0, T];L^q_\gamma(\RR^N) \big)$  if $q<\infty$ and $u\in C\big((0, T];L^\infty_\gamma(\RR^N) \big),$   {  $\lim_{t\to 0}\|u(t)-e^{t \Delta}u_0\|_{L^\infty_\gamma(\R^N)}=0$ if $q=\infty$}. Moreover,
 $u$ can be extended to a maximal interval $[0, T_{\max})$ such that
either $T_{\max}= \infty$ or  $T_{\max}< \infty$ and
$\displaystyle\lim_{t\to  T_{\max}} \|u(t)\|_{L^q_\gamma} = \infty.$
\item[(ii)] Assume that $q>q_c$ with ${N\over N-\gamma} < q\leq\infty$. It follows that equation \eqref{intHHNLH} is locally well-posed in $L^q_\gamma(\RR^N)$ as in part (i) except that uniqueness is guaranteed only among  functions $u\in C\big([0, T];L^q_\gamma(\RR^N) \big)$ which also verify
 $t^{{N\over 2}({1\over q}-{1\over r})}\|u(t)\|_{L^r_\gamma},$  is bounded on $(0,T]$, where $r$ is given below, (we replace $[0,T]$ by $(0,T]$ if $q=\infty$  {  and  $u$ satisfies  $\lim_{t\to 0}\|u(t)-e^{t \Delta}u_0\|_{L^\infty_\gamma(\R^N)}=0$}).  Moreover,
 $u$ can be extended to a maximal interval $[0, T_{\max})$ such that
either $T_{\max}= \infty$ or  $T_{\max}< \infty$ and
$\displaystyle\lim_{t\to T_{\max}}\|u(t)\|_{L^q_\gamma} = \infty.$  Furthermore,
\begin{equation}
\label{VitesseInf3} \|u(t)\|_{L^q_\gamma}\geq C\left(T_{\max}-t\right)^{{N\over
2q}-{1\over \alpha}},\; \forall\; t\in [0, T_{\max}),
\end{equation}
where $C$ is a positive constant.
\end{itemize}
\end{To}
\begin{rem}
\label{ComparisonwithCIT}
{\rm Unlike in \cite{CIT}, here we dot not impose  $q\geq \alpha+1$.   Also, our strategy is different from that of \cite{CIT}. In fact, we use a method of \cite{W2,BTW}. Precisely, to prove the local well-posedness in $L^q_\gamma,$ we use an auxiliary space  $L^r_\gamma$
 for some $r$ as an auxiliary parameter, while in \cite{CIT}  the weight $\gamma=l/\alpha$ is replaced by a real number $s$  that is considered as an auxiliary parameter.}
\end{rem}
\begin{proof}[Proof of Theorem \ref{localNLHGC}]
(i)  Let us define the maps
\begin{equation*}
    K_{t,l}(u)=e^{t \Delta}\left(|\cdot|^l|u|^\alpha u\right),\; t>0.
\end{equation*}
 By the H\"older inequality and Proposition \ref{smoothingeffectorl-lebeg} with $\gamma={l\over \alpha}=\mu,\; \, q_1=q/(\alpha+1),\; q_2=q$, for each $t>0$ and if $q>\frac{N (\alpha+1)}{N-\gamma},$ $q\leq \infty,$ $K_{t,l}: L^q_\gamma\longrightarrow  L^q_\gamma$ is locally Lipschitz with
\begin{eqnarray*}
 \|K_{t,l}(u)-K_{t,l}(v)\|_{L^q_\gamma} &\leq & Ct^{-{N\over 2}({\alpha+1\over q}-{1\over q})}\left\||\cdot|^l(|u|^\alpha u-|v|^\alpha v)\right\|_{L^{q\over \alpha+1}_{\gamma}} \\&= & Ct^{-{N\alpha\over 2q}}\left\||\cdot|^{\alpha\gamma}(|u|^\alpha u-|v|^\alpha v)\right\|_{L^{q\over \alpha+1}_\gamma}
 \\ &\leq & C t^{-\frac{N \alpha}{2 q}} (\|u\|_{L^q_\gamma}^\alpha+\|v\|_{L^q_\gamma}^\alpha)\|u-v\|_{L^q_\gamma} \\
 &\leq & 2 CM^\alpha t^{-\frac{N \alpha}{2 q}} \|u-v\|_{L^q_\gamma},
\end{eqnarray*}
for $\|u\|_{L^q_\gamma}\leq M \mbox{ and } \;\|v\|_{L^q_\gamma}\leq M.$
We have also,  that $t^{-\frac{N \alpha}{2 q}}\in L^1_{loc}(0,\infty),$ since $q>q_c=\frac{N \alpha}{2}.$  Obviously $t\mapsto \|K_{t,l}(0)\|_\infty=0
\in L^1_{loc}(0,\infty),$ also $e^{s \Delta}K_{t,l}=K_{t+s,l}$ for $s,\;t>0.$ Then the proof follows by \cite[Theorem 1, p. 279]{W2}.

(ii) We begin with the observation that, since $q>{N\over N-\gamma},$ there exists $r > q$ satisfying
\begin{equation}
\label{subcrit1}
 {1\over q(\alpha+1)}<\frac{1}{r} < \frac{N - \gamma}{N(\alpha+1)}.
\end{equation}
We then observe that, since $q>q_c,$ we have
$$\frac{1}{q} - \frac{2}{N(\alpha+1)}<\frac{1}{q(\alpha+1)}.$$
Hence any $r>q$ satisfying \eqref{subcrit1} verifies
$$\frac{1}{q} - \frac{2}{N(\alpha+1)}<{1\over r}.$$
This last  inequality
 implies that
 $$
 \beta(\alpha + 1) < 1,
 $$
 where
 \begin{equation}
 \label{beta1}
  \beta = \frac{N}{2q} - \frac{N}{2r}.
    \end{equation}
      We choose $K>0,\; T>0,\; M>0$ such that
\begin{equation}
\label{equKMT5}
K+\mathcal{C}M^{\alpha+1}T^{1-\frac{N\alpha}{2q}}\leq M,
\end{equation}
where $\mathcal{C}$ is a positive constant given below.  We will show that  there exists a
unique solution $u$ of  (\ref{intHHNLH}) such that $u\in C\left([0,T], L^q_\gamma(\R^N)\right)$ and $u \in C\left((0,T], L^r_\gamma(\R^N)\right)$  with $$\|u\|=\max\left[\sup_{t\in[0, T]}\|u(t)\|_{L^q_\gamma},\;\sup_{t\in(0, T]}t^{\beta}\|u(t)\|_{L^r_\gamma}\right]\leq M.$$
The proof is based on a contraction mapping argument in the set
   $$ Y_{M,T}^{q,\gamma} = \{u \in C\left([0,T],L^q_\gamma(\R^N)\right)\cap C((0,T], L^r_\gamma) :
\|u\|\leq M\}.$$  Endowed with the metric $d(u,v)=\|u-v\|,$  $Y_{M,T}^{q,\gamma}$ is a nonempty complete metric space.
We note that for  $u_0\in L^q_\gamma$ we have
$$ \|{\rm e}^{t\Delta}u_0\|_{L^r_\gamma} \le Ct^{-{N\over 2}({1\over q}-{1\over r})}\|u_0\|_{L^q_\gamma}=Ct^{-\beta}\|u_0\|_{L^q_\gamma}.$$
The condition on initial data $\|u_0\|_{L^q_\gamma}\leq K$ implies that $t^\beta \|{\rm e}^{t\Delta}u_0\|_{L^r_\gamma}\leq K.$
We will show that
\begin{equation}
\label{iteratnHH}
\F_{u_0} u(t) = {\rm e}^{t\Delta}u_0 + \int_0^t {\rm
e}^{(t-\sigma)\Delta} \big[|\cdot|^l|u(\sigma)|^\alpha u(\sigma)\big]
d\sigma.
\end{equation}
 is a strict contraction on $Y_{M,T}^{q,\gamma}$.

Using Proposition \ref{smoothingeffectorl-lebeg}, that is the boundedness of the map $e^{t \Delta}: L^{q}_{\gamma}\rightarrow L^r_\gamma,$ for the first term and the boundedness of the map $e^{t \Delta}: L^{r\over \alpha+1}_{\gamma}\rightarrow L^r_\gamma,$  for the second term, we have
\begin{eqnarray*}
t^{\beta}\|\mathcal{F}_\varphi(u)(t)-\mathcal{F}_\psi(v)(t)\|_{L^r_\gamma}
  & \leq & t^{\beta}\|e^{t \Delta}(\varphi-\psi)\|_{L^r_\gamma}+
 \displaystyle t^{\beta} \int_0^t\|e^{(t-s) \Delta}[|.|^{l}\big(|u|^\alpha
 u(s)-|v|^\alpha v(s)\big)]\|_{L^{r}_\gamma} ds\\
 &&    \hspace{-2cm}  \leq \|\varphi-\psi\|_{L^q_\gamma}+
   C \displaystyle t^{\beta}\int_0^t (t-s)^{-\frac{N}{2}(\frac{\alpha+1}{r}-\frac{1}{r})} \left\| |.|^{\gamma\alpha}\big(|u|^ \alpha
 u(s)-|v|^\alpha v(s)\big)\right\|_{L^{{r\over\alpha+1}}_\gamma} ds \\
  &&    \hspace{-2cm}  \leq    \|\varphi-\psi\|_{L^q_\gamma}+\left(2 (\alpha+1) C  M^{\alpha} \displaystyle t^{\beta}\int_0^t \left(t-s\right)^{-\frac{N \alpha}{2 r}} s^{-\beta(\alpha+1)} ds\right) d(u,v)\\
  &&    \hspace{-2cm}  \leq
    \|\varphi-\psi\|_{L^q_\gamma}+ \left(2 (\alpha+1) C  M^{\alpha} t^{1-\frac{N \alpha}{2 q}} \displaystyle \int_0^1 \left(1-\sigma\right)^{-\frac{N \alpha}{2 r}} \sigma^{-\beta(\alpha+1)} d\sigma\right) d(u,v)\\
  &&    \hspace{-2cm}  \leq \|\varphi-\psi\|_{L^q_\gamma}+ C_1 M^\alpha T^{1-\frac{N \alpha}{2 q}} d(u,v),
 \end{eqnarray*}
 where $C_1=2 (\alpha+1) C  \displaystyle \int_0^1 \left(1-\sigma\right)^{-\frac{N \alpha}{2 r}} \sigma^{-\beta(\alpha+1)} d\sigma<\infty.$

Using \cite[Lemma 2.1]{CIT} that is the boundedness of the map $e^{t \Delta}: L^{q}_{\gamma}\rightarrow L^q_\gamma,$ for the first term and Proposition \ref{smoothingeffectorl-lebeg}, the boundedness of the map $e^{t \Delta}: L^{r\over \alpha+1}_{\gamma}\rightarrow L^q_\gamma,$  for the second term, we have
 \begin{eqnarray*}
    \|\mathcal{F}_\varphi(u)(t)-\mathcal{F}_\psi(v)(t)\|_{L^q_\gamma}
   &\leq &
      \|e^{t \Delta}(\varphi-\psi)\|_{L^q_\gamma}+
 \displaystyle \int_0^t\|e^{(t-s) \Delta}[|.|^{l}\big(|u|^\alpha
 u(s)-|v|^\alpha v(s)\big)]\|_{L^{q}_\gamma} ds\\
 &&    \hspace{-2cm}  \leq    \|\varphi-\psi\|_{L^q_\gamma}+
   C \displaystyle \int_0^t (t-s)^{-\frac{N}{2}(\frac{\alpha+1}{r}-\frac{1}{q})} \left\| |.|^{\gamma\alpha}\big(|u|^ \alpha
 u(s)-|v|^\alpha v(s)\big)\right\|_{L^{{r\over \alpha+1}}_\gamma} ds\\
  &&    \hspace{-2cm}  \leq   \|\varphi-\psi\|_{L^q_\gamma}+\left(2 (\alpha+1) C  M^{\alpha} \displaystyle \int_0^t \left(t-s\right)^{-\frac{N}{2}(\frac{\alpha+1}{r}-\frac{1}{q})} s^{-\beta(\alpha+1)} ds\right) d(u,v) \\
   &&    \hspace{-2cm}  \leq  \|\varphi-\psi\|_{q}+ \left(2 (\alpha+1) C  M^{\alpha} t^{1-\frac{N \alpha}{2 q}} \displaystyle \int_0^1 \left(1-\sigma\right)^{-\frac{N}{2}(\frac{\alpha+1}{r}-\frac{1}{q})} \sigma^{-\beta(\alpha+1)} d\sigma\right) d(u,v) \\
  &&    \hspace{-2cm}  \leq \|\varphi-\psi\|_{L^q_\gamma}+ C_2 M^\alpha T^{1-\frac{N \alpha}{2 q}} d(u,v).
 \end{eqnarray*}
 where $C_2=2 (\alpha+1) C  \displaystyle \int_0^1 \left(1-\sigma\right)^{-\frac{N}{2}(\frac{\alpha+1}{r}-\frac{1}{q})} \sigma^{-\beta(\alpha+1)} d\sigma<\infty.$

From the above estimates, it follows  that
 \begin{equation}\label{eq5tris}d(\mathcal F_\varphi(u), \mathcal F_\psi(v))\leq   \|\varphi-\psi\|_{L^q_\gamma}+ \mathcal{C} M^\alpha T^{1-\frac{N \alpha}{2 q}} d(u,v), \end{equation} where $\mathcal{C}=\max(C_1,C_2).$ The rest of the proof follows similarly as that of Theorem \ref{th3} and as in \cite{BTW}. This completes the proof.
\end{proof}

Theorem \ref{localNLHGC} allows us to obtain the following.
\begin{cor}[H\'enon parabolic equations]
\label{lowerLql} Let $N\geq 1,$ $\alpha>0,$ $0<l<N\alpha.$ If $\varphi \in L^q_\gamma(\R^N)$, where $$\gamma=l/\alpha<N,\; q>{N\alpha\over 2}\quad \mbox{ and } \quad {N\over N-\gamma} < q\leq \infty,$$ then the life-span of \eqref{intHHNLH} with initial data $\lambda\varphi$ satisfies
\begin{equation}
\label{bulambda3}
T_{\max}(\lambda\varphi)  { \geq}C(\lambda\|\varphi\|_{L^q_{{\gamma}}})^{-(\frac{1}{\alpha} - \frac{N}{2q})^{-1}},
\end{equation}
for all $\lambda > 0$, where $C = C(\alpha,q,l,N).$
\end{cor}
\begin{rem}{\rm $\,$
\begin{itemize}
\item[1)] We see that $l$ has no effect on the lower bound of the life span.   This is because blow up may not occurs at the origin nor at $|x|$ infinite.
\item[2)] Corollary \ref{lowerLql} is totally new for $q<\infty$.
\end{itemize}}
\end{rem}
\begin{proof}[Proof of Corollary \ref{lowerLql}]
The proof follows using \eqref{equKMT5} and is similar to that of Theorem  \ref{ThNLH}.
\end{proof}

In the case of initial data in $L^q(\RR^N)\cap L^q_\gamma(\RR^N)$ we have the following result which generalizes that of \cite{W} known for $q=\infty$.
\begin{To}[Local well-posedness in $L^q\cap L^q_\gamma$]
\label{localNLHGCb} Let $N\geq 1$ be an integer, $\alpha>0$ and $l>0$ be
such that
\begin{equation*}
 0<\gamma:={l\over \alpha}<N. \end{equation*} Let $q_c$ be given by
\eqref{qvalues13}. Then we have the following.
\begin{itemize}
\item[(i)] Equation \eqref{intHHNLH} is locally well-posed in $L^\infty(\mathbb R^N)\cap L^\infty_\gamma(\mathbb R^N).$ More precisely, given $u_0\in L^\infty(\mathbb R^N)\cap L^\infty_\gamma(\mathbb R^N)$, then there exist $T>0$ and a unique solution $u\in C\big((0, T];L^\infty(\mathbb R^N)\cap L^\infty_\gamma(\mathbb R^N))$ of
\eqref{intHHNLH}  {  and  $u$ satisfies $\lim_{t\to 0}\|u(t)-e^{t \Delta}u_0\|_{L^\infty\cap L^\infty_\gamma(\R^N)}=0$}. Moreover,
 $u$ can be extended to a maximal interval $(0, T_{\max})$ such that
either $T_{\max}= \infty$ or  $T_{\max}< \infty$ and
$\displaystyle\lim_{t\to T_{\max}}(\|u(t)\|_\infty+\|u(t)\|_{L^\infty_\gamma}) = \infty.$
\item[(ii)] If  $q$ is such that
$$q >\frac{N(\alpha+1)}{N-\gamma}, \quad q> q_c\quad \mbox{ and } \quad q<\infty,$$
then equation \eqref{intHHNLH} is locally well-posed in $L^q(\RR^N)\cap L^q_\gamma(\RR^N).$ More precisely, given $u_0\in L^q(\RR^N)\cap L^q_\gamma(\RR^N)$, then there exist $T>0$ and a unique solution $u\in C\big([0, T];L^q(\RR^N)\cap L^q_\gamma(\RR^N) \big)$ of
\eqref{intHHNLH}. Moreover,
 $u$ can be extended to a maximal interval $[0, T_{\max})$ such that
either $T_{\max}= \infty$ or  $T_{\max}< \infty$ and
$\displaystyle\lim_{t\to T_{\max}}(\|u(t)\|_q+\|u(t)\|_{L^q_\gamma}) = \infty.$
\item[(iii)] Assume that $q>q_c$ with ${N\over N-\gamma} < q\leq\infty$. It follows that equation \eqref{intHHNLH} is locally well-posed in $L^q(\RR^N)\cap L^q_\gamma(\RR^N)$ as in part (ii) except that uniqueness is guaranteed only among  functions $u\in C\big([0, T];L^q(\RR^N)\cap L^q_\gamma(\RR^N) \big)$ which also verify
 $t^{{N\over 2}({1\over q}-{1\over r})}\|u(t)\|_{L^r_\gamma},$ $t^{{N\over 2}({1\over q}-{1\over r})}\|u(t)\|_{r}$ are bounded on $(0,T]$, where $r$ is as above (we replace $[0,T]$ by $(0,T]$ if $q=\infty$ {  and  $u$ satisfies   $\lim_{t\to 0}\|u(t)-e^{t \Delta}u_0\|_{L^\infty\cap L^\infty_\gamma(\R^N)}=0$}).  Moreover,
 $u$ can be extended to a maximal interval $[0, T_{\max})$ such that
either $T_{\max}= \infty$ or  $T_{\max}< \infty$ and
$\displaystyle\lim_{t\to T_{\max}}(\|u(t)\|_q+\|u(t)\|_{L^q_\gamma}) = \infty.$  Furthermore,
\begin{equation}
\label{VitesseInf3bis} \|u(t)\|_{L^q\cap L^q_\gamma}\geq C\left(T_{\max}-t\right)^{{N\over
2q}-{1\over \alpha}},\; \forall\; t\in [0, T_{\max}),
\end{equation}
where $C$ is a positive constant.
\end{itemize}
\end{To}
\begin{proof}[Proof of Theorem \ref{localNLHGCb}]
We will just give the new elements of the proof.

(i)-(ii) By the H\"older inequality and Proposition \ref{smoothingeffectorl-lebeg} with $\gamma={l\over \alpha}=\mu,\; \, q_1=q/(\alpha+1),\; q_2=q$, for each $t>0$ and if $q>\frac{N (\alpha+1)}{N-\gamma},$ $q\leq \infty,$ $K_{t,l}: L^q\cap L^q_\gamma\longrightarrow L^q\cap L^q_\gamma$ is locally Lipschitz
and, since $q>\frac{N (\alpha+1)}{N-\gamma}>\alpha+1,$
\begin{eqnarray*}
 \|K_{t,l}(u)-K_{t,l}(v)\|_q &\leq & Ct^{-{N\over 2}({\alpha+1\over q}-{1\over q})}\left\||\cdot|^{l}(|u|^\alpha u-|v|^\alpha v)\right\|_{q/(\alpha+1)}\\ &\leq & C t^{-\frac{N \alpha}{2 q}} (\|u\|_{L^q_\gamma}^\alpha+\|v\|_{L^q_\gamma}^\alpha)\|u-v\|_{q} \\
 &\leq & 2 CM^\alpha t^{-\frac{N \alpha}{2 q}} \|u-v\|_{q},
\end{eqnarray*}
for $\max(\|u\|_q,\|u\|_{L^q_\gamma})\leq M \mbox{ and } \;\max(\|v\|_q,\|v\|_{L^q_\gamma})\leq M.$
We have also,  that $t^{-\frac{N \alpha}{2 q}}\in L^1_{loc}(0,\infty),$ since $q>q_c=\frac{N \alpha}{2}.$   Then the proof follows by \cite[Theorem 1, p. 279]{W2}.

(iii) We choose $r>q$ satisfying \eqref{subcrit1}. We consider $K>0,\; T>0,\; M>0$ such that
\begin{equation}
\label{equKMT5b}
K+\mathcal{C}M^{\alpha+1}T^{1-\frac{N\alpha}{2q}}\leq M,
\end{equation}
where $\mathcal{C}$ is a positive constant given below.  We will show that  there exists a
unique solution $u$ of  (\ref{intHHNLH}) such that $u\in C\left([0,T];L^q(\R^N)\cap L^q_\gamma(\R^N)\right)$ and $u \in C\left((0,T];L^{r}(\R^N)\cap L^r_\gamma(\R^N)\right)$  with $$\|u\|=\max\left[\sup_{t\in[0, T]}\|u(t)\|_q,\; \sup_{t\in[0, T]}\|u(t)\|_{L^q_\gamma},\;\sup_{t\in(0, T]}t^{\beta}\|u(t)\|_{L^r_\gamma},  \sup_{t\in(0, T]}t^{\beta}\|u(t)\|_{r}\right]\leq M.$$
The proof is based on a contraction mapping argument in the set
   $$ Y_{M,T}^{q,\gamma} = \{u \in C\left([0,T];L^q(\R^N)\cap L^q_\gamma(\R^N)\right)\cap C((0,T]; L^r\cap L^r_\gamma) :
\|u\|\leq M\}.$$  Endowed with the metric $d(u,v)=\|u-v\|,$  $Y_{M,T}^{q,\gamma}$ is a nonempty complete metric space.
We note that for   $u_0\in L^q,$
$$ \|{\rm e}^{t\Delta}u_0\|_{r} \le Ct^{-{N\over 2}({1\over q}-{1\over r})}\|u_0\|_{q}=Ct^{-\beta}\|u_0\|_{q}.$$
The condition on initial data $\max(\|u_0\|_{q},\|u_0\|_{L^q_\gamma})\leq K$ implies that $t^\beta \|{\rm e}^{t\Delta}u_0\|_{L^r_\gamma}\leq K,\; t^{\beta} \|{\rm e}^{t\Delta}u_0\|_{L^r}\leq K .$
We will show that  $\mathcal{F}_{u_0}$ defined in \eqref{iteratn} is a strict contraction on $Y_{M,T}^{q,\gamma}$.
Using Proposition \ref{smoothingeffectorl-lebeg}, that is $e^{t \Delta}: L^{q}\rightarrow L^r,$ for the first term and $e^{t \Delta}: L^{r\over \alpha+1}\rightarrow L^r,$  for the second term, we have
\begin{eqnarray*}
    t^{\beta} \|\mathcal{F}_\varphi(u)(t)-\mathcal{F}_\psi(v)(t)\|_{r}
  & \leq &
      t^{\beta}\|e^{t \Delta}(\varphi-\psi)\|_{r}+
 \displaystyle t^{\beta}\int_0^t\|e^{(t-s) \Delta}[|.|^{l}\big(|u|^\alpha
 u(s)-|v|^\alpha v(s)\big)]\|_{r} ds\\
&&    \hspace{-2cm}  \leq   \|\varphi-\psi\|_{q}+
   C \displaystyle  t^{\beta}\int_0^t (t-s)^{-\frac{N}{2}(\frac{\alpha+1}{r}-\frac{1}{r})} \left\| |.|^{\alpha\gamma}\big(|u|^ \alpha
 u(s)-|v|^\alpha v(s)\big)\right\|_{{r\over\alpha+1}} ds \\
  &&    \hspace{-2cm}  \leq     \|\varphi-\psi\|_{q}+\left(2 (\alpha+1) C  M^{\alpha} \displaystyle  t^{\beta}\int_0^t \left(t-s\right)^{-\frac{N \alpha}{2 r}} s^{-\beta(\alpha+1)} ds\right) d(u,v) \\
 &&    \hspace{-2cm}  \leq
    \|\varphi-\psi\|_{q}+ \left(2 (\alpha+1) C  M^{\alpha} t^{1-\frac{N \alpha}{2 q}} \displaystyle \int_0^1 \left(1-\sigma\right)^{-\frac{N \alpha}{2 r}} \sigma^{-\beta(\alpha+1)} d\sigma\right) d(u,v)\\
 &&    \hspace{-2cm}  \leq  \|\varphi-\psi\|_{L^q}+ C_3 M^\alpha T^{1-\frac{N \alpha}{2 q}} d(u,v),
 \end{eqnarray*}
where $C_3=2 (\alpha+1) C  \displaystyle \int_0^1 \left(1-\sigma\right)^{-\frac{N \alpha}{2 r}} \sigma^{-\beta(\alpha+1)} d\sigma<\infty.$

Using Proposition \ref{smoothingeffectorl-lebeg}, that is $e^{t \Delta}: L^{q}\rightarrow L^q,$ for the first term and $e^{t \Delta}: L^{r\over \alpha+1}\rightarrow L^q,$  for the second term, we have
\begin{eqnarray*}
\|\mathcal{F}_\varphi(u)(t)-\mathcal{F}_\psi(v)(t)\|_{q}
  & \leq &
      \|e^{t \Delta}(\varphi-\psi)\|_{q}+
 \displaystyle \int_0^t\|e^{(t-s) \Delta}[|.|^{l}\big(|u|^\alpha
 u(s)-|v|^\alpha v(s)\big)]\|_{q} ds\\
 &&    \hspace{-2cm}  \leq   \|\varphi-\psi\|_{q}+
   C \displaystyle  \int_0^t (t-s)^{-\frac{N}{2}(\frac{\alpha+1}{r}-\frac{1}{q})} \left\| |.|^{\alpha\gamma}\big(|u|^ \alpha
 u(s)-|v|^\alpha v(s)\big)\right\|_{{r\over\alpha+1}} ds\\
 &&    \hspace{-2cm}  \leq     \|\varphi-\psi\|_{q}+\left(2(\alpha+1) C  M^{\alpha} \displaystyle   \int_0^t \left(t-s\right)^{-\frac{N}{2}(\frac{\alpha+1}{r}-\frac{1}{q})} s^{-\beta(\alpha+1)} ds\right) d(u,v) \\
 &&    \hspace{-2cm}  \leq   \|\varphi-\psi\|_{q}+ \left(2 (\alpha+1) C  M^{\alpha} t^{1-\frac{N \alpha}{2 q}} \displaystyle \int_0^1 \left(1-\sigma\right)^{-\frac{N}{2}(\frac{\alpha+1}{r}-\frac{1}{q})} \sigma^{-\beta(\alpha+1)} d\sigma\right) d(u,v) \\
  &&    \hspace{-2cm}  \leq  \|\varphi-\psi\|_{L^q}+ C_4 M^\alpha T^{1-\frac{N \alpha}{2 q}} d(u,v).
 \end{eqnarray*}
where $C_4=2 (\alpha+1) C  \displaystyle \int_0^1 \left(1-\sigma\right)^{-\frac{N}{2}(\frac{\alpha+1}{r}-\frac{1}{q})} \sigma^{-\beta(\alpha+1)} d\sigma<\infty.$
From the above estimates, it follows  that
 \begin{equation}\label{eq5trisbis}d(\mathcal F_\varphi(u), \mathcal F_\psi(v))\leq   \|\varphi-\psi\|_{L^q\cap L^q_\gamma}+ \mathcal{C} M^\alpha T^{1-\frac{N \alpha}{2 q}} d(u,v), \end{equation} where $\mathcal{C}=\max(C_1,C_2,C_3,C_4).$ The rest of the proof follows similarly as above and as in \cite{BTW}.
\end{proof}
We have also the following result.
\begin{prop}
\label{th2} Let  $\alpha>0$ and  let $0<\gamma:=l/\alpha<N$. Assume the hypotheses of Theorem \ref{localNLHGCb}. Let $T_{\max}(\varphi,L^q\cap L^q_\gamma)$ denotes the existence time of the maximal solution of
\eqref{intHHNLH} with initial data $\varphi\in L^q\cap L^q_\gamma.$
Then we have the following.
\begin{itemize}
\item[(i)] If $\varphi \in L^q\cap L^q_\gamma,$ then for $t\in \left(0,T_{\max}(\varphi,q)\right),$ $u(t)\in L^\infty\cap L^\infty_\gamma.$
\item[(ii)] If $\varphi \in L^p\cap L^p_\gamma\cap L^q\cap L^q_\gamma$,  ${N\over N-\gamma}< q<p\leq\infty$  and $q>q_c.$ Then $T_{\max}(\varphi,L^p\cap L^p_\gamma)= T_{\max}(\varphi,L^q\cap L^q_\gamma).$
\end{itemize}
\end{prop}
\begin{proof}
(i) Let  $\varphi\in L^q(\RR^N),$ $q>q_c$ and $ q>{N\over N-\gamma}.$ Let $r$ and $\beta$ be as above and \eqref{beta1}. Let $p$ be such that $r<p\leq \infty.$ Hence $p>q,$
$$0\leq {1\over p}<{\gamma\over N}+{\alpha+1\over r}<1,\; {1\over p}<{\gamma\over N}+{1\over q}<1,$$ and  for $0<T<T_{\max}(\varphi,q),$ we have
\begin{eqnarray*}
    \|u(t)\|_{p}&\leq& \|e^{t \Delta}\varphi\|_{p}+  C \displaystyle\int_0^t (t-\sigma)^{-\frac{N }{2 }(\frac{\alpha+1}{r}-\frac{1}{p})} \|u(\sigma)\|^{\alpha}_{L^r_\gamma} \|u(\sigma)\|_{r}d\sigma\\
    &\leq&(4 \pi t)^{-\frac{N}{2}(\frac{1}{q}-\frac{1}{p})}\|\varphi\|_{q}+  C  t^{1-\frac{N }{2 }(\frac{\alpha+1}{r}-\frac{1}{p})-{\beta(\alpha+1)}} \displaystyle\sup_{s\in(0,T]}\left(s^{\beta\alpha}\|u(s)\|_{L^r_\gamma}^{\alpha}\right)\times
    \\&& \displaystyle\sup_{s\in(0,T]}\left(s^{\beta}\|u(s)\|_{r}\right)\int_0^1 (1-\sigma)^{-\frac{N }{2 }(\frac{\alpha+1}{r}-\frac{1}{p})} \sigma^{-\beta(\alpha+1)}d\sigma\\
    &\leq&(4 \pi t)^{-\frac{N}{2}(\frac{1}{q}-\frac{1}{p})}\|\varphi\|_{q}+   M^{\alpha+1}Ct^{1-\frac{N }{2 }(\frac{\alpha+1}{q}-\frac{1}{p})}\int_0^1 (1-\sigma)^{-\frac{N }{2 }(\frac{\alpha+1}{r}-\frac{1}{p})} \sigma^{-\beta(\alpha+1)}d\sigma.
   \end{eqnarray*}
   Also,
   \begin{eqnarray*}
    \|u(t)\|_{L^p_\gamma}&\leq& \|e^{t \Delta}\varphi\|_{L^p_\gamma}+  C \displaystyle\int_0^t (t-\sigma)^{-\frac{N }{2 }(\frac{\alpha+1}{r}-\frac{1}{p})} \|u(\sigma)\|^{\alpha+1}_{L^r_\gamma} d\sigma\\
    &\leq&Ct^{-\frac{N}{2}(\frac{1}{q}-\frac{1}{p})}\|\varphi\|_{L^q_\gamma}+  C  t^{1-\frac{N }{2 }(\frac{\alpha+1}{r}-\frac{1}{p})-{\beta(\alpha+1)}} \displaystyle\sup_{s\in(0,T]}\left(s^{\beta(\alpha+1)}\|u(s)\|_{L^r_\gamma}^{\alpha+1}\right)\times
    \\&& \displaystyle\int_0^1 (1-\sigma)^{-\frac{N }{2 }(\frac{\alpha+1}{r}-\frac{1}{p})} \sigma^{-\beta(\alpha+1)}d\sigma\\
    &\leq&C t^{-\frac{N}{2}(\frac{1}{q}-\frac{1}{p})}\|\varphi\|_{{L^q_\gamma}}+   M^{\alpha+1}Ct^{1-\frac{N }{2 }(\frac{\alpha+1}{q}-\frac{1}{p})}\int_0^1 (1-\sigma)^{-\frac{N }{2 }(\frac{\alpha+1}{r}-\frac{1}{p})} \sigma^{-\beta(\alpha+1)}d\sigma.
   \end{eqnarray*}
 Since $r>q>q_c,$  it follows that if $$\frac{\alpha+1}{r}-\frac{2}{N}<\frac{1}{p}<\frac{1}{r},$$
 then $u(t)$ is in $L^p\cap L^p_\gamma$ for all $t\in \big(0,T_{\max}(\varphi, q)\big).$
 The result for general $p>q$ follows by iteration. Hence $u(t)$ is in $L^\infty\cap L^\infty_\gamma,$ for $t\in \big(0,T_{\max}(\varphi, q)\big).$

 (ii) Follows as in Proposition \ref{C0}. This finishes the proof of the proposition.
\end{proof}

Theorem \ref{localNLHGCb} and inequality  \eqref{equKMT5b} allow us to obtain that under the same hypotheses of Corollary \ref{lowerLql} if $\varphi \in L^q(\R^N)\cap L^q_\gamma(\R^N)$ then the life-span of \eqref{HHNLH} with initial data $\lambda\varphi$ satisfies \begin{equation}
\label{bulambda3bis}
T_{\max}(\lambda\varphi)  { \geq}  C\left(\lambda\|\varphi\|_{L^q_{{\gamma}}\cap L^q}\right)^{-(\frac{1}{\alpha} - \frac{N}{2q})^{-1}},
\end{equation}
for all $\lambda > 0$, which gives a power of $\lambda$ not depending on $l,$ unlike the case $l<0.$
\section{The Hardy-H\'enon equations with decaying initial data}
\setcounter{equation}{0}
In this part of the  appendix, we investigate  lower bound estimates for life-span for the solutions of the equation \eqref{HHNLH} with initial data having some decay. As in Section \ref{NLHweight}, we  work in  $L^q_\gamma$  with $\gamma>0,\; \gamma>l/\alpha.$ This allows us to obtain a lower bound of the life span for initial data having more decay than $l/\alpha,$ if $l>0.$
 We consider the Duhamel formulation of \eqref{HHNLH}-\eqref{HHNLHinitial}, that is  the equation \eqref{intHHNLH} and suppose that
\begin{equation}
\label{conditionsParametersHH}
N\geq 1,\; \alpha>0,\; -\min(2,N)< l<N\alpha.
\end{equation} Let  $\gamma$ be such that
\begin{equation}
\label{intHHNLHdecay}
0< \gamma<N,\; {l\over \alpha}<\gamma<{2+l\over \alpha}
\end{equation} and
$q$ satisfying
\begin{equation}
\label{intHHNLHdecay1}
{N\over N-\gamma}<q\leq \infty, \; q>{N\alpha\over 2+l-\gamma\alpha}=:q_c(\gamma,l).
\end{equation}
$q_c(\gamma,l)$ is the critical exponent of \eqref{HHNLH} for initial data in $L^q_\gamma.$
The condition \eqref{intHHNLHdecay1} can be reformulated as follows:
$${\gamma\over N}\leq {1\over q}+{\gamma\over N}<1,\; {N\alpha\over 2q}+{\gamma\alpha\over 2}-{l\over 2}<1.$$
Let $0<\nu<\gamma$ be such that
$${\gamma+l\over \alpha+1}<\nu<{N+l\over \alpha+1}.$$ Hence, using \eqref{conditionsParametersHH} and \eqref{intHHNLHdecay}, we have $${l\over \alpha}< \nu<\gamma,\; 0<\nu<\nu(\alpha+1)-l<N,\; 0<\gamma<\nu(\alpha+1)-l<N.$$
Let now $r>q$ be such that
$${1\over q(\alpha+1)}-{\nu(\alpha+1)-l-\gamma\over N(\alpha+1)}< {1\over r}<{N-\nu(\alpha+1)+l\over N(\alpha+1)}.$$
This is possible by \eqref{intHHNLHdecay1}. Hence, we have
$${1\over r}<{\alpha+1\over r}+{\nu(\alpha+1)-l-\nu\over N}<{\alpha+1\over r}+{\nu(\alpha+1)-l\over N}<1,$$
$${1\over q}<{\alpha+1\over r}+{\nu(\alpha+1)-l-\gamma\over N}<{\alpha+1\over r}+{\nu(\alpha+1)-l\over N}<1.$$
That is, by \cite[Lemma 2.1]{CIT}  $e^{t \Delta}: L^{q}_\gamma\rightarrow L^q_\gamma,$  is bounded and we may apply  Proposition \ref{smoothingeffectorl-lebeg}, so that the maps  $e^{t \Delta}: L^{q}_\gamma\rightarrow L^r_\nu,$ $e^{t \Delta}: L^{r/(\alpha+1)}_{\nu(\alpha+1)-l} \rightarrow L^{q}_\gamma$ and $e^{t \Delta}: L^{r/(\alpha+1)}_{\nu(\alpha+1)-l}\rightarrow L^{r}_\nu$ are bounded.

Let us introduce
$$\tilde{\beta_l}={N\over 2}\left({1\over q}-{1\over r}\right)+{\gamma-\nu\over 2}.$$
Hence
$$\tilde{\beta_l}> {\gamma-\nu\over 2}> 0.$$ We have,
\begin{eqnarray*}
\tilde{\beta_l}(\alpha+1)&=&{N\over 2}\left({\alpha+1\over q}-{\alpha+1\over r}\right)+(\alpha+1){\gamma-\nu\over 2}\\
&\leq & {N\over 2}\left({\alpha+1\over q}-{1\over q}+{\nu(\alpha+1)-l-\gamma\over N}\right)+(\alpha+1){\gamma-\nu\over 2}
\\
&= & {N\alpha\over 2q}+{\alpha\gamma\over 2}-{l\over 2}<1.
\end{eqnarray*}
We also have
$$\frac{N\alpha}{2r}+{\nu\alpha\over 2}-{l\over 2}<\frac{N\alpha}{2q}+{\gamma\alpha\over 2}-{l\over 2}<1,$$
$$\frac{N}{2}\left({\alpha+1\over r}-{1\over q}\right)+{\nu(\alpha+1)-l-\gamma\over 2}<\frac{N\alpha}{2r}+{\nu\alpha\over 2}-{l\over 2}<1.$$
These last three  estimates are crucial to the local existence argument below. We note that if $l>0,$ we may take
$\nu(\alpha+1)-l=\gamma,\; r=(\alpha+1)q.$
With the above choice of the parameters, we can show the following local well-posedness result.
\begin{To}
\label{localNLHGCD} Let $N\geq 1$ be an integer, $\alpha>0,\; -\min(2,N)<l$ and
\begin{equation}
\label{e03bc} {l\over N}<\alpha. \end{equation} Let $\gamma$ be satisfying  \eqref{intHHNLHdecay} and $q_c(\gamma,l)$ be given by
\eqref{intHHNLHdecay1}. Then we have the following.
\begin{itemize}
\item[(i)] If $\gamma(\alpha+1)<N+l$ and $q$ is such that
$$q >\frac{N(\alpha+1)}{N+l-\gamma(\alpha+1)}, \quad q> q_c(\gamma,l)\quad \mbox{ and } \quad q\leq \infty,$$
then equation \eqref{intHHNLH} is locally well-posed in $L^q_\gamma(\RR^N).$ More precisely, given $u_0\in  L^q_\gamma(\RR^N)$, then there exist $T>0$ and a unique solution $u\in C\big([0, T]; L^q_\gamma(\RR^N) \big)$ of
\eqref{intHHNLH} (we replace $[0,T]$ by $(0,T]$ if $q=\infty$ {  and  $u$ satisfies $\lim_{t\to 0}\|u(t)-e^{t \Delta}u_0\|_{L^\infty_\gamma(\R^N)}=0$}). Moreover,
 $u$ can be extended to a maximal interval $[0, T_{\max})$ such that
either $T_{\max}= \infty$ or  $T_{\max}< \infty$ and
$\displaystyle\lim_{t\to T_{\max}}\|u(t)\|_{L^q_\gamma} = \infty.$
\item[(ii)] Assume that $q>q_c(\gamma,l)$ with ${N\over N-\gamma} < q\leq\infty$. It follows that equation \eqref{intHHNLH} is locally well-posed in $ L^q_\gamma(\RR^N)$ as in part (i) except that uniqueness is guaranteed only among  functions $u\in C\big([0, T]; L^q_\gamma(\RR^N) \big)$ which also verify
 $t^{\tilde{\beta_l}}\|u(t)\|_{L^r_\nu},$ is bounded on $(0,T]$, where $r$ and $\nu$ are as above (we replace $[0,T]$ by $(0,T]$ if $q=\infty$ {  and  $u$ satisfies $\lim_{t\to 0}\|u(t)-e^{t \Delta}u_0\|_{L^\infty_\gamma(\R^N)}=0$}).  Moreover,
 $u$ can be extended to a maximal interval $[0, T_{\max})$ such that
either $T_{\max}= \infty$ or  $T_{\max}< \infty$ and
$\displaystyle\lim_{t\to T_{\max}}\|u(t)\|_{L^q_\gamma} = \infty.$  Furthermore,
\begin{equation}
\label{VitesseInf2} \|u(t)\|_{L^q_\gamma}\geq C\left(T_{\max}-t\right)^{{N\over
2q}+{\gamma\over 2}-{2+l\over 2\alpha}},\; \forall\; t\in [0, T_{\max}),
\end{equation}
where $C$ is a positive constant.
\end{itemize}
\end{To}
\begin{proof}  (i)   Using the inequality \eqref{ineqwithHolder}, and Proposition \ref{smoothingeffectorl-lebeg} that is $e^{t \Delta}: L^{q\over \alpha+1}_{(\alpha+1)\gamma-l}\rightarrow L^q_\gamma$ is bounded, for each $t>0$ we have that  $K_{t,l}: L^q_\gamma\longrightarrow L^q_\gamma$ is locally Lipschitz with
\begin{eqnarray*}
 \|K_{t,l}(u)-K_{t,l}(v)\|_{L^q_\gamma} &\leq & Ct^{-{N\over 2}({\alpha+1\over q}-{1\over q})-{\alpha\gamma-l\over 2}}\left\||u|^\alpha u-|v|^\alpha v\right\|_{L^{q\over \alpha+1}_{(\alpha+1)\gamma}}\\  &\leq & C t^{-\frac{N \alpha}{2 q}-{\alpha\gamma\over 2}+{l\over 2}} (\|u\|_{L^q_\gamma}^\alpha+\|v\|_{L^q_\gamma}^\alpha)\|u-v\|_{L^q_\gamma} \\
 &\leq & 2 CM^\alpha t^{-\frac{N \alpha}{2 q}-{\alpha\gamma\over 2}+{l\over 2}} \|u-v\|_{L^q_\gamma},
\end{eqnarray*}
for $\|u\|_{L^q_\gamma}\leq M \mbox{ and } \;\|v\|_{L^q_\gamma}\leq M.$ The rest of the proof  is similar to that of Theorems  \ref{th3} and  \ref{localNLHGC}.

(ii) For  $u_0\in L^q_\gamma$ we have
$ \|{\rm e}^{t\Delta}u_0\|_{L^r_\nu} \le Ct^{-{N\over 2}({1\over q}-{1\over r})-{\gamma-\nu\over 2}}\|u_0\|_{L^q_\gamma}=Ct^{-\tilde{\beta_l}}\|u_0\|_{L^q_\gamma}.$ We choose $K>0,\; T>0,\; M>0$ such that
\begin{equation}
\label{equKMTlD}
K+CM^{\alpha+1}T^{1-\frac{N\alpha}{2q}-{\gamma\alpha\over 2}+{l\over 2}}\leq M,
\end{equation}
where $C$ is a positive constant.  We will show that  there exists a
unique solution $u$ of  \eqref{intHHNLH} such that $u\in C\left([0,T]; L^q_\gamma(\R^N)\right)$ and $u \in C\left((0,T]; L^r_\nu(\R^N)\right)$  with $$\|u\|=\max\left[ \sup_{t\in[0, T]}\|u(t)\|_{L^q_\gamma},\;\sup_{t\in(0, T]}t^{{\tilde{\beta_l}}}\|u(t)\|_{L^r_\nu}\right]\leq M.$$

The proof is based on a contraction mapping argument in the set
   $$ Y_{M,T}^{q,\gamma} = \{u \in C\left([0,T];L^q_\gamma(\R^N)\right)\cap C((0,T]; L^r_\nu) :
\|u\|\leq M\}.$$  Endowed with the metric $d(u,v)=\|u-v\|,$  $Y_{M,T}^{q,\gamma}$ is a nonempty complete metric space. We consider $u_0$ such that  $\|u_0\|_{L^q_\gamma}\leq K$ and we estimate as follows:
\begin{eqnarray*}
t^{\tilde{\beta_l}} \|\F_{u_0}u(t)\|_{L^r_\nu} &\le& t^{\tilde{\beta_l}} \|{\rm e}^{t\Delta}u_0\|_{L^r_\nu} + t^{\tilde{\beta_l}} \int_0^t \|{\rm
e}^{(t-\sigma)\Delta} \big[|\cdot|^l|u(\sigma)|^\alpha u(\sigma)\big]\|_{L^r_\nu} d\sigma\\
&\le& K +C
 t^{\tilde{\beta_l}} \int_0^t (t-\sigma)^{-\frac{N\alpha}{2r}-{(\nu(\alpha+1)-l)-\nu\over 2}} \||\cdot|^{\nu(\alpha+1)}|u(\sigma)|^\alpha u(\sigma)\|_{r/(\alpha+1)} d\sigma\\
 &=& K +C
t^{\tilde{\beta_l}} \int_0^t (t-\sigma)^{-\frac{N\alpha}{2r}-{\nu\alpha\over 2}+{l\over 2}} \|u(\sigma)\|_{L^r_\nu}^{\alpha+1} d\sigma\\
  &\le& K +
CM^{\alpha+1}t^{\tilde{\beta_l}} \int_0^t (t-\sigma)^{-\frac{N\alpha}{2r}-{\nu\alpha\over 2}+{l\over 2}} \sigma^{-{\tilde{\beta_l}}(\alpha+1)} d\sigma \\
  &\le& K +
CM^{\alpha+1}t^{1-\frac{N\alpha}{2q}-{\gamma\alpha\over 2}+{l\over 2}} \int_0^1(1-\sigma)^{-\frac{N\alpha}{2r}-{\nu\alpha\over 2}+{l\over 2}} \sigma^{-{\tilde{\beta_l}}(\alpha+1)} d\sigma
\\
  &\le& K +
CM^{\alpha+1}T^{1-\frac{N\alpha}{2q}-{\gamma\alpha\over 2}+{l\over 2}}\int_0^1(1-\sigma)^{-\frac{N\alpha}{2r}-{\nu\alpha\over 2}+{l\over 2}} \sigma^{-{\tilde{\beta_l}}(\alpha+1)} d\sigma,
\end{eqnarray*}
and similarly for the contraction. The other estimates can be handled similarly as above, see also \cite{BTW}. So we omit the details. This completes the proof of Theorem \ref{localNLHGCD}.
\end{proof}
\begin{rem}
{\rm $\,$
\begin{itemize}
\item[1)] We can take $\gamma=\max(0,{l\over \alpha})$ in  Theorem \ref{localNLHGCD} as well as $l=0$,  it is then a generalization of Theorems \ref{th3} and \ref{localNLHGC}.
\item[2)] See \cite[Theorem 1.13]{CIT} for related results.  The range of the values of $q$ in  (ii) are larger than in \cite{CIT},
while (i) is essentially contained in \cite{CIT} which we give for completeness. Also the methods are different. In fact, we work in an auxiliary space $L^r_\nu,$ for some $r$ and $\nu$ while in \cite{CIT} some auxiliary spaces $L^q_{\tilde\nu}$ for some ${\tilde\nu}$ but $q$ is fixed are considered.
\item[3)] If $q=\infty$ we may replace $L^\infty_\gamma$ by the space obtained by the closure, with respect to the
$L^\infty_\gamma$-topology, of $\mathcal{D}(\R^N),$ the space of  compactly supported $C^\infty(\R^N)$ functions.
For initial data in this sub-space  of $L^\infty_\gamma$ the result holds on $[0,T]$ instead of  $(0,T].$
    \item[4)] Using argument of \cite{BenSlimene}, we can show that uniqueness in the part (ii) of Theorem \ref{localNLHGCD}   holds  in
$u\in C\big([0, T]; L^q_\gamma(\RR^N) \big)\cap C\big((0, T]; L^r_\nu(\RR^N) \big).$
    \end{itemize}}
\end{rem}
Theorem  \ref{localNLHGCD} gives the following.
\begin{cor}[H\'enon parabolic equations with decaying initial data]
\label{lowerLql2} Let $N\geq 1$ be an integer, $\alpha>0$ and $-\min(2,N)<l<N\alpha.$ If $\varphi \in  L^q_\gamma(\R^N)$,  where
$$0< \gamma<N,\; {l\over \alpha}<\gamma<{2+l\over \alpha},$$
$${\gamma\over N}\leq {1\over q}+{\gamma\over N}<1,\; {N\alpha\over 2q}+{\gamma\alpha\over 2}-{l\over 2}<1.$$
 Then the life-span of \eqref{intHHNLH} with initial data $\lambda\varphi$ satisfies
\begin{equation}
\label{bulambda32}
T_{\max}(\lambda\varphi)  { \geq}  C\left(\lambda\|\varphi\|_{L^q_{{\gamma}}}\right)^{-\left({2+l\over 2\alpha}-{N\over 2q}-{\gamma\over 2}\right)^{-1}},
\end{equation}
for all $\lambda > 0$, where $C = C(\alpha,q,l,\gamma,N)>0$ is a constant.
\end{cor}
\begin{rem}{\rm $\,$
\begin{itemize}
\item[1)] Corollary \ref{lowerLql2} answerers a problem left open in \cite{P}. In fact, when $l>0$ only exponentially decaying initial data are considered in \cite{P}.
\item[2)] Similar results, using scaling argument,  seems to be proved in \cite{ZW,ZW2} for related equations, but only for small $\lambda,\; q=\infty$ and positive initial data.
\item[3)] If $\varphi\in L^q\cap L^q_\gamma,$ or $\varphi\in L^q_{l_+/\alpha}\cap L^q_\gamma,$ where $l_+=\max(l,0)$ then \eqref{bulambda32} is better than \eqref{lowerHH} and \eqref{bulambda3} for $0<\lambda<1.$

    \end{itemize}}
\end{rem}
\begin{proof}[Proof of Corollary \ref{lowerLql2}]
The proof follows using \eqref{equKMTlD} and is similar to that of Theorem \ref{ThNLH}, so we omit the details.
\end{proof}
We complement Corollary \ref{lowerLql2} by  the following upper bound estimates.
 \begin{prop}[Upper bounds of life-span for Hardy-H\'enon equations]
 \label{upperHH}
 Let $N\geq 1$ be an integer $\alpha>0$ and $-\min(2,N)<l.$ Assume that
\begin{equation}
\label{borblowuphh}
{l\over N}<\alpha<{2+l\over N}.
\end{equation}
Let  $\omega \in L^\infty(\R^N)$
be homogeneous of degree $0$, $\omega \ge 0$ and $\omega \not\equiv 0,$  $\tilde\varphi$ be given by \eqref{phitilde} and $\tilde{\tilde\varphi}$ be given by \eqref{phitildetilde}.
Let $0<\gamma<N$ be such that $${l\over \alpha}< \gamma,$$
 and $\varphi\in L^q_\gamma(\R^N),$ where ${N\over N-\gamma}<q\leq \infty.$
Then we have the following.
\begin{itemize}
\item[(i)] If $\varphi\geq \tilde\varphi$  then $T_{\rm{max}}(\lambda \varphi)\leq C\lambda^{-\left({2+l\over 2 \alpha}-{\gamma\over 2}\right)^{-1}},\; \lambda>1.$
\item[(ii)]If $\varphi\geq \tilde{\tilde\varphi}$  then $T_{\rm{max}}(\lambda \varphi)\leq C\lambda^{-\left({2+l\over 2\alpha}-{\gamma\over 2}\right)^{-1}},\; 0<\lambda<1.$
\end{itemize}
 \end{prop}

 To prove Proposition \ref{upperHH}, we use a scaling argument.  We recall the definition  of the dilation operators $D_\mu \varphi=\varphi(\mu \,\cdot),\; \mu>0.$
 It is clear that if $u$ is a solution of the equation \eqref{HHNLH} then for any $\mu>0,$ $u_\mu$ is also a solution of \eqref{HHNLH}, where  $u_{\mu}(t,x)=\mu^{2+l\over \alpha}u(\mu^2t,\mu x).$ Hence, $\sigma$ in \eqref{eqinvt} is given by
 $$\sigma={2+l\over \alpha}.$$
So that, for $\lambda= \mu^{\gamma-\frac{2+l}{\alpha}}$, \eqref{lifespanscale} reads
\begin{equation}
\label{scalingforHH}
\lambda^{-[({2+l\over 2\alpha}-{\gamma\over 2})^{-1}]}T_{\max}(\lambda \varphi)=T_{\max}(\lambda\mu^{\frac{2+l}{\alpha}}D_{\mu}\varphi)=T_{\max}(\mu^{\gamma}D_{\mu}\varphi).
\end{equation}

 Let  $0<\gamma<(2+l)/\alpha.$   Let $\varphi$ be a nonnegative function, satisfying
$\mu^\gamma D_\mu \varphi\leq \varphi,\;\mbox{for some}\; \mu>0.$ Then, since $\lambda=\mu^{\gamma-{2+l\over \alpha}}$,  we have
 by comparison argument (see \cite[Theorem 2.4, p. 564]{W}) and \eqref{scalingforHH} that
 $$T_{\rm{max}}(\lambda\varphi)\geq \lambda^{-\left({2+l\over 2\alpha}-{\gamma\over 2}\right)^{-1}}T_{\rm{max}}(\varphi).$$

 Similarly, if $ \mu^\gamma D_\mu \varphi\geq \varphi,\;\mbox{for some}\; \mu>0,$
and $T_{\rm{max}}(\varphi)<\infty,$  we have that
$$T_{\rm{max}}(\lambda \varphi)\leq \lambda^{-\left({2+l\over 2\alpha}-{\gamma\over 2}\right)^{-1}}T_{\rm{max}}(\varphi).
$$
\begin{proof}[Proof of Proposition \ref{upperHH}] The condition \eqref{borblowuphh} implies that $T_{\rm{max}}(\lambda\varphi)<\infty$ as well as $T_{\rm{max}}(\lambda\tilde\varphi)<\infty,$ and $T_{\rm{max}}(\lambda\tilde{\tilde\varphi})<\infty,$ for any $\lambda>0.$ See \cite{Qi}.

(i) By comparison argument it suffices to give the proof for $T_{\rm{max}}(\lambda \tilde\varphi).$ We have that
$$ \mu^\gamma D_\mu \tilde\varphi \geq\tilde\varphi,\; \mu<1.$$ Since $\gamma<(2+l)/\alpha$ and $\lambda=\mu^{\gamma-{2+l\over \alpha}}$ then $\mu<1$ is equivalent to $\lambda>1.$ By the above calculations,
$$T_{\rm{max}}(\lambda \varphi_1)\leq C\lambda^{-\left({2+l\over 2\alpha}-{\gamma\over 2}\right)^{-1}},\; \lambda>1.$$

(ii) By comparison argument it suffices to give the proof for $T_{\rm{max}}(\lambda \tilde{\tilde\varphi}).$ We have that
$$ \mu^\gamma D_\mu \tilde{\tilde\varphi} \geq\tilde{\tilde\varphi},\; \mu>1.$$ Since $\gamma<(2+l)/\alpha$ and $\lambda=\mu^{\gamma-{2+l\over \alpha}}$ then $\mu>1$ is equivalent to $\lambda<1.$ Then by the above calculations,
$$T_{\rm{max}}(\lambda \tilde{\tilde\varphi})\leq C\lambda^{-\left({2+l\over 2\alpha}-{\gamma\over 2}\right)^{-1}},\; \lambda<1.$$
This completes the proof of the proposition.
\end{proof}
\begin{rem}{\rm $\,$ We may take $q=\infty$ in Proposition \ref{upperHH}. In particular, combining Corollary \ref{lowerLql2} and   Proposition \ref{borblowuphh}, we have
$T_{\rm{max}}(\lambda \tilde\varphi)\sim\lambda^{-\left({2+l\over 2 \alpha}-{\gamma\over 2}\right)^{-1}},$ as $\lambda\to \infty,$ and $T_{\rm{max}}(\lambda \tilde{\tilde\varphi})\sim\lambda^{-\left({2+l\over 2 \alpha}-{\gamma\over 2}\right)^{-1}},$ as $\lambda\to 0.$ This shows that, for large initial data  the life-span increases as the power $l$  increases, while,  for small initial data the life-span decreases as the power $l$ increases.
}
\end{rem}
\end{appendix}

\end{document}